\documentclass[11pt, twoside,leqno]{article}

\usepackage{amsmath}
\usepackage{amssymb}
\usepackage{titletoc}
\usepackage{mathrsfs}
\usepackage{amsthm}

\usepackage{indentfirst}
\usepackage{color}
\usepackage{txfonts}

\allowdisplaybreaks

\def\bint{{\ifinner\rlap{\bf\kern.30em--}
\int\else\rlap{\bf\kern.35em--}\int\fi}\ignorespaces}

\def\sbint{{\ifinner\rlap{\bf\kern.32em--}
\hspace{0.078cm}\int\else\rlap{\bf\kern.45em--}\int\fi}\ignorespaces}

\def\vpz{\vphantom}

\def\rr{\mathbb{R}}
\def\rn{\mathbb{R}^n}

\def\cc{\mathbb{C}}
\def\nn{\mathbb{N}}
\def\zz{\mathbb{Z}}

\def\bz{\beta}
\def\gz{{\gamma}}

\def\ls{\lesssim}

\def\fz{\infty}
\def\az{\alpha}

\def\ca{{\mathcal A}}

\def\cg{{\mathcal G}}

\def\cn{{\mathcal N}}

\def\cx{{\mathcal X}}
\def\cy{{\mathcal Y}}

\def\hs{\hspace{0.35cm}}

\def\lp{{L^p(\mathcal{X})}}

\def\icgg{\mathcal{G}_0^\eta(\beta,\gamma)}
\def\cgg{\mathring{\mathcal{G}}_0^\eta(\beta,\gamma)}

\def\hb{\dot{B}^s_{p,q}(\mathcal{X})}
\def\hf{\dot{F}^s_{p,q}(\mathcal{X})}
\def\ihb{B^s_{p,q}(\mathcal{X})}
\def\ihf{F^s_{p,q}(\mathcal{X})}
\def\hfi{\dot{F}^s_{\infty,q}(\mathcal{X})}

\def\ihfi{F^s_{\infty,q}(\mathcal{X})}

\def\lips{\dot{L}(s,p,q;\cx)}
\def\ilips{L(s,p,q;\cx)}
\def\lipsb{\dot{L}_b(s,p,q;\cx)}
\def\ilipsb{L_b(s,p,q;\cx)}
\def\lipst{\dot{L}_t(s,p,q;\cx)}
\def\ilipst{L_t(s,p,q;\cx)}
\def\lipstu{\dot{L}_t(s,p,q,u;\cx)}
\def\ilipstu{L_t(s,p,q,u;\cx)}

\def\r{\right}
\def\lf{\left}

\def\noz{{\nonumber}}

\def\r{\right}
\def\lf{\left}

\def\supp{{\mathop\mathrm{\,supp\,}}}

\def\diam{{\mathop\mathrm{\,diam\,}}}

\def\loc{{\mathop\mathrm{\,loc\,}}}

\DeclareMathOperator*{\esssup}{ess\ sup}

\def\eqref#1{(\ref{#1})}

\def\func#1{\mathop{\mathrm{#1}}}
\def\diam{\func{diam}}
\def\supp{\func{supp}}

\def\ya{y_\alpha^{k,m}}
\def\qa{Q_\alpha^{k,m}}
\def\qo{Q_\alpha^{0,m}}

\def\red{\color{red}}

\newtheorem{theorem}{Theorem}[section]
\newtheorem{lemma}[theorem]{Lemma}

\newtheorem{proposition}[theorem]{Proposition}

\theoremstyle{definition}
\newtheorem{remark}[theorem]{Remark}
\newtheorem{definition}[theorem]{Definition}

\numberwithin{equation}{section}

\begin{document}

\title{\bf\Large  Difference Characterization
of Besov and Triebel--Lizorkin
Spaces on Spaces of Homogeneous Type
\footnotetext{\hspace{-0.35cm} 2020 {\it Mathematics Subject Classification}. Primary 46E36;
Secondary 46E35, 42B25, 42B20, 42B35, 30L99.\endgraf
{\it Key words and phrases.} space of
homogeneous type, Calder\'on reproducing formula,
Besov space, Triebel--Lizorkin space, Lipschitz-type space, difference.\endgraf
This project is partially supported by the National Key Research and Development Program of China
(Grant No.\ 2020YFA0712900) and the National
Natural Science Foundation of China (Grant Nos.\
11971058, 12071197 and 11871100).}}
\date{}
\author{Fan Wang, Ziyi He, Dachun Yang\footnote{Corresponding author,
E-mail: \texttt{dcyang@bnu.edu.cn}/{\red January 9, 2021}/Final version.}\ \ and Wen Yuan}
\maketitle

\vspace{-0.8cm}

\begin{center}
\begin{minipage}{13cm}
{\small {\bf Abstract}\quad In this article, the
authors introduce
the spaces of Lipschitz type on spaces of
homogeneous type
in the sense of Coifman and Weiss,
and discuss their  relations  with Besov and
Triebel--Lizorkin spaces.
As an application, the authors establish the
difference characterization of Besov
and Triebel--Lizorkin
spaces on spaces of homogeneous type. A
major novelty of this article is
that all results presented in this article get rid
of the dependence on
the reverse doubling assumption of the
considered measure of
the underlying space ${\mathcal X}$ via using the
geometrical property of ${\mathcal X}$
expressed by its dyadic reference points,
dyadic cubes, and the  (local) lower bound.
Moreover, some results when $p\le 1$ but near to $1$ are
new even when ${\mathcal X}$ is an RD-space.
}
\end{minipage}
\end{center}

\vspace{0.2cm}



\section{Introduction}\label{intro}

In 1970s, Coifman and Weiss \cite{cw71,cw77}
introduced the space of homogeneous type
which is a natural generalization of $\rn$ and
has provided a suitable setting
for the study of function spaces and the
boundedness
of Calder\'on--Zygmund operators. Let us
recall the following notion of quasi-metric
spaces,
which is the basis of the notion of spaces of
homogeneous type.

\begin{definition}\label{10.21.1}
Let $\cx$ be a non-empty set and $d$ a
\emph{quasi-metric} on $\cx$,
namely, a non-negative function on $
\cx\times\cx$ satisfying that, for any $x,\ y,\ z
\in \cx$,
\begin{enumerate}
\item[{\rm(i)}] $d(x,y)=0$ if and only if $x=y$;
\item[{\rm(ii)}] $d(x,y)=d(y,x)$;
\item[{\rm(iii)}] there exists a constant $A_0 \in
[1, \infty)$, independent of $x$, $y$,
and $z$, such that
$$d(x,z)\leq A_0[d(x,y)+d(y,z)].$$
\end{enumerate}
Then $(\cx, d)$ is called a \emph{quasi-metric
space}.
\end{definition}

If $A_0=1$, then $d$ becomes a \emph{metric}
and $(\cx,d)$ is called a \emph{metric space}.
The \emph{ball} $B$ of $\cx$, centered at
$x_0 \in \cx$ with radius
$r\in(0, \infty)$, is defined by setting
$$
B:=\{x\in\cx: d(x,x_0)<r\}=:B(x_0,r).
$$
We denote by $\tau B$ the ball with the same
center as that of $B$ but
of radius $\tau$ times that of $B$ for any ball
$B\subset\cx$ and $\tau\in(0,\infty)$.

Now, let us recall the notion of spaces of
homogeneous type in the sense of Coifman and
Weiss (\cite{cw71,cw77}).
\begin{definition}\label{10.21.2}
Let $(\cx,d)$ be a quasi-metric space and
$\mu$ a measure on $\cx$.
The triple $(\cx,d,\mu)$ is called
a \emph{space of homogeneous type} if $\mu$
has the following doubling property: there
exists a
positive constant $C_{(\mu)}\in[1,\infty)$ such
that, for any ball $B \subset \cx$,
$$
\mu(2B)\leq C_{(\mu)}\mu(B).
$$
\end{definition}

Note that the above doubling condition implies
that, for any ball $B$ and any
$\lambda\in[1,\infty)$,
\begin{equation}\label{eq-doub}
\mu(\lambda B)\leq C_{(\mu)}
\lambda^\omega\mu(B),
\end{equation}
where $\omega:= \log_2 C_{(\mu)}$ is called
the \emph{upper dimension} of $\cx$.
If $A_0 = 1$, then
$(\cx,d,\mu)$ is called a \emph{metric
measure space of homogeneous type} or,
simply, a \emph{doubling metric measure space}.

In recent years, function spaces and their
applications on spaces of homogeneous type,
with some additional assumptions,
have been extensively investigated; see, for
instance, \cite{am15,
dh09,hy02,y02,yy19} for the Ahlfors
$d$-regular space case, and
\cite{hmy06,hmy08,yz10,zyh20,zsy16} for the
RD-space case.
Recall that an \emph{RD-space} is a doubling
metric measure space satisfying the
following \emph{reverse doubling condition}:
there exist positive constants $\widetilde{C}
_{(\mu)}\in(0,1]$
and $\kappa\in(0,\omega]$ such that, for any
ball $B(x, r)$ with $r\in(0, \diam \cx/2)$ and
$\lambda\in[1,\diam \cx/(2r))$,
\begin{equation*}
\widetilde{C}_{(\mu)}\lambda^\kappa\mu(B(x,
r))\leq\mu(B(x, \lambda r));
\end{equation*}
here and thereafter, $\diam \cx:=\sup\{d(x,y):\
x,\ y\in \cx\}$.

Recently, Auscher and Hyt\"{o}nen  established
a wavelet system
on $(\cx,d,\mu)$ in \cite{ah13} based on
\cite[Theorem 2.2]{hk}.
Motivated by this, He et al.\ \cite[Definition
2.7]{hlyy} introduced a new kind of
approximations of the
identity with exponential decay (see also Definitions
\ref{10.23.2} and \ref{iati}
below for details).
He et al. \cite{hlyy} also established (in)homogeneous
continuous/discrete Calder\'on reproducing
formulae (which are re-stated in Lemmas \ref{crf}
and \ref{icrf} below) on a space of
homogeneous type by using
these approximations of the identity.
Using the wavelet system in \cite{ah13} and
the Calder\'{o}n reproducing formulae
in \cite{hlyy}, a real-variable theory of function
spaces on a space of homogeneous type has
been developed
rapidly. For instance, Han et al.\ \cite{hlw}
established the wavelet reproducing formulae
by using the
wavelets in \cite{ah13}. Then Han et al.\
\cite{hhl16} characterized the atomic Hardy
spaces via discrete Littlewood--Paley square
functions in \cite{hlyy} by using those formulae.
Later, Han et al.
\cite{hhl18} introduced a new kind of Hardy
spaces by using another kind of distribution
spaces.
Right after the Calder\'{o}n reproducing
formulae were established in \cite{hlyy}, He et
al.\ \cite{hhllyy}
obtained a complete real-variable theory of
atomic Hardy spaces on a
space $(\cx,d,\mu)$ of homogeneous type
with $\mu(\cx)=\fz$ which is equivalent to
$\diam\cx=\fz$
(see, for instance, Nakai and Yabuta
\cite[Lemma 5.1]{ny97}
or Auscher and Hyt\"{o}nen \cite[Lemma 8.1]
{ah13}). Besides, He et al.\ \cite{hyy19}
established a real-variable theory of local
Hardy spaces on
$(\cx,d,\mu)$ without the assumption $
\mu(\cx)=\fz$. We point out that,
in both \cite{hhllyy} and \cite{hyy19}, He et al.\
gave a complete answer to an open question
asked by Coifman
and Weiss \cite[p.\ 642]{cw77} on the radial
maximal function characterization
of Hardy spaces over spaces of homogeneous
type (see also \cite[p.\ 5]{bdl20}). Later,
Fu et al.\ \cite{fmt19} obtained a real-variable
theory of Musielak--Orlicz Hardy spaces on $\cx$.
Besides, Zhou et al. \cite{zyh20} established a
real-variable theory of
Hardy--Lorentz spaces on spaces of
homogeneous type.
On another hand, Duong and Yan \cite{dy03}
investigated Hardy spaces defined
by means of the area integral function
associated with the Poisson semigroup.
Later, Song and Yan \cite{sy18} obtained
various maximal function
characterizations of Hardy spaces associated
with operators. Moreover, Bui et al.\
\cite{bdl,bdl20,bdk,bd20}
obtained various maximal function
characterizations of a new local-type Hardy
spaces associated with operators.
Besides, S. Yang and D. Yang \cite{yy19}
established atomic and maximal function
characterizations of Musielak--Orlicz--Hardy
spaces associated to non-negative
self-adjoint operators on spaces of
homogeneous type.

As a unified frame
of many
well-known classical concrete function spaces,
the study on Besov and Triebel--Lizorkin spaces
has a long history.
We refer the reader to monographs
\cite{bin,t83,t92,t06,w88} for a comprehensive
treatment of these function spaces and their
history on $\rn$. Also, Besov and
Triebel--Lizorkin spaces on spaces of
homogeneous type with some additional
assumptions were also studied; see, for instance,
\cite{hly99,hy03,y041,y051,y052}. Han et al.\
\cite{hmy08} and M\"uller and Yang \cite{my09}
introduced and studied
Besov and Triebel--Lizorkin spaces on RD-spaces.
Later, in \cite{yz11}, Yang and Zhou established a new
characterization of these Besov and Triebel--Lizorkin spaces.
Besides, Koskela et al. introduced the Haj\l
asz--Besov and Triebel--Lizorkin spaces on
RD-spaces in \cite{kyz10,kyz11}.
Later, Grafakos et al. \cite{glmy14} developed a
systematic theory on the multilinear analysis of
Besov and Triebel--Lizorkin spaces on RD-spaces.
Besides, in \cite{my09}, M\"uller and Yang also
introduced the spaces of Lipschitz type on RD-spaces
and discussed their relations with Besov and
Triebel--Lizorkin spaces in \cite{hmy08}.
As an application,
a difference characterization of those Besov
and Triebel--Lizorkin spaces was obtained.
On one hand, using the wavelet
reproducing formulae in \cite{hlw},
Han et al. \cite{hhhlp20} introduced Besov and
Triebel--Lizorkin spaces on
spaces of homogeneous type and established
embedding theorems.
On the other hand, Wang et al. \cite{whhy} also
introduced
Besov and Triebel--Lizorkin spaces on space
of homogenous type,
based on the Calder\'{o}n reproducing formulae established
in \cite{hlyy}, and established the boundedness
of Cader\'on--Zygmund operators on these spaces
as an application. Later, He et al. \cite{hwyy20}
obtained  characterizations of
Besov and Triebel--Lizorkin spaces via
wavelets, molecules, Lusin area functions,
and Littlewood--Paley $g_\lambda^\ast$-functions
and, moreover, He et al. also showed that two
kinds of Besov and Triebel--Lizorkin spaces
studied, respectively, in \cite{hhhlp20}
and \cite{whhy} coincide.

To complete the theory of Besov and Triebel--
Lizorkin spaces on spaces of homogenous type,
it is natural to ask whether or not
we can also establish
a difference characterization for Besov and
Triebel--Lizorkin spaces on space of homogenous type.
The main target of this article is to give an
affirmative answer to this question.

Precisely speaking, in this article, we introduce
spaces of Lipschitz type on spaces of
homogeneous type (see Definitions \ref{lipi}
and \ref{lipii} below),
and discuss their relations with Besov and
Triebel--Lizorkin spaces introduced
in \cite{whhy}. As an application,
we obtain a difference characterization of
Besov and Triebel--Lizorkin spaces (see
Theorems \ref{dc} and \ref{idc} below).

The organization of the remainder of
of this article is as follows.

In Section \ref{Scrf}, we make some
preliminaries via recalling the Calder\'on
reproducing formulae, the dyadic cube system,
and some basic properties of the spaces of
homogeneous type.

In Section \ref{s2}, we recall the notion of
Besov and Triebel--Lizorkin spaces
on spaces of homogeneous type introduced in
\cite{whhy}, and introduce the Lipschitz-type
spaces on spaces of homogenous type.

In Section \ref{s3}, we discuss the relations
between Lipschitz-type spaces
and homogeneous Besov and Triebel--Lizorkin
spaces, while the inhomogeneous counterparts
are given in Section \ref{s4}.

We point out that all the proofs get rid of the
dependence on
the reverse doubling assumption of the
considered underlying space by using the
Calder\'on reproducing formulae (see Lemmas
\ref{h_c_crf}, \ref{crf}, \ref{h_c_crf},  and
\ref{crf} below) in which the related approximation of the
identity includes the terms of exponential decay,
consisting of the side length of related dyadic
cubes and the distance between point and
dyadic reference points, so that one can fully
use the geometrical properties of $\cx$. To
deal with the case $p\le 1$ but near to $1$, we use the (local) lower bound of $\cx$,
which also reflects the geometrical properties of $\cx$. Moreover,
some results in the case $p\le 1$ but near to $1$ are new
even when the underlying space is an
RD-space [see Propositions \ref{btlinlip}(iv)
and \ref{ibtlinlip}(iv) below]. To obtain these results, we
introduce some new Lipschitz-type spaces
$\lipstu$ and $\ilipstu$ (see Definition \ref{lipii}
below), in which we add a parameter $u$ that enables us to use the
boundedness of the Hardy--Littlewood maximal
operator on $L^p(\cx)$ with $p\in(1,\infty]$.

Finally, let us make some conventions on notation.
Throughout this article, $A_0$ always denotes the
positive constant appearing in the \emph{quasi-triangle
inequality} of $d$ (see Definition \ref{10.21.1}),
the parameter $\omega$ means the \emph{upper dimension}
in Definition \ref{10.21.2} [see
\eqref{eq-doub}], and $\eta$ is defined to be the
smoothness index of the exp-ATI in
Definition \ref{10.23.2} below. Moreover, $
\delta$ is a small positive number, for instance,
$\delta\leq(2A_0)^{-10}$, coming from the
construction of the
dyadic cubes on $\cx$ (see Lemma \ref{10.22.1}).
For any $r,\ a,\ b\in \rr$, define
$r_+:=\max\{0,r\}$, $a\wedge b:=\min\{a,b\}$,
and $a\vee b :=\max\{a,b\}$.
The symbol $C$ denotes a positive constant
which is independent
of the main parameters involved, but may vary
from line to line. We use $C_{(\alpha,\beta,\dots)}$
to denote a positive constant depending on the indicated
parameters $\alpha,\ \beta,\ \dots$.
The symbol $A\lesssim B$ means that $A\leq
CB$ for some positive constant
$C$, while $A\sim B$ means $A\lesssim
B\lesssim A$.
If $f\le Cg$ and $g=h$ or $g\le h$, we then
write $f\ls g\sim h$ or $f\ls g\ls h$,
\emph{rather than} $f\ls g=h$ or $f\ls g\le h$.
For any $r\in(0,\infty)$ and $x,\ y\in\cx$ with
$x\neq y$, define $V(x,y):=\mu(B(x,d(x,y)))$
and $V_r(x):=\mu(B(x,r))$. For any  $\beta,\
\gamma\in(0,\eta)$ and $s\in (-
(\beta\wedge\gamma),\beta\wedge\gamma)$,
we let
\begin{equation}\label{pseta}
p(s,\beta\wedge\gamma):=\max
\left\{\frac{\omega}{\omega+
(\beta\wedge\gamma)},
\frac{\omega}{\omega+(\beta\wedge\gamma)
+s}\right\},
\end{equation}
where $\omega$ and $\eta$ are, respectively,
as in \eqref{eq-doub} and Definition \ref{10.23.2}.
The operator $M$ always denotes the \emph{central
Hardy--Littlewood maximal operator}
which is defined by setting,
for any locally integral function $f$ on $\cx$
and any $x\in\cx$,
\begin{equation}\label{m}
M(f)(x):=\sup_{r\in(0,\infty)}\frac{1}{\mu(B(x,r))}
\int_{B(x,r)}|f(y)|\,d\mu(y).
\end{equation}
Finally, for any set $E\subset\cx$, we use
$\mathbf 1_E$ to denote its characteristic
function and,
for any set $J$, we use $\#J$ to denote its
\emph{cardinality}. For any $p\in[1,\infty]$,
we use $p'$ to denote its conjugate index,
namely, $1/p + 1/p' = 1$.

\section{Preliminaries}\label{Scrf}

In this section, we make some preliminaries.
Let us begin with the notion of
Lebesgue spaces.

\begin{definition}
The \emph{Lebesgue space} $L^p(\cx)$ for any given
$p\in(0,\infty]$ is defined by setting, when $p\in(0,\infty)$,
$$
L^p(\cx):=\left\{f \ \text{is measurable on} \ \cx :\
\|f\|_{\lp}:=\left[\int_{\cx}|f(x)|^p\,d\mu(x)\right]^{1/p}<\infty\right\},
$$
and
$$
L^\infty(\cx):=\left\{f \ \text{is measurable on}
\ \cx :\  \|f\|_{L^\infty(\cx)}:=
\displaystyle{\esssup_{x\in\cx}|f(x)|<\infty}\right\}.
$$
For any given $p\in(0,\infty)$, the \emph{locally
$p$-integrable Lebesgue spaces} $L^p_{\loc}
(\cx)$ is defined by setting
\begin{align*}
L^p_{\loc}(\cx)&:=\Big\{f \ \text{is measurable on} \ \cx :\
\text{for any}\ x\in\cx, \text{there exists an}\\
&\qquad\qquad \  r\in(0,\infty)\ \text{such that}\ \|f\|_{L^p(B(x,r))}<\infty\Big\}.
\end{align*}
\end{definition}

Now, we recall the notions of test functions and
distributions, whose
following versions were originally given in
\cite{hmy08} (see also \cite{hmy06}).

\begin{definition}
Let $x_1\in\cx$, $r\in(0,\infty)$, $\beta \in
(0,1]$, and $\gamma \in (0,\infty)$. A
measurable function $f$ on $\cx$ is
called a \emph{test function of type $(x_1, r,
\beta, \gamma)$}, if there exists a positive
constant $C$ such that
\begin{enumerate}
\item[{\rm(i)}] for any $x\in\cx$,
\begin{equation}\label{12.5.9}
|f(x)|\leq C \frac{1}{V_r(x_1)+V(x_1,x)}
\left[\frac{r}{r+d(x_1,x)}\right]^\gamma;
\end{equation}
\item[{\rm(ii)}] for any $x,\ y \in \cx$ satisfying
$d(x, y)\leq(2A_0)^{-1}[r + d(x_1, x)]$,
\begin{equation}\label{12.5.10}
|f(x)-f(y)|\leq C\left[\frac{d(x,y)}{r+d(x_1,x)}
\right]^\beta
\frac{1}{V_r(x_1)+V(x_1,x)}\left[\frac{r}
{r+d(x_1,x)}\right]^\gamma.
\end{equation}
\end{enumerate}
The set of all test functions of type $(x_1, r,
\beta, \gamma)$ is denoted by $\cg(x_1, r, \beta,
\gamma)$. For any  $f\in \cg(x_1, r, \beta,
\gamma)$, its norm $\|f\|_{\cg(x_1, r, \beta,
\gamma)}$ is defined by setting
$$
\|f\|_{\cg(x_1, r, \beta, \gamma)}:=\inf
\{C\in(0,\infty):\, \text{\eqref{12.5.9} and \eqref{12.5.10} hold
true}\}.
$$
Its subspace $\mathring{\cg}(x_1, r, \beta,
\gamma)$ is defined by setting
$$
\mathring{\cg}(x_1, r, \beta, \gamma):=
\left\{f\in\cg(x_1, r, \beta, \gamma):\int_{\cx}
f(x)\,d\mu(x)=0\right\}
$$
equipped with the norm $\|\cdot\|
_{\mathring{\cg}(x_1, r, \beta, \gamma)}:=\|
\cdot\|_{\cg(x_1, r, \beta,
\gamma)}$.
Both $\cg(x_1, r, \beta, \gamma)$ and
$\mathring{\cg}(x_1, r, \beta, \gamma)$ are called
the \emph{spaces of test functions} on $\cx$.
\end{definition}

Note that, for any fixed $x_1,\ x_2\in\cx$ and
$r_1,\ r_2 \in(0,\infty)$,
$\cg(x_1,r_1,\beta,\gamma) =
\cg(x_2,r_2,\beta,\gamma)$
and $\mathring\cg(x_1,r_1,\beta,
\gamma)=\mathring\cg(x_2,r_2,\beta,\gamma)$
with equivalent norms, but the positive
equivalence  constants
may depend on $x_1$, $x_2$, $r_1$, and $r_2$.
Thus, for fixed $x_0\in\cx$ and
$r_0\in(0,\infty)$, we may denote $\cg(x_0,
r_0, \beta, \gamma)$ and
$\mathring{\cg}(x_0, r_0, \beta, \gamma)$
simply, respectively, by $\cg(\beta,\gamma)$
and $\mathring{\cg}(\beta,\gamma)$.

Fix $\epsilon\in (0, 1]$ and $\beta,\ \gamma
\in (0, \epsilon]$. Let $\cg^\epsilon_0(\beta,
\gamma)$ [resp.,
$\mathring{\cg}^\epsilon_0(\beta, \gamma)$]
be the completion of the set $\cg(\epsilon,
\epsilon)$ [resp.,
$\mathring{\cg}(\epsilon, \epsilon)$] in $
\cg(\beta,\gamma)$ [resp., $\mathring{\cg}
(\bz,\gz)$]. Furthermore,
the norm of $\cg^\epsilon_0(\beta,\gamma)$
[resp., $\mathring{\cg}^\epsilon_0(\beta,
\gamma)$] is defined
by setting $\|\cdot\|_{\cg^\epsilon_0(\beta,
\gamma)}:=\|\cdot\|_{\cg(\beta,\gamma)}$ [resp.,
$\|\cdot\|_{\mathring{\cg}^\epsilon_0(\beta,
\gamma)}:=\|\cdot\|_{\cg(\beta,\gamma)}$]. The dual space
$(\cg^\epsilon_0(\beta,\gamma))'$ [resp., $
(\mathring{\cg}^\epsilon_0(\beta,\gamma))'$]
is defined to be the
set of all continuous linear functionals from
$\cg^\epsilon_0(\beta,\gamma)$ [resp.,
$\mathring{\cg}^\epsilon_0(\beta,\gamma)$]
to $\mathbb{C}$, equipped with the weak-$
\ast$ topology. The
spaces $(\cg^\epsilon_0(\beta,\gamma))'$ and
$(\mathring{\cg}^\epsilon_0(\beta,\gamma))'$ are called the
\emph{spaces of distributions} on $\cx$.

The  following lemma, which comes from
\cite[Theorem 2.2]{hk}, establishes the dyadic cube system of
$(\cx,d,\mu)$.

\begin{lemma}\label{10.22.1}
Let constants $0 < c_0\leq C_0 <\infty$ and
$\delta\in(0, 1)$ be such that
$12A_0^3C_0\delta\leq c_0$.
Assume that a set of points, $\{ z_\alpha^k : k
\in\zz , \alpha \in \ca_k \} \subset \cx$ with $\ca_k$ for
any $k \in\zz$ being a set of indices, has the
following properties: for any $k \in\zz$,
$$d(z^k_\alpha, z^k_\beta)\geq
c_0\delta^k\quad\text{if}
\quad\alpha\neq\beta,\quad
\text{and}\quad\min_{\alpha\in\ca_k}
d(x,z_\alpha^k)<C_0\delta^k\quad\text{for any}
\quad x\in\cx.$$
Then there exists a family of sets,
$\{ Q_\alpha^k : k \in\zz , \alpha \in \ca_k \}$, satisfying
\begin{enumerate}
\item[{\rm(i)}] for any $k\in\zz$,
$\bigcup_{\alpha\in\ca_k}
Q_\alpha^k=\cx$ and $\{ Q_\alpha^k : \alpha
\in \ca_k \}$ is disjoint;
\item[{\rm(ii)}] if $l,\ k\in\zz$ and $k\leq l$,
then, for any $\alpha\in\ca_k$ and $\beta\in\ca_l$,
either $Q_\beta^l\subset Q_\alpha^k$ or
$Q_\beta^l\cap Q_\alpha^k=\emptyset$;
\item[{\rm(iii)}] for any $k\in\zz$ and $
\alpha\in\ca_k$,
$B(z^k_\alpha, (3A_0^2)^{-1}c_0\delta^k)
\subset Q_\alpha^k\subset B(z^k_\alpha,
2A_0C_0\delta^k)$.
\end{enumerate}
\end{lemma}

Throughout this article, for any $k\in\zz$, define
$$
\cg_k:=\ca_{k+1}\setminus\ca_k\quad
\text{and}\quad\cy^k:=\left\{z^{k+1}
_\alpha\right\}_{\alpha\in\cg_k}
=: \left\{y_\alpha^k\right\}_{\alpha\in\cg_k}
$$
and, for any $x \in \cx$, define
$$
d(x,\cy^k):=\inf_{y\in\cy^k}d(x,y)
\quad\text{and}\quad V_{\delta^k}(x):=\mu(B(x,
\delta^k)).
$$

Now, we recall the notion of approximations of
the identity with exponential decay from \cite{hlyy}.

\begin{definition}\label{10.23.2}
A sequence $\{Q_k\}_{k\in\zz}$ of bounded
linear integral operators on $L^2(\cx)$ is called an
\emph{approximation of the identity with
exponential decay} (for short, exp-ATI) if there exist constants $C$,
$\nu\in(0, \infty)$, $a\in(0, 1]$, and $
\eta\in(0,1)$ such that,
for any $k \in\zz$, the kernel of the operator $Q_k$, a
function on $\cx \times \cx$ , which is still
denoted by $Q_k$,
satisfies the following conditions:
\begin{enumerate}
\item[{\rm(i)}] (the \emph{identity condition}) $
\sum_{k=-\infty}^\infty Q_k=I$
in $L^2(\cx)$, where $I$ denotes the
\emph{identity operator} on $L^2(\cx)$;
\item[{\rm(ii)}] (the \emph{size condition}) for
any $x,\ y\in\cx$,
\begin{align*}
|Q_k(x,y)|\leq C \frac{1}{\sqrt{V_{\delta^k}
(x)V_{\delta^k}(y)}}
\exp\left\{-\nu\left[\frac{d(x,y)}{\delta^k}
\right]^a\right\}h_k(x,y),
\end{align*}
here and thereafter,
$$h_k(x,y):=\exp\left\{-\nu\left[\frac{\max
\{d(x,\cy^k),d(y,\cy^k)\}}{\delta^k}
\right]^a\right\};$$
\item[{\rm(iii)}] (the \emph{regularity
condition}) for any
$x,\ x',\ y\in\cx$ with $d(x, x')\leq\delta^k$,
\begin{align*}
&|Q_k(x,y)-Q_k(x',y)|+|Q_k(y,x)-Q_k(y,x')|\\
&\quad\leq C\left[\frac{d(x,x')}{\delta^k}
\right]^\eta
\frac{1}{\sqrt{V_{\delta^k}(x)V_{\delta^k}(y)}}
\exp\left\{-\nu\left[\frac{d(x,y)}{\delta^k}
\right]^a\right\}
h_k(x,y);
\end{align*}
\item[{\rm(iv)}] (the \emph{second difference
regularity condition}) for any
$x,\ x',\ y,\ y'\in\cx$ with $d(x, x')
\leq\delta^k$ and $d(y, y')\leq\delta^k$,
\begin{align*}
&|[Q_k(x,y)-Q_k(x',y)]-[Q_k(x,y')-Q_k(x',y')]|\\
&\quad\leq C\left[\frac{d(x,x')}{\delta^k}
\right]^\eta\left[\frac{d(y,y')}{\delta^k}\right]^\eta
\frac{1}{\sqrt{V_{\delta^k}(x)V_{\delta^k}(y)}}
\exp\left\{-\nu\left[\frac{d(x,y)}{\delta^k}\right]^a\right\}
h_k(x,y);
\end{align*}
\item[{\rm(v)}] (the \emph{cancellation
condition}) for any $x,\ y \in \cx$,
$$\int_\cx Q_k(x,y')\,d\mu(y')=0=\int_\cx
Q_k(x',y)\,d\mu(x').$$
\end{enumerate}
\end{definition}

The existence of such an exp-ATI on spaces of homogeneous type
is guaranteed by \cite[Theorem 7.1]{ah13}, due to
$\mu(\cx)=\infty$, with $\eta$ the same as
in \cite[Theorem 3.1]{ah13}
which might be very small (see also
\cite[Remark 2.8(i)]{hlyy}).
However, if $d$ is a metric,
then $\eta$ can be taken arbitrarily close to 1
(see \cite[Corollary 6.13]{ht14}).

The following lemma states some basic
properties of exp-ATIs.
One can find more details in \cite[Remarks 2.8
and 2.9, and
Proposition 2.10]{hlyy}. In what follows, for
any $\gamma\in(0,\infty)$, $k\in\zz$, and
$x,\ y\in\cx$, let
\begin{equation}\label{rxy}
R_\gamma(x,y;k):= \frac{1}{V_{\delta^k}(x)+V(x,y)}
\left[\frac{\delta^k}{\delta^k+d(x,y)}
\right]^\gamma.
\end{equation}

\begin{lemma}
Let $\{Q_k\}_{k\in\zz}$ be an {\rm exp-ATI}
and $\eta\in(0,1)$ as in Definition \ref{10.23.2}.
Then, for any given $\Gamma \in (0,\infty)$,
there exists a positive constant $C$ such that,
for any $k\in\zz$, the
kernel $Q_k$ has the following properties:
\begin{enumerate}
\item[{\rm(i)}] for any $x,\ y \in\cx$,
\begin{equation}\label{10.23.3}
|Q_k(x,y)|\leq CR_\Gamma(x,y;k);
\end{equation}
\item[{\rm(ii)}] for any $x,\ x',\ y\in\cx$ with
$d(x,x')\leq(2A_0)^{-1}[\delta^k+d(x,y)]$,
\begin{align*}
|Q_k(x,y)-Q_k(x',y)|+|Q_k(y,x)-Q_k(y,x')|\leq C\left[\frac{d(x,x')}{\delta^k+d(x,y)}
\right]^\eta R_\Gamma(x,y;k);
\end{align*}
\item[{\rm(iii)}] for any $x,\ x',\ y,\ y' \in \cx$
with $d(x,x')\leq(2A_0)^{-2}[\delta^k +d(x, y)]$ and
$d(y, y')\leq(2A_0)^{-2}[\delta^k +d(x, y)]$,
\begin{align*}
&|[Q_k(x,y)-Q_k(x',y)]-[Q_k(x,y')-Q_k(x',y')]|\\
&\quad\leq C\left[\frac{d(x,x')}{\delta^k+d(x,y)}
\right]^\eta\left[\frac{d(y,y')}{\delta^k+d(x,y)}\right]^\eta
R_\Gamma(x,y;k).
\end{align*}
\end{enumerate}
\end{lemma}

Using the above exp-ATI, He et al.
\cite[Theorem 4.18]{hlyy} established
the following homogeneous continuous/
discrete Calder\'on reproducing formulae.
We point out that these homogeneous
reproducing formulae need the assumption that $\mu(\cx)=\fz$.

\begin{lemma}\label{h_c_crf}
Let $\beta,\ \gamma \in (0, \eta)$ with $\eta$
as in Definition \ref{10.23.2}, and $\{Q_k\}
_{k=-\infty}^\infty$ be an {\rm exp-ATI}.
Then there exists a sequence $\{\widetilde{Q}
_k\}_{k=-\infty}^\infty$
of bounded linear integral operators on
$L^2(\mathcal{X})$ such that,
for any $f \in (\cgg)'$,
$$f= \sum_{k=-\infty}^\infty\widetilde{Q}
_kQ_kf \qquad \text{in}\quad (\cgg)',$$
and, moreover, there exists a positive constant
$C$ such that, for any $k\in\zz$, the kernel of
$\widetilde{Q}_k$, still denoted by
$\widetilde{Q}_k$, satisfies that
\begin{enumerate}
\item[{\rm(i)}] for any $x,\ y \in\cx$,
\begin{equation}\label{4.23x}
\left|\widetilde{Q}_k(x,y)\right|\leq CR_\gamma(x,y;k);
\end{equation}
\item[{\rm(ii)}] for any $x,\ x',\ y\in\cx$ with
$d(x,x')\leq(2A_0)^{-1}[\delta^k+d(x,y)]$,
\begin{align}\label{4.23y}
&\left|\widetilde{Q}_k(x,y)-\widetilde{Q}_k(x',y)
\right|\leq C\left[\frac{d(x,x')}{\delta^k+d(x,y)}\right]^\beta
R_\gamma(x,y;k);
\end{align}
\item[{\rm(iii)}]for any $x\in\mathcal{X}$,
\begin{equation}\label{4.23a}
\int_{\mathcal{X}}\widetilde{Q}_k(x,y)
\,d\mu(y)=0=
\int_{\mathcal{X}}\widetilde{Q}_k(y,x)\,d\mu(y),
\end{equation}
where $R_\gamma(x,y;k)$ is as in \eqref{rxy}.
\end{enumerate}
\end{lemma}

To recall the homogeneous discrete Calder\'on
reproducing formula obtained in
\cite[Theorem 5.11]{hlyy}, we need  more notions.
Let $j_0\in\nn$ be sufficiently large such that
$\delta^{j_0}\leq (2A_0)^{-3}C_0$. Based on
Lemma \ref{10.22.1}, for any $k\in\zz$ and
$\alpha\in\ca_k$, let
$$
\cn(k,\alpha):=\{\tau\in\ca_{k+j_0}:
\  Q_\tau^{k+j_0}\subset Q_\alpha^k\}
$$
and $N(k,\alpha):=\#\cn(k,\alpha)$. From
Lemma \ref{10.22.1},
it follows that there exists a positive constant
$C$, independent of $k$ and $\alpha$, such that
$N(k,\alpha) \leq C \delta^{-j_0\omega}$ and
$\bigcup_{\tau\in\cn(k,\alpha)}
Q_\tau^{k+j_0}= Q_\alpha^k$. We
rearrange the set $\{Q_\tau^{k+j_0}:
\tau\in\cn(k,\alpha)\}$ as $\{\qa\}_{m=1}^{N(k,
\alpha)}$. Denote by
$\ya$ an arbitrary point in $\qa$, and
$z_\alpha^{k,m}$ the ``center" of $\qa$.

\begin{lemma}\label{crf}
Let $\{Q_k\}_{k=-\infty}^\infty$ be an {\rm exp-ATI}
and $\beta,\ \gamma \in (0, \eta)$ with
$\eta$ as in Definition \ref{10.23.2}.
For any $k\in\zz$, $\alpha\in\ca_k$, and
$m\in\{1,\dots,N(k,\alpha)\}$,
suppose that $\ya$ is an arbitrary point in $\qa$.
Then there exists a sequence $\{\widetilde{Q}
_k\}_{k=-\infty}^\infty$ of
bounded linear integral operators on
$L^2(\mathcal{X})$ such that,
for any $f \in (\cgg)'$,
$$f(\cdot) = \sum_{k=-\infty}
^\infty\sum_{\alpha \in \ca_k}
\sum_{m=1}^{N(k,\alpha)}\mu\left(\qa\right)
\widetilde{Q}_k(\cdot,\ya)Q_kf
\left(\ya\right) \qquad \text{in}\quad (\cgg)'.$$
Moreover, there exists a positive constant $C$,
independent of
$\ya$ with $k\in\zz,\ \alpha\in\ca_k$, and
$m\in\{1,\dots,N(k,\alpha)\}$,
and $f$, such that,
for any $k\in\zz$,
the kernel of $\widetilde{Q}_k$
satisfies \eqref{4.23x}, \eqref{4.23y}, and \eqref{4.23a}.
\end{lemma}

Now, we recall the inhomogeneous approximation
of the identity with exponential decay (see
\cite[Definition 6.1]{hlyy}),
which is the basis of inhomogeneous Calder\'on
reproducing formulae.

\begin{definition}\label{iati}
Let $\eta\in(0,1)$ be as in Definition \ref{10.23.2}.
A sequence $\{Q_k\}_{k=0}^\infty$ of bounded
linear integral operators on
$L^2(\cx)$ is called an  \emph{inhomogeneous
approximation
of the identity with exponential decay} (for
short, exp-IATI) if $\{Q_k\}_{k=0}^\infty$ has
the following properties:
\begin{enumerate}
\item[{\rm(i)}] $\sum_{k=0}^\infty Q_k=I$ in
$L^2(\cx)$;
\item[{\rm(ii)}] for any $k\in\nn$, $Q_k$
satisfies (ii) through (v) of Definition \ref{10.23.2};
\item[{\rm(iii)}] $Q_0$ satisfies (ii), (iii), and (iv)
of Definition \ref{10.23.2} with $k=0$ but
without the term
$$\exp\left\{-\nu\left[\max\left\{d(x,
\cy^0),d(y,\cy^0)\right\}\right]^a\right\};$$
moreover, for any $x\in\cx$,
$$\int_\cx Q_0(x,y)\,d\mu(y)=1=\int_\cx
Q_0(y,x)\,d\mu(y).$$
\end{enumerate}
\end{definition}

Note that the existence of such an exp-IATI
does not need the assumption that
$\mu(\cx)=\infty$.

Next, we state the inhomogeneous continuous
and  discrete Calder\'on reproducing formulae
which were obtained, respectively, in
\cite[Theorem 6.8]{hlyy} and \cite[Theorem 6.13]{hlyy}.

\begin{lemma}\label{ih_c_crf}
Let $\beta,\ \gamma \in (0, \eta)$ with
$\eta$ as in Definition \ref{10.23.2},
and $\{Q_k\}_{k\in\zz_+}$ an {\rm exp-IATI}.
Then there exist an $N\in\nn$ and
a sequence $\{\widetilde{Q}
_k\}_{k\in\zz_+}$
of bounded linear integral operators on
$L^2(\mathcal{X})$ such that,
for any $f \in (\icgg)'$,
$$f= \sum_{k=0}^\infty\widetilde{Q}_kQ_kf
\qquad \text{in}\quad (\icgg)',$$
where, for any $k\in\zz_+$, the kernel of $
\widetilde{Q}_k$,
still denoted by $\widetilde{Q}_k$,  satisfies
\eqref{4.23x}, \eqref{4.23y}, and the following
integral condition: for any $x\in\mathcal{X}$,
$$\int_{\mathcal{X}}\widetilde{Q}_k(x,y)
\,d\mu(y)=
\int_{\mathcal{X}}\widetilde{Q}_k(y,x)
\,d\mu(y)=\begin{cases}
1 &\text{if } k \in \{0,\dots,N\},\\
0 &\text{if } k\in \{N+1,N+2,\ldots\}.
\end{cases}$$
\end{lemma}

\begin{lemma}\label{icrf}
Let $\beta,\ \gamma \in (0, \eta)$ with $\eta$
as in Definition
\ref{10.23.2}, and $\{Q_k\}_{k\in\zz_+}$ an
{\rm exp-IATI}. For any $k\in\zz_+,$ $
\alpha\in\ca_k$, and
$m\in\{1,\dots,N(k,\alpha)\}$, suppose that $
\ya$ is an arbitrary point in $\qa$.
Then there exist an $N\in\nn$ and a sequence
$\{\widetilde{Q}_k\}_{k\in\zz_+}$ of bounded
linear integral operators on $L^2(\mathcal{X})$ such
that, for any $f \in (\icgg)'$,
\begin{align*}
f(\cdot) &= \sum_{\alpha \in \ca_0}\sum_{m=1}
^{N(0,\alpha)}\int_{\qo}
\widetilde{Q}_0(\cdot,y)\,d\mu(y)Q^{0,m}
_{\alpha,1}(f)\\
&\qquad+\sum_{k=1}^N\sum_{\alpha \in \ca_k}
\sum_{m=1}^{N(k,\alpha)}
\mu\left(\qa\right)\widetilde{Q}_k(\cdot,
\ya)Q^{k,m}_{\alpha,1}(f)\\
& \qquad+\sum_{k=N+1}^\infty\sum_{\alpha
\in \ca_k}\sum_{m=1}^{N(k,\alpha)}
\mu\left(\qa\right)\widetilde{Q}_k(\cdot,
\ya)Q_kf\left(\ya\right)
\end{align*}
in $(\icgg)',$ where, for any $k\in\{0,\ldots,N\}$,
$\alpha\in\ca_k$, and $m\in\{1,\dots,N(k,\alpha)\}$,
$$Q^{k,m}_{\alpha,1}(f):=\frac{1}{\mu(\qa)}
\int_{\qa}Q_kf(u)\,d\mu(u).$$
Moreover, for any $k\in\zz_+$, the kernel of
$\widetilde{Q}_k$,
still denoted by $\widetilde{Q}_k$, satisfies the
same conditions as in Lemma
\ref{ih_c_crf} with the positive constant $C$
independent of $y_\az^{k,m}$, where
$\alpha\in\ca_k$ and $m\in\{1,\dots,N(k,\alpha)\}$.
\end{lemma}

\section{Besov and Triebel--Lizorkin spaces
and spaces of Lipschitz-type on spaces of
homogeneous type}\label{s2}

In this section, we introduce some notions of
function spaces
on spaces of homogeneous type. Let us begin
with the notions
of Besov and Triebel--Lizorkin spaces on
spaces of homogeneous type;
see \cite[Definitions 3.1, 4.1,  5.1 and 5.8]{whhy}.

\begin{definition}\label{h}
Let $\beta,\ \gamma \in (0, \eta)$ with $\eta$ as in
Definition \ref{10.23.2}, $s\in(-(\beta\wedge\gamma), \beta\wedge\gamma)$,
and $p,\ q\in(0,\infty]$ satisfy
\begin{equation*}
\beta\in\left(\max\left\{0,-s+
\omega\left(\frac{1}{p}-1\right)_+\right\},
\eta\right)\quad \textit{and}\quad
\gamma\in\left(\max\left\{s,
\omega\left(\frac{1}{p}-1\right)_+\right\},\eta\right)
\end{equation*}
with $\omega$ as in \eqref{eq-doub}.
Let $\{Q_k\}_{k\in\zz}$ be an exp-ATI.
\begin{enumerate}
\item[\rm{(i)}] Let $p\in(p(s,
\beta\wedge\gamma),\infty]$
with $p(s,\beta\wedge\gamma)$ as in
\eqref{pseta}, and $q \in (0,\infty]$.
The \emph{homogenous Besov space} $\hb$ is
defined by setting
$$
\hb := \left\{f  \in\lf(\cgg\r)' :\  \|f\|
_{\hb}:=\left[\sum_{k \in \zz}
\delta^{-ksq}\|Q_k(f)\|_{\lp}^q\right]^{1/q}
<\infty\right\}
$$
with usual modifications made when $p=\infty$
or $q=\infty$.
\item[\rm{(ii)}] Let $p\in(p(s,
\beta\wedge\gamma),\infty)$ and $q \in (p(s,
\beta\wedge\gamma),\infty]$.
The \emph{homogenous Triebel--Lizorkin space} $\hf$ is defined by setting
$$
\hf := \left\{f  \in \lf(\cgg\r)' :\  \|f\|
_{\hf}:=\left\|\left[\sum_{k \in \zz}
\delta^{-ksq}|Q_k(f)|^q\right]^{1/q}\right\|_{\lp}
<\infty\right\}
$$
with usual modification made when $q=\infty$.
\end{enumerate}
\end{definition}

For any measurable set $E\subset \cx$ with
$\mu(E)\in(0,\fz)$,
and for any non-negative measurable
function $f$, let
\begin{equation*}
m_E(f):=\frac{1}{\mu(E)}\int_E f(x)\,d\mu(x).
\end{equation*}
Now, we recall the notion of inhomogeneous
Besov and Triebel--Lizorkin spaces on spaces
of homogeneous type.

\begin{definition}\label{ih}
Let  $\beta,\ \gamma \in (0, \eta)$ with
$\eta$ as in
Definition \ref{10.23.2}, $s\in(-
(\beta\wedge\gamma),\beta\wedge\gamma)$,
and $p,\ q\in(0,\infty]$ satisfy
\begin{equation*}
\beta\in\left(\max\left\{0,-s+\omega
\left(\frac{1}{p}-1\right)_+\right\},\eta\right)
\qquad\text{and}\qquad
\gamma\in\left(\omega\left(\frac{1}
{p}-1\right)_+,\eta\right)
\end{equation*}
with $\omega$ as in \eqref{eq-doub}.
Let $\{Q_k\}_{k\in\zz_+}$ be an exp-IATI and
$N\in\nn$ as in Lemma \ref{icrf}.
\begin{enumerate}
\item[\rm{(i)}] Let $p\in(p(s,
\beta\wedge\gamma),\infty]$ with $p(s,
\beta\wedge\gamma)$ as in \eqref{pseta},
and $q \in (0,\infty]$. The
\emph{inhomogeneous Besov space $\ihb$} is
defined by setting
\begin{align*}
\ihb :={}& \left\{f  \in (\icgg)' :\  \|f\|_{\ihb}:=
\left\{\sum_{k=0}^N\sum_{\alpha \in \ca_k}
\sum_{m=1}^{N(k,\alpha)}
\mu\left(\qa\right)\left[m_{\qa}\left(|Q_k(f)|
\right)\r]^p\right\}^{1/p}\right.\\
&\qquad\qquad\qquad\qquad\qquad\qquad
\qquad\left.+\left[\sum_{k=N+1}^{\infty} \delta^{-ksq}
\|Q_k(f)\|_{\lp}^q\right]^{1/q}<\infty\right\}
\end{align*}
with usual modifications made when
$p=\infty$ or  $q=\infty$.
\item[\rm{(ii)}] Let $p\in(p(s,
\beta\wedge\gamma),\infty)$ and
$q \in (p(s,\beta\wedge\gamma),\infty]$. The
\emph{inhomogeneous
Triebel--Lizorkin space $\ihf$} is defined by setting
\begin{align*}
\ihf :={}& \left\{f  \in (\icgg)' :\  \|f\|
_{\ihf}:=\left\{\sum_{k=0}^N
\sum_{\alpha \in \ca_k}\sum_{m=1}^{N(k,
\alpha)}\mu\left(\qa\right)\left[m_{\qa}
\left(|Q_k(f)|\right)\r]^p\right\}^{1/p}\r.\\
&\qquad\qquad\qquad\qquad\qquad\qquad
\qquad\left.+\left\|\left[\sum_{k=N+1}^\infty
\delta^{-ksq}|Q_k(f)|^q\right]^{1/q}\right\|_{\lp}<\infty\right\}
\end{align*}
with usual modification made when $q=\infty$.
\end{enumerate}
\end{definition}

\begin{definition}\label{hfi}
Let $\beta,\ \gamma \in (0, \eta)$, $s\in(-
(\beta\wedge\gamma),\beta\wedge\gamma)$,
and $q\in (p(s,\beta\wedge\gamma),\infty]$ with
$\eta$ as in Definition \ref{10.23.2} and
$p(s,\beta\wedge\gamma)$ as in
\eqref{pseta}. Let $\{Q_k\}_{k\in\zz}$ be an
exp-ATI. For  any $k\in\zz$ and $\alpha\in \ca_k$,
let $Q_\az^k$ be as in Lemma \ref{10.22.1}.
Then the \emph{homogeneous Triebel--
Lizorkin space $\hfi$} is defined by setting
\begin{align*}
\hfi := \Bigg\{f  \in (\cgg)' :\  \|f\|_{\hfi}:={}&\sup_{l \in \zz}
\sup_{\alpha\in\ca_l}\left[\frac{1}
{\mu(Q_\alpha^l)}\r.\\
&\quad\left.\left.\times\int_{Q_\alpha^l}
\sum_{k=l}^\infty\delta^{-ksq}
|Q_k(f)(x)|^q\,d\mu(x)\right]^{1/q}<\infty\right\}
\end{align*}
with usual modification made when $q=\infty$.
\end{definition}

\begin{definition}\label{ihfi}
Let  $\beta,\ \gamma \in (0, \eta)$ with
$\eta$ as in
Definition \ref{10.23.2}, and $s\in(-
(\beta\wedge\gamma),\beta\wedge\gamma)$.
Let $\{Q_k\}_{k\in\zz}$ be an exp-IATI and
$N\in\nn$ as in Lemma \ref{icrf}.
The \emph{inhomogeneous Triebel--Lizorkin
space $\ihfi$} is defined by setting
\begin{align*}
\ihfi := &{}\left\{f  \in (\icgg)' :\  \|f\|_{\ihfi}:={}
\max\left[\sup_{k\in\{0,\dots,N\}}
\sup_{\alpha \in \ca_k}\sup_{m\in\{1,\dots,N(k,
\alpha)\}}m_{\qa}(|Q_k(f)|),\r.\r.\\
&\qquad\qquad\qquad\left.\left.\sup_{l
\in \nn,\ l>N} \sup_{\alpha\in\ca_l}\left
\{\frac{1}{\mu(Q_\alpha^l)}
\int_{Q_\alpha^l}\sum_{k=l}^\infty\delta^{-ksq}|
Q_k(f)(x)|^q\,d\mu(x)\right\}^{1/q}
\r]<\infty\right\}
\end{align*}
with usual modification made when $q=\infty$.
\end{definition}

\begin{remark}
\begin{enumerate}
\item[{\rm(i)}] Recall that homogeneous Besov
and Triebel--Lizorkin spaces on $\cx$ need to
require $\mu(\cx)=\infty$ which is not
necessary for inhomogeneous Besov
and Triebel--Lizorkin spaces on $\cx$.
\item[{\rm(ii)}] It was proved in \cite{whhy} that those spaces
defined in Definitions \ref{h}, \ref{ih}, \ref{hfi},
and \ref{ihfi} are independent of the choices of
$\beta$, $\gamma$, and exp-(I)ATIs (see
\cite[Propositions 3.12, 3.15, 4.3,
4.4, 5.4, 5.5, 5.10, and 5.11]{whhy}), which
makes those spaces well defined.
\end{enumerate}
\end{remark}

Next, we introduce the notion of Lipschitz-type
spaces on spaces of homogeneous type,
which originates from \cite{s70,jw84,j94}; see also
\cite[Definition 3.1]{my09}.

\begin{definition}\label{lipi}
Let $p,\ q\in(0,\infty]$, $s\in(0,\infty)$, and
$\widetilde{C}\in (0,\infty)$ be a constant.
\begin{enumerate}
\item[{\rm(i)}] A function $f\in L_\loc^p(\cx)$
is said to belong to the \emph{Lipschitz-type
space} $\lips$ if
\begin{align*}
\|f\|_{\lips}&:=\left\{\sum_{k=-\infty}
^\infty\delta^{-ksq}
\left[\int_{\cx}\frac{1}{\mu(B(x,\widetilde{C}
\delta^k))}\int_{B(x,
\widetilde{C}\delta^k)}
|f(x)-f(y)|^p\,d\mu(y)\,d\mu(x)\right]^{q/p}
\right\}^{1/q},
\end{align*}
with usual modifications made when
$p=\infty$ or  $q=\infty$, is finite.
\item[{\rm(ii)}] A function $f\in L^p(\cx)$ is
said to belong to the \emph{Lipschitz-type
space} $\ilips$ if
$$\|f\|_{\lips}<\infty.$$
Moreover, for any $f\in\ilips$, let
$$\|f\|_{\ilips}:=\|f\|_{L^p(\cx)}+\|f\|_{\lips}.$$
\item[{\rm(iii)}] A function $f\in L_\loc^1(\cx)$
is said to belong to the \emph{Lipschitz-type
space} $\lipsb$ if
\begin{align*}
\|f\|_{\lipsb}&:=\left(\sum_{k=-\infty}
^\infty\delta^{-ksq}
\left\{\int_{\cx}\left[\frac{1}{\mu(B(x,
\widetilde{C}\delta^k))}\int_{B(x,
\widetilde{C}\delta^k)}
|f(x)-f(y)|\,d\mu(y)\right]^p\,d\mu(x)\right\}
^{q/p}\right)^{1/q},
\end{align*}
with usual modifications made when
$p=\infty$ or  $q=\infty$, is finite.
\item[{\rm(iv)}] A function $f\in L^p(\cx)$ is
said to belong to the \emph{Lipschitz-type
space} $\ilipsb$ if
$$\|f\|_{\lipsb}<\infty.$$
Moreover, for any $f\in\ilipsb$, let
$$\|f\|_{\ilipsb}:=\|f\|_{L^p(\cx)}+\|f\|_{\lipsb}.$$
\end{enumerate}
\end{definition}

When $\cx$ is a $d$-set of $\rn$, the
Lipschitz-type space $\lips$,
for any given $s\in\rr$ and $p,\ q\in (0,\infty]$,
was introduced in \cite{j96,jw97};
see also \cite{g02,ghl03}.
When $\cx$ is an Ahlfors $n$-regular metric
measure space, these spaces were introduced
in \cite{yl04}. We also point out that,
when $\cx$ is a metric measure space, these
spaces may be non-trivial even
when $s\in(1,\infty)$ (see, for instance,
\cite{yl04} for more details).

It is easy to show that the spaces $\lips$,
$\ilips$, $\lipsb$, and $\ilipsb$
are independent of the choice of $\widetilde{C}
$ (see, for instance, \cite{j96} and \cite{yl04}).
These spaces have the following basic properties, whose
proofs are similar to those presented in \cite{yl04},
and we omit the details here.

\begin{proposition}\label{pro-lipi}
Let $s\in(0,\infty)$.
\begin{enumerate}
\item[{\rm(i)}] If $p\in[1,\infty]$ and $q\in (0,\infty]$,
then $$\lips\subset\lipsb\quad \text{and}
\quad \ilips\subset\ilipsb.$$
\item[{\rm(ii)}] If $p\in(0,\infty]$ and
$0<q_0\leq q_1\leq\infty$, then
$\dot{L}(s,p,q_0;\cx)\subset\dot{L}(s,p,q_1;\cx)$,
$\dot{L}_b(s,p,q_0;\cx)\subset\dot{L}_b(s,p,q_1;\cx)$,
$L(s,p,q_0;\cx)\subset L(s,p,q_1;\cx)$, and
$L_b(s,p,q_0;\cx)\subset L_b(s,p,q_1;\cx)$.
\item[{\rm(iii)}] If $p\in(0,\infty]$, $q_0,\ q_1\in(0,\infty]$,
and $\varepsilon\in(0,\infty)$, then
$L(s+\varepsilon,p,q_0;\cx)\subset L(s,p,q_1;\cx)$
and
$L_b(s+\varepsilon,p,q_0;\cx)\subset L_b(s,p,q_1;\cx)$.
\end{enumerate}
\end{proposition}

Now, we introduce the following Lipschitz-type
spaces $\lipstu$ and $\ilipstu$.
\begin{definition}\label{lipii}
Let $p,\ q\in(0,\infty]$, $s,\ u\in(0,\infty)$, and
$\widetilde{C}\in (0,\infty)$ be a constant..
\begin{enumerate}
\item[{\rm(i)}]
A function $f\in L_\loc^u(\cx)$ is said to
belong to the \emph{Lipschitz-type space}
$\lipstu$ if
\begin{align*}
\|f\|_{\lipstu}&:=\left\|\left\{\sum_{k=-\infty}
^\infty\delta^{-ksq}
\left[\frac{1}{\mu(B(\cdot,\widetilde{C}
\delta^k))}\int_{B(\cdot,\widetilde{C}\delta^k)}
|f(\cdot)-f(y)|^u\,d\mu(y)\right]^{\frac{q}{u}}
\right\}^{1/q}\right\|_{L^p(\cx)},
\end{align*}
with usual modification made when
$q=\infty$, is finite.
\item[{\rm(ii)}]
A function $f\in L^p(\cx)$ is said to belong to
the \emph{Lipschitz-type space} $\ilipstu$
if
\begin{align*}
&\|f\|_{\ilipstu}\\
&\hs:=\|f\|_{L^p(\cx)}+\left\|\left
\{\sum_{k=0}^\infty\delta^{-ksq}
\left[\frac{1}{\mu(B(\cdot,\widetilde{C}
\delta^k))}\int_{B(\cdot,
\widetilde{C}\delta^k)}
|f(\cdot)-f(y)|^u\,d\mu(y)\right]^{\frac{q}{u}}
\right\}^{1/q}\right\|_{L^p(\cx)},
\end{align*}
with usual modification made when
$q=\infty$, is finite.
\end{enumerate}
\end{definition}

It is easy to see that the spaces $\ilipstu$
with $u\in(0,p)$, and $\lipstu$ are independent
of the choice of $\widetilde{C}$ (see, for
instance, \cite{j96} and \cite{yl04}).
If $u=1$, we denote $\lipstu$ and $\ilipstu$
simply, respectively, by $\lipst$ and $\ilipst$
which  were also introduced in \cite{yl04} when
$\cx$ is an Ahlfors
$n$-regular metric measure space.
Now, we give some basic properties of $\lipst$
and $\ilipst$, which can be proved in a
way similar to that used in the proof
of \cite[Proposition 3.5]{yl04};
we omit the details here.

\begin{proposition}\label{pro-lipii}
Let $s\in(0,\infty)$.
\begin{enumerate}
\item[{\rm(i)}] If $p\in(0,\infty]$ and
$0<q_0\leq q_1\leq\infty$, then
$\dot{L}_t(s,p,q_0;\cx)\subset\dot{L}
_t(s,p,q_1;\cx)$ and
$L_t(s,p,q_0;\cx)\subset L_t(s,p,q_1;\cx)$.
\item[{\rm(ii)}] If $p\in(0,\infty]$, $q_0,\
q_1\in(0,\infty]$,
and $\varepsilon\in(0,\infty)$, then $L_t(s+
\varepsilon,p,q_0;\cx)\subset L_t(s,p,q_1;\cx)$.
\item[{\rm(iii)}] If $p,\ q\in(0,\infty]$, then
$$\dot{L}_b(s,p\min(p,q);\cx)\subset \lipst
\subset \dot{L}_b(s,p,\max(p,q);\cx)$$
and
$$L_b(s,p\min(p,q);\cx)\subset \ilipst \subset
L_b(s,p,\max(p,q);\cx).$$
\item[{\rm (iv)}] If $p\in(0,\infty]$, then
$\dot{L}_b(s,p,p;\cx)=\dot{L}_t(s,p,p;\cx)$
and $L_b(s,p,p;\cx)=L_b(s,p,p;\cx)$.
\end{enumerate}
\end{proposition}

The following characterizations of $\ilips$ and
$\ilipst$ come from \cite[Proposition 3.3]{my09}
which still holds true in the setting of spaces of
homogenous type; we omit the details here.

\begin{proposition}\label{llblt}
Let $s\in(0,\infty)$ and $\widetilde{C}
\in(0,\infty)$ be a constant.
\begin{enumerate}
\item[{\rm(i)}] Let $p,\ q\in (0,\infty]$. Then
$f\in\ilips$ if and only if $f\in L^p(\cx)$ and
\begin{align*}
&\|f\|_{\widetilde{L}(s,p,q;\cx)}\\
&\hs:=\left
\{\sum_{k=0}^\infty\delta^{-ksq}
\left[\int_{\cx}\frac{1}{\mu(B(x,\widetilde{C}
\delta^k))}\int_{B(x,
\widetilde{C}\delta^k)}
|f(x)-f(y)|^p\,d\mu(y)\,d\mu(x)\right]^{q/p}
\right\}^{1/q}\notag
\end{align*}
is finite. Moreover, in this case,
$$\|f\|_{\ilips}\sim\|f\|_{L^p(\cx)}+\|f\|
_{\widetilde{L}(s,p,q;\cx)}$$
with the positive equivalence constants
independent of $f$.
\item[{\rm(ii)}]Let $p\in[1,\infty]$ and
$q\in(0,\infty]$. Then $f\in\ilipsb$ if and only
if $f\in L^p(\cx)$ and
\begin{align*}
&\|f\|_{\widetilde{L}_b(s,p,q;\cx)}\\
&\quad:=\left(\sum_{k=0}^\infty\delta^{-ksq}
\left\{\int_{\cx}\left[\frac{1}{\mu(B(x,
\widetilde{C}\delta^k))}\int_{B(x,
\widetilde{C}\delta^k)}
|f(x)-f(y)|\,d\mu(y)\right]^p\,d\mu(x)\right\}
^{q/p}\right)^{1/q}\notag
\end{align*}
is finite.
 Moreover, in this case,
$$\|f\|_{\ilipsb}\sim\|f\|_{L^p(\cx)}+\|f\|
_{\widetilde{L}_b(s,p,q;\cx)}$$
with the positive equivalence constants
independent of $f$.
\end{enumerate}
\end{proposition}

\section{Relations with homogeneous Besov
and Triebel--Lizorkin spaces}\label{s3}

In this section, we study the relations between
spaces of Lipschitz-type
and homogeneous Besov spaces or Triebel--
Lizorkin spaces on spaces of homogenous type.
In this section, we always assume that
$\mu(\cx)=\infty$. Let us begin with the following proposition.

\begin{proposition}\label{lipinbtl}
Let $\eta$ be as in Definition \ref{10.23.2},
$\beta,\ \gamma\in(0,\eta)$, and
$s\in(0,\beta\wedge\gamma)$.
\begin{enumerate}
\item[{\rm(i)}] If $p\in(\omega/[\omega+
(\beta\wedge\gamma)],\infty]$
and $q\in(0,\infty]$, then $\lipsb\subset \hb$;
\item[{\rm(ii)}] If $p,\ q\in(\omega/[\omega+
(\beta\wedge\gamma)],\infty]$,
then $\lipst\subset \hf$.
\end{enumerate}
\end{proposition}

To prove Proposition \ref{lipinbtl}, we need the
following basic and useful inequality.

\begin{lemma}
Let $\theta\in (0,1]$. It then holds true that, for
any $\{a_j\}_{j\in\nn} \subset \mathbb{C}$,
\begin{equation}\label{r}
\left(\sum_{j=1}^\infty|a_j|
\right)^\theta\leq\sum_{j=1}^\infty|a_j|^\theta.
\end{equation}
\end{lemma}

The following lemma contains  some basic and
very useful estimates
related to $d$ and $\mu$ on a space
$(\cx,d,\mu)$ of homogeneous type.
One can find the details in \cite[Lemma 2.1]
{hmy08} or \cite[Lemma 2.4]{hlyy}.

\begin{lemma}\label{6.15.1}
Let $\beta,\ \gamma\in(0,\infty)$.
\begin{enumerate}
\item[{\rm(i)}] For any $x,\ y\in \cx$ and $r \in (0, \infty)$,
$V(x, y)\sim V (y, x)$ and
$$V_r(x) + V_r(y) + V (x,y) \sim V_r(x) + V (x,y)
\sim V_r(y) + V (x,y) \sim \mu(B(x,r + d(x,y)));$$
moreover, if $d(x,y)\leq r$, then $V_r(x)\sim V_r(y)$.
Here the positive equivalence  constants are
independent of $x$, $y$, and $r$.
\item[{\rm(ii)}] There exists a positive constant
$C$ such that,
for any $x_1 \in \cx$ and $r \in (0, \fz)$,
$$
\int_\cx \frac{1}{V_r(x_1)+V(x_1,y)}\left[\frac{r}
{r+d(x_1,y)}\right]^\gamma\,d\mu(y)\leq C.
$$
\item[{\rm(iii)}] There exists a positive constant
$C$ such that, for any $x \in \cx $ and $R \in (0, \fz)$,
$$\int_{\{z\in \cx:\ d(x,z)\leq R\}}\frac{1}{V(x,y)}
\left[\frac{d(x,y)}{R}\right]^\beta\,d\mu(y)\leq C$$
and
$$\quad \int_{\{z\in \cx:\ d(x,z)\geq R\}}
\frac{1}{V(x,y)}
\left[\frac{R}{d(x,y)}\right]^\beta\,d\mu(y)\leq C.$$
\item[{\rm(iv)}] There exists a positive constant
$C$ such that, for any $x_1\in\cx$ and $r,\
R\in (0,\fz)$,
$$\int_{\{x\in\cx:\ d(x_1,x)\geq R\}}\frac{1}
{V_r(x_1)+V(x_1,x)}
\left[\frac{r}{r+d(x_1,x)}
\right]^\gamma\,d\mu(x)\leq
C\left(\frac{r}{r+R}\right)^\gamma.$$
\end{enumerate}
\end{lemma}

Now, we show Proposition \ref{lipinbtl}.

\begin{proof}[Proof of Proposition \ref{lipinbtl}]
We first prove (i). We claim that, for any given
$\beta,\ \gamma\in(0,\eta)$,
$s\in(0,\beta\wedge\gamma)$, $p\in(\omega/
[\omega+(\beta\wedge\gamma)],\infty]$,
and $q\in(0,\infty]$,
\begin{equation}\label{lipb-d}
\lipsb\subset (\cgg)'.
\end{equation}
Indeed, assume that $f\in\dot{L}_b(s,p,\infty;
\cx)$ and $x_1\in\cx$ is a fixed point.
Then $f$ is finite for almost every $x\in\cx$.
Without loss of generality,
we may assume that, for any $x\in\cx$,
$f(x)<\infty$. To simplify the presentation of
this proof, here and thereafter, for any
$k\in\zz$ and $x\in\cx$, let
\begin{equation}\label{jkf}
J_k(f;x):=\frac{1}{\mu(B(x,
\widetilde{C}\delta^{-k}))}
\int_{B(x,\widetilde{C}\delta^{-k})}|f(x)-f(y)|
\,d\mu(y).
\end{equation}
For any $g\in\cgg$, by Lemma \ref{6.15.1}, we
have, for any $x\in\cx$,
\begin{align}\label{4.2x}
\left|\int_{\cx}f(y)g(y)\,d\mu(y)\right|
&=\left|\int_{\cx}[f(x)-f(y)]g(y)\,d\mu(y)
\right|\\
&\lesssim \|g\|_{\cgg}\int_{\cx}|f(x)-f(y)|
R_\gamma(x_1,y;0)\,d\mu(y)\notag\\
&\lesssim\|g\|_{\cgg}\sum_{j=0}
^\infty\delta^{j\gamma}J_j(f;x),\notag
\end{align}
where $R_\gamma(x_1,y;0)$ is as in \eqref{rxy}.
When $p\in(\omega/[\omega+
(\beta\wedge\gamma)],1]$,
from \eqref{4.2x}, \eqref{r}, and
$s\in(0,\beta\wedge\gamma)$,
we deduce that
\begin{align*}
\left|\int_{\cx}f(y)g(y)\,d\mu(y)\right|
&=\left\{\frac{1}{\mu(B(x_1,1))}
\int_{B(x_1,1)}\left|\int_{\cx}[f(x)-f(y)]g(y)
\,d\mu(y)\right|^p\,d\mu(x)\right\}^{1/p}\\
&\lesssim\|g\|_{\cgg}\left\{\sum_{j=0}^\infty
\delta^{j\gamma p}\int_{\cx}\left[J_j(f;x)\right]^p\,d\mu(x)\right\}^{1/p}\\
&\lesssim\|g\|_{\cgg}\|f\|_{\dot{L}
_b(s,p,\infty;\cx)};
\end{align*}
while, when $p\in(1,\infty]$, by the Minkowski
inequality, \eqref{4.2x}, and
$s\in(0,\beta\wedge\gamma)$, we have
\begin{align*}
\left|\int_{\cx}f(y)g(y)\,d\mu(y)\right|
&\lesssim\|g\|_{\cgg}\sum_{j=0}^\infty
\delta^{j\gamma}\left\{\int_{\cx}\left[J_j(f;x)\right]^p\,d\mu(x)\right\}^{1/p}\\
&\lesssim\|g\|_{\cgg}\|f\|_{\dot{L}
_b(s,p,\infty;\cx)}.
\end{align*}
By this and Proposition \ref{pro-lipi}(ii), we
conclude that \eqref{lipb-d} holds true.

Now, we prove that $\lipsb\subset\hb$ for any
given $\beta,\ \gamma\in(0,\eta)$,
$s\in(0,\beta\wedge\gamma)$, $p\in(\omega/
[\omega+(\beta\wedge\gamma)],\infty]$,
and $q\in(0,\infty]$. To this end, let
$f\in\lipsb$ and $\{Q_k\}_{k=-\infty}^\infty$
be an exp-ATI.
Then, by the cancellation and the size conditions of $Q_k$,
we have, for any $k\in\zz$ and $x\in\cx$,
\begin{align}\label{qfi}
|Q_kf(x)|&=\left|\int_{\cx}Q_k(x,y)[f(x)-f(y)]
\,d\mu(y)\right|\\
&\lesssim\int_{\cx}R_\gamma(x,y;k)|f(x)-f(y)|\,d\mu(y)
\lesssim \sum_{j=0}^\infty\delta^{j\gamma}
J_{j-k}(f;x),\notag
\end{align}
where $R_\gamma(x,y;k)$ is as in \eqref{rxy}.

If $p\in(\omega/[\omega+
(\beta\wedge\gamma)],1]$,
by \eqref{r} and \eqref{qfi}, we find that,
for any $k\in\zz$,
\begin{equation}\label{qfii}
\|Q_kf\|_{L^p(\cx)}\lesssim \left\{\sum_{j=0}
^\infty\delta^{j\gamma p}\int_{\cx}\left[J_{j-k}(f;x)\right]^p\,d\mu(x)\right\}^{1/p}.
\end{equation}
When $q/p\in(0,1]$, from \eqref{qfii}, \eqref{r},
and $s\in(0,\gamma)$, we deduce that
\begin{align}\label{blb1}
\|f\|_{\hb}&=\left\{\sum_{k=-\infty}
^\infty\delta^{-ksq}\|Q_kf\|_{L^p(\cx)}
^q\right\}^{1/q}\\
&\lesssim\left[\sum_{j=0}
^\infty\delta^{j\gamma p}
\sum_{k=-\infty}^\infty\delta^{-ksq}\left
\{\int_{\cx}
\left[J_{j-k}(f;x)\right]^p\,d\mu(x)\right\}^{q/p}
\right]^{1/q}\notag\\
&\lesssim\|f\|_{\lipsb}\left(\sum_{j=0}
^\infty\delta^{j(\gamma-s)p}\right)^{1/q}
\lesssim\|f\|_{\lipsb};\notag
\end{align}
while, when $q/p\in(1,\infty]$, by \eqref{qfii}
and the Minkowski inequality,
we conclude that
\begin{align}\label{blb2}
\|f\|_{\hb}&=\left\{\sum_{k=-\infty}
^\infty\delta^{-ksq}\|Q_kf\|_{L^p(\cx)}
^q\right\}^{1/q}\\
&\lesssim\left[\sum_{j=0}
^\infty\delta^{j\gamma p}
\left(\sum_{k=-\infty}^\infty\delta^{-ksq}\left\{\int_{\cx}
\left[J_{j-k}(f;x)\right]^p\,d\mu(x)\right\}
^{q/p}\right)^{p/q}\right]^{1/p}\notag\\
&\lesssim\|f\|_{\lipsb}.\notag
\end{align}

If $p\in(1,\infty]$, then, from \eqref{qfi} and the
Minkowski inequality,
we deduce that, for any $k\in\zz$,
\begin{equation}\label{qfiii}
\|Q_kf\|_{L^p(\cx)}\lesssim \sum_{j=0}
^\infty\delta^{j\gamma}\left\{\int_{\cx}
\left[J_{j-k}(f;x)\right]^p\,d\mu(x)\right\}^{1/p}.
\end{equation}
When $q\in(0,1]$, by \eqref{r} and \eqref{qfiii},
we obtain
\begin{align}\label{blb3}
\|f\|_{\hb}&=\left\{\sum_{k=-\infty}
^\infty\delta^{-ksq}\|Q_kf\|_{L^p(\cx)}
^q\right\}^{1/q}\\
&\lesssim\left[\sum_{j=0}
^\infty\delta^{j\gamma q}
\sum_{k=-\infty}^\infty\delta^{-ksq}\left
\{\int_{\cx}
\left[J_{j-k}(f;x)\right]^p\,d\mu(x)\right\}
^{q/p}\right]^{1/q}\notag\\
&\lesssim\|f\|_{\lipsb};\notag
\end{align}
while, when $q\in(1,\infty]$, by \eqref{qfiii} and
the Minkowski inequality, we conclude that
\begin{align*}
\|f\|_{\hb}\lesssim\sum_{j=0}
^\infty\delta^{j\gamma}
\left[\sum_{k=-\infty}^\infty\delta^{-ksq}
\left\{\int_{\cx}
\left[J_{j-k}(f;x)\right]^p\,d\mu(x)\right\}
^{q/p}\right]^{1/q}\lesssim\|f\|_{\lipsb}.
\end{align*}
This, together with \eqref{blb1}, \eqref{blb2},
and \eqref{blb3}, then finishes the proof of (i).

Next, we prove (ii). By Proposition
\ref{pro-lipii}(iii), we have, for any given $s,\ p$,
and $q$ as in Proposition \ref{lipinbtl}(ii),
$$\lipst \subset \dot{L}_b(s,p,\max\{p,q\};\cx).$$
Thus, from \eqref{lipb-d}, we deduce that
$\lipst\subset(\cgg)'$
with $\beta$, $\gamma$, $s$, $p$, and $q$ as in (ii).

Let $\{Q_k\}_{k=-\infty}^\infty$ be an exp-ATI
and $f\in\lipst$. Then,
by the cancellation and the size conditions of
$Q_k$, we find that, for any fixed
$\Gamma\in(0,\infty)$ and for any $k\in\zz$ and
$x\in\cx$,
\begin{align}\label{qfiv}
|Q_kf(x)|&=\left|\int_{\cx}[f(x)-f(y)]Q_k(x,y)
\,d\mu(y)\right|
\lesssim \int_{\cx}R_\Gamma(x,y;k)|f(x)-f(y)|\,d\mu(y)\\
&\lesssim \sum_{j=0}^\infty\delta^{j\Gamma}
J_{j-k}(f;x),\notag
\end{align}
where $R_\Gamma(x,y;k)$ is as in \eqref{rxy}.
If $q\in(\omega/[\omega+
(\beta\wedge\gamma)],1]$, using \eqref{r} and  \eqref{qfiv},
and choosing $\Gamma\in(\gamma,\infty)$,
we conclude that, for any $k\in\zz$ and
$x\in\cx$,
\begin{equation}\label{qf6}
|Q_kf(x)|\lesssim \left\{\sum_{j=0}
^\infty\delta^{j\gamma q}
\left[J_{j-k}(f;x)\right]^q\right\}^{1/q};
\end{equation}
while, when $q\in(1,\infty]$, using the H\"older inequality and
choosing $\Gamma\in(\gamma,\infty)$ again,
we have, for any $k\in\zz$ and $x\in\cx$,
\begin{align}\label{qf7}
|Q_kf(x)|\lesssim \left\{\sum_{j=0}
^\infty\delta^{j\gamma q}
\left[J_{j-k}(f;x)\right]^q\right\}^{1/q}
\left[\sum_{j=0}^\infty\delta^{j(\Gamma-
\gamma)q'}\right]^{1/q'}
\lesssim \left\{\sum_{j=0}
^\infty\delta^{j\gamma q}
\left[J_{j-k}(f;x)\right]^q\right\}^{1/q}.
\end{align}
Combining the above two estimates,
we conclude that,
for any given $q\in(\omega/[\omega+
(\beta\wedge\gamma)],\infty]$
and for any $k\in\zz$ and $x\in\cx$,
\begin{equation*}
|Q_kf(x)|\lesssim \left\{\sum_{j=0}
^\infty\delta^{j\gamma q}
\left[J_{j-k}(f;x)\right]^q\right\}^{1/q}.
\end{equation*}

Now, we consider the following two cases.

{\it Case 1)} $p\in(\omega/[\omega+
(\beta\wedge\gamma)],\infty)$.
In this case, if $p/q\in(0,1]$, then, by \eqref{qf6},
\eqref{qf7}, \eqref{r}, and $s\in(0,\gamma)$, we obtain
\begin{align}\label{tllip1}
\|f\|_{\hf}&=\left\|\left[\sum_{k=-\infty}
^\infty\delta^{-ksq}|Q_kf|^q\right]^{1/q}\right\|
_{L^p(\cx)}\\
&\lesssim \left(\int_{\cx}\left\{\sum_{k=-
\infty}^\infty\delta^{-ksq}\sum_{j=0}
^\infty\delta^{j\gamma q}
\left[J_{j-k}(f;x)
\right]^q\right\}^{p/q}\,d\mu(x)\right)^{1/p}\notag\\
&\lesssim \left(\int_{\cx}\sum_{j=0}
^\infty\delta^{j\gamma p}\left\{\sum_{k=-
\infty}^\infty\delta^{-ksq}
\left[J_{j-k}(f;x)
\right]^q\right\}^{p/q}\,d\mu(x)\right)^{1/p}\notag\\
&\sim \left(\sum_{j=0}^\infty\delta^{j\gamma
p}\left\|\left\{\sum_{k=-\infty}^\infty\delta^{-ksq}
\left[J_{j-k}(f;\cdot)
\right]^q\right\}^{1/q}\right\|_{L^p(\cx)}
^p\right)^{1/p}\notag\\
&\lesssim \|f\|_{\lipst};\notag
\end{align}
while, if $p/q\in(1,\infty)$, by \eqref{qf6},
\eqref{qf7}, the Minkowski inequality,
and an argument similar to that used in the
estimation of \eqref{tllip1}, we have
\begin{align}\label{tllip2}
\|f\|_{\hf}&\lesssim \left(\sum_{j=0}
^\infty\delta^{j\gamma q}\left\|
\left\{\sum_{k=-\infty}^\infty\delta^{-ksq}
\left[J_{j-k}(f;\cdot)
\right]^q\right\}^{1/q}\right\|_{L^p(\cx)}
^q\right)^{1/q}
\lesssim \|f\|_{\lipst},
\end{align}
which, combined with \eqref{tllip1}, then completes
the proof of (ii) in this case.

{\it Case 2)} $p=\infty$. In this case,
by \eqref{qf6} and \eqref{qf7},
we find that, for any $l\in\zz$ and $x\in\cx$,
\begin{align}\label{tlinfty}
\sum_{k=l}^\infty\delta^{-ksq}|Q_kf(x)|^q
\lesssim \sum_{j=0}^\infty\delta^{j\gamma q}
\sum_{k=j}^\infty\delta^{-ksq}\left[J_{j-k}(f;x)\right]^q
\lesssim\|f\|_{\dot{L}_t(s,\infty,q;\cx)}^q,
\end{align}
which implies that, for any $l\in\zz$ and
$\alpha\in\ca_{l}$,
$$\frac{1}{\mu(Q_\alpha^l)}\int_{Q_\alpha^l}
\sum_{k=l}^\infty\delta^{-ksq}|Q_kf(x)|
^q\,d\mu(x)\lesssim \|f\|_{\dot{L}_t(s,\infty,q;\cx)}^q.$$
Thus, $f\in\hfi$ and $\|f\|_{\hfi}\lesssim\|f\|
_{\dot{L}_t(s,\infty,q;\cx)}$, which
completes the proof of (ii) in this case.

Combining with the two above cases, we complete
the proof of (ii) and hence of Proposition \ref{lipinbtl}.
\end{proof}

To establish the contrary of Proposition \ref{lipinbtl},
we need the following notion of lower bounds;
see, for instance, \cite[Definition 1.1]{hhhlp20} and \cite{agh20}.

\begin{definition}\label{lower}
Suppose that $(\cx, d, \mu)$ is a space of
homogeneous type with upper dimension
$\omega$ as in \eqref{eq-doub}.
The measure $\mu$ is said to have a
\emph{lower bound} $Q$ with $Q\in(0,\omega]$, if
there is a positive constant $C$ such that, for
any $x\in\cx$ and $r\in(0,\infty)$,
\begin{equation}\label{lb}
\mu(B(x,r))\geq Cr^Q.
\end{equation}
\end{definition}

\begin{remark}\label{lowerx}
\begin{itemize}
\item[(i)] We point out that, in Definition \ref{lower},
the restriction $Q\in(0,\omega]$ is necessary.
Indeed, fix a point $x_0\in\cx$. Then,
by \eqref{eq-doub} and \eqref{lb},
we find that, for any
$r\in[1,\infty)$,
$$
r^Q\lesssim \mu(B(x_0,r))\lesssim r^\omega\mu(B(x_0,1)),
$$
which implies that $r^{Q-\omega}\lesssim 1$.
Letting $r\to\infty$, we then obtain $Q\leq \omega$,
which explains the reasonability of the above restriction on $Q$.
\item[(ii)] Observe that the lower bound $Q$ of $\cx$ in
\cite[Definition 1.1]{hhhlp20} is directly restricted to the same as
the upper dimension $\omega$ of $\cx$.
\end{itemize}
\end{remark}

Now, we establish the inverse of Proposition \ref{lipinbtl}.

\begin{proposition}\label{btlinlip}
Let $\eta$ be as in Definition \ref{10.23.2}, $
\omega$ as in \eqref{eq-doub},
$\beta,\ \gamma\in(0,\eta)$, and
$s\in(0,\beta\wedge\gamma)$.
\begin{enumerate}
\item[{\rm(i)}] If $p\in[1,\infty]$
and $q\in(0,\infty]$, then $\hb\subset\lips$;
\item[{\rm(ii)}] If $p\in(\omega/(\omega+s),1)$
satisfies $-\eta<s-\omega/p$,
$q\in(0,\infty]$, and $\cx$ has a lower bound $\omega$,
then $\hb\subset\lips$;
\item[{\rm(iii)}] If $p\in(1,\infty)$ and
$q\in(1,\infty]$,
then $\hf\subset\lipst$;
\item[{\rm(iv)}] If $p\in(\omega/(\omega+s),1]$
satisfies $-\eta<s-\omega/p$,
$q\in(\omega/[\omega+
(\beta\wedge\gamma)],\infty]$, and $\cx$ has
a lower bound $\omega$,
then $\hf\subset\lipstu$ with $u\in (0,\min\{p,q\})$.
\end{enumerate}
\end{proposition}

To show Proposition \ref{btlinlip}, we need
several technical lemmas.
Han et al. \cite[Theorem 1.3]{hhhlp20} obtained
a necessary and sufficient condition for the
validity of the embedding theorems for Besov and
Triebel--Lizorkin spaces on spaces of homogeneous type.
The following lemma comes from the
combination of \cite[Theorem 1.3]{hhhlp20} and
\cite[Theorem 3.2]{hwyy20}.

\begin{lemma}\label{embed}
Let $\eta$ be as in Definition \ref{10.23.2},
$\omega$ as in \eqref{eq-doub},
$\beta,\ \gamma\in(0,\eta)$, $s\in(0,
\beta\wedge\gamma)$, and $p\in (\omega/
(\omega+s),1]$ satisfy $-\eta<s-\omega/p$.
Assume that $\cx$ has a lower bound
$\omega$.
\begin{enumerate}
\item[{\rm (i)}] If $q\in (0,\infty]$, then
$$\dot{B}_{p,q}^s(\cx)\subset \dot{B}_{1,q}^{s-
\omega(\frac{1}{p}-1)}(\cx).$$
\item[{\rm (ii)}] If $q\in (\omega/[\omega+
(\beta\wedge\gamma)],\infty]$, then
$$\dot{F}_{p,q}^s(\cx)\subset \dot{F}_{1,q}^{s-
\omega(\frac{1}{p}-1)}(\cx).$$
\end{enumerate}
\end{lemma}

\begin{remark} We point that the lower bound of $\cx$
in Proposition \ref{btlinlip} and
Lemma \ref{embed} is required to be
the same as the upper dimension $\omega$. It is still
unclear whether or not the conclusions
of Proposition \ref{btlinlip} and Lemma \ref{embed}
still hold true if the lower bound $Q$ of $\cx$ belongs to $(0,\omega)$.
Indeed, we prove Proposition \ref{btlinlip}
by using Lemma \ref{embed}, while
Lemma \ref{embed} strongly depends on
\cite[Proposition 3.1]{hhhlp20}
which needs $Q=\omega$. But, it is easy
to check that \cite[Proposition 3.1]{hhhlp20}
still holds true when $Q\in[\omega,\infty)$,
which, together with Remark \ref{lowerx}(i),
results in the restriction of Lemma \ref{embed} and hence Proposition \ref{btlinlip}
on the lower bound of $\cx$.
\end{remark}

The following two technical lemmas play a very
important role, respectively, in
dealing with Besov and Triebel--Lizorkin
spaces on spaces of homogeneous type.
Lemma \ref{9.14.1} comes from \cite[Lemma 3.13]{whhy}.

\begin{lemma}\label{9.14.1}
Let $\gamma \in (0,\infty)$ and $p \in
(\omega/(\omega+\gamma),1]$ with $\omega$ as in
\eqref{eq-doub}. Then there exists a positive
constant $C$ such that,
for any $k,\ k' \in \zz$, $x \in \mathcal{X}$,
and  $\ya \in \qa$ with
$\alpha \in \ca_k$ and $m\in\{1,\dots,N(k,\alpha)\}$,
\begin{align*}
C^{-1}[V_{\delta^{k\wedge k'}}(x)]^{1-p}
&\le\sum_{\alpha \in \ca_k}\sum_{m=1}^{N(k,
\alpha)}\mu\left(\qa\right)
\left[\frac{1}{V_{\delta^{k\wedge k'}}(x)+V(x,\ya)}\right]^p
\left[\frac{\delta^{k\wedge k'}}
{\delta^{k\wedge k'}
+d(x,\ya)}\right]^{\gamma p}\\
& \leq C[V_{\delta^{k\wedge k'}}(x)]^{1-p}.\noz
\end{align*}
\end{lemma}

The following lemma is just \cite [Lemma 5.3]
{hmy08} whose proof remains valid for any
quasi-metric $d$ and
is independent of the reverse doubling property
of the considered measure; we omit the details here.

\begin{lemma}\label{10.18.5}
Let $\gamma \in (0,\infty)$ and $r \in
(\omega/(\omega+\gamma),1]$ with $\omega$ as in
\eqref{eq-doub}. Then there exists a positive
constant $C$ such that,
for any $k,\ k' \in \zz$, $x \in \mathcal{X}$,
and $a_\alpha^{k,m}\in\cc$
and $\ya \in \qa$ with $\alpha \in \ca_k$ and
$m\in\{1,\dots,N(k,\alpha)\}$,
\begin{align*}
&\sum_{\alpha \in \ca_k}\sum_{m=1}^{N(k,
\alpha)}\mu\left(\qa\right)
\frac{1}{V_{\delta^{k\wedge k'}}(x)+V(x,\ya)}
\left[\frac{\delta^{k\wedge k'}}
{\delta^{k\wedge k'}+d(x,\ya)}\right]^{\gamma}|
a_\alpha^{k,m}|\\
&\quad \leq C \delta^{[k-(k\wedge k')]
\omega(1-1/r)}\left[M\left(\sum_{\alpha \in \ca_k}
\sum_{m=1}^{N(k,\alpha)}|a_\alpha^{k,m}|
^r\mathbf 1_{\qa}\right)(x)\right]^{1/r},\notag
\end{align*}
where $M$ is as in \eqref{m}.
\end{lemma}

The following lemma is the Fefferman--Stein
vector-valued maximal inequality  established
in \cite[Theorem 1.2]{gly09}.
\begin{lemma}\label{fsvv}
Let $p\in(1,\infty)$, $q\in(1,\infty]$, and $M$
be the Hardy--Littlewood maximal operator on
$\mathcal{X}$ as
in \eqref{m}. Then there exists a positive
constant $C$ such that, for any sequence
$\{f_j\}_{j\in\zz}$ of
measurable functions on $\mathcal{X}$,
$$
\left\|\left\{\sum_{j\in\zz}[M(f_j)]^q\right\}^{1/
q}\right\|_{L^p(\mathcal{X})}
\leq C\left\|\left(\sum_{j\in\zz}|f_j|^q\right)^{1/
q}\right\|_{L^p(\mathcal{X})}
$$
with the usual modification made when $q = \infty$.
\end{lemma}

Now, we prove Proposition \ref{btlinlip}.

\begin{proof}[Proof of Proposition \ref{btlinlip}]
To prove (i), by \cite[Theorem 3.3 and Lemma
3.16]{awyy}, we have, for any $f\in\hb$ with
$s$, $p$, and $q$ as in (i), $f\in
L_\loc^p(\cx)$.
Let $\{Q_k\}_{k=-\infty}^\infty$ be an exp-ATI
and fix $x_1\in\cx$. Since $f\in \dot{B}_{p,q}^s(\cx)$,
from \cite[Propositions 3.13(i) and 3.15]{whhy},
it follows that $f\in\dot{B}_{p,\infty}^s(\cx)$
and
$\|f\|_{\dot{B}_{p,\infty}^s(\cx)}
\lesssim\|f\|_{\hb}$, and, moreover,
$f\in (\cgg)'$ with $\eta,
\ \beta$, and $\gamma$ as in this proposition.
Thus, for any $g\in\cgg$,
by $\int_{\cx}g(x)d\mu(x)=0$ and Lemma
\ref{crf}, we obtain
\begin{align*}
\langle f,g\rangle&= \sum_{k=-\infty}
^{l_0-1}\sum_{\alpha \in \ca_k}
\sum_{m=1}^{N(k,\alpha)}
\mu\left(\qa\right)Q_kf\left(\ya\right)
\left\langle\widetilde{Q}_k\left(\cdot,
\ya\right)-\widetilde{Q}
_k\left(x_1,\ya\right),g\right\rangle\\
&\qquad +\sum_{k=l_0}^\infty\sum_{\alpha \in \ca_k}
\sum_{m=1}^{N(k,\alpha)}
\mu\left(\qa\right)Q_kf\left(\ya\right)
\left\langle\widetilde{Q}_k\left(\cdot,
\ya\right),g\right\rangle.
\end{align*}
Define, formally, for any $l_0\in\zz$ and $x\in
B(x_1,\delta^{l_0})$,
\begin{align}\label{de-f}
\widetilde{f}(x)&:=\sum_{k=-\infty}
^{l_0-1}\sum_{\alpha \in \ca_k}
\sum_{m=1}^{N(k,\alpha)}
\mu\left(\qa\right)Q_kf\left(\ya\right)
\left[\widetilde{Q}_k\left(x,\ya\right)-
\widetilde{Q}_k\left(x_1,\ya\right)\right]\\
&\qquad +\sum_{k=l_0}^\infty\sum_{\alpha \in \ca_k}
\sum_{m=1}^{N(k,\alpha)}
\mu\left(\qa\right)Q_kf\left(\ya\right)
\widetilde{Q}_k\left(x,\ya\right)\notag\\
&=:{\rm J}_1(x)+{\rm J}_2(x).\notag
\end{align}

We next show that $\widetilde{f}$ is well
defined. We first consider the case
$p\in[1,\infty)$.
In this case, we show that, for any $l_0\in\zz$,
\begin{equation}\label{floc}
\left[\int_{B(x_1,\delta^{l_0})}\left|\widetilde{f}
(x)\right|^p\,d\mu(x)\right]^{1/p}
\lesssim \delta^{l_0s}\|f\|_{\dot{B}_{p,\infty}^s(\cx)}.
\end{equation}

We first estimate ${\rm J}_1$.
By \eqref{4.23y}, we have, for any
$x\in B(x_1,\delta^{l_0})$,
\begin{align}\label{j1}
|{\rm J}_1(x)|&\lesssim \sum_{k=-\infty}
^{l_0-1}\sum_{\alpha \in \ca_k}
\sum_{m=1}^{N(k,\alpha)}\mu\left(\qa\right)
\left|Q_kf\left(\ya\right)\right|
\left[\frac{d(x,x_1)}
{\delta^k+d(x_1,\ya)}\right]^\beta
R_\gamma(x_1,\ya;k),
\end{align}
where $R_\gamma(x_1,\ya;k)$ is as in \eqref{rxy}.
From this, \eqref{r}, Lemma \ref{6.15.1}(i),
the H\"older inequality, Lemma \ref{9.14.1},
$s\in(0,\beta\wedge\gamma)$,
and $\mu(\qa)\lesssim V_{\delta^k}(\ya)$, we
deduce that, for any $x\in B(x_1,\delta^{l_0})$,
\begin{align*}
|{\rm J}_1(x)|^p&\lesssim [d(x,x_1)]^{\beta p}
\left\{\sum_{k=-\infty}^{l_0-1}
\sum_{\alpha \in \ca_k}
\sum_{m=1}^{N(k,\alpha)}\delta^{-k\beta}
\mu\left(\qa\right)\left|Q_kf\left(\ya\right)
\right|
R_\gamma\left(x_1,\ya;k\right)\right\}^p\\
&\lesssim [d(x,x_1)]^{\beta p}\left\{\sum_{k=-
\infty}^{l_0-1}
\left[\sum_{\alpha \in \ca_k}
\sum_{m=1}^{N(k,\alpha)}\delta^{-k\beta p}
\mu\left(\qa\right)\left|Q_kf\left(\ya\right)
\right|^p
R_\gamma\left(x_1,\ya;k\right)\right]^{1/p}\right.\notag\\
&\qquad\times\left.\left[\sum_{\alpha \in \ca_k}
\sum_{m=1}^{N(k,\alpha)}
\mu\left(\qa\right)
R_\gamma\left(x_1,\ya;k\right)
\right]^{1/p'}\right\}^p\notag\\
&\lesssim \frac{[d(x,x_1)]^{\beta p}}
{V_{\delta^{l_0}}(x_1)}
\left\{\sum_{k=-\infty}^{l_0-1}\delta^{-k\beta}
\left[\sum_{\alpha \in \ca_k}
\sum_{m=1}^{N(k,\alpha)}\mu\left(\qa\right)
\left|Q_kf\left(\ya\right)\right|^p\right]^{1/p}
\right\}^p\notag\\
&\lesssim \|f\|_{\dot{B}_{p,\infty}^s(\cx)}^p
\frac{[d(x,x_1)]^{\beta p}}{V_{\delta^{l_0}}(x_1)}
\delta^{-l_0(\beta-s)p},\notag
\end{align*}
which further implies that, for any
$p\in[1,\infty)$ and $l_0\in\zz$,
\begin{align}\label{j1loc}
\int_{B(x_1,\delta^{l_0})}|{\rm J}_1(x)|
^p\,d\mu(x)
&\lesssim \|f\|_{\dot{B}_{p,\infty}^s(\cx)}^p
\frac{\delta^{-l_0(\beta-s)p}}{V_{\delta^{l_0}}(x_1)}
\int_{B(x_1,\delta^{l_0})}[d(x,x_1)]^{\beta p}\,d\mu(x)\\
&\lesssim \delta^{l_0sp}\|f\|_{\dot{B}_{p,\infty}^s(\cx)}^p.\noz
\end{align}

Now, we estimate ${\rm J}_2$. By
\eqref{4.23x}, we find that, for any
$x\in B(x_1,\delta^{l_0})$,
\begin{equation}\label{j2}
|{\rm J}_2(x)|\lesssim \sum_{k=l_0}^{\infty}
\sum_{\alpha \in \ca_k}
\sum_{m=1}^{N(k,\alpha)}\mu\left(\qa\right)
\left|Q_kf\left(\ya\right)\right|
R_\gamma\left(x,\ya;k\right).
\end{equation}
Thus, using the H\"older inequality and Lemma \ref{9.14.1},
we conclude that, for any $x\in B(x_1,
\delta^{l_0})$,
\begin{align*}
|{\rm J}_2(x)|&\lesssim \sum_{k=l_0}^\infty
\left\{\sum_{\alpha \in \ca_k}
\sum_{m=1}^{N(k,\alpha)}\mu\left(\qa\right)
\left|Q_kf\left(\ya\right)\right|^p
R_\gamma\left(x,\ya;k\right)\right\}^{1/p}\\
&\qquad\times\left\{\sum_{\alpha \in \ca_k}
\sum_{m=1}^{N(k,\alpha)}\mu\left(\qa\right)
R_\gamma\left(x,\ya;k\right)\right\}^{1/p'}\\
&\lesssim \sum_{k=l_0}^\infty
\left\{\sum_{\alpha \in \ca_k}
\sum_{m=1}^{N(k,\alpha)}\mu\left(\qa\right)
\left|Q_kf\left(\ya\right)\right|^p
R_\gamma\left(x,\ya;k\right)\right\}^{1/p}.
\end{align*}
From this, the Minkowski inequality, and
\eqref{9.14.1x}, we deduce that
\begin{align}\label{j2locii}
&\left[\int_{B(x_1,\delta^{l_0})}|{\rm J}_2(x)|
^p\,d\mu(x)\right]^{1/p}\\
&\qquad\lesssim \sum_{k=l_0}^\infty
\left\{\sum_{\alpha \in \ca_k}
\sum_{m=1}^{N(k,\alpha)}\mu\left(\qa\right)
\left|Q_kf\left(\ya\right)\right|^p\int_{\cx}
R_\gamma\left(x,\ya;k\right)\,d\mu(x)\right\}^{1/p}\notag\\
&\qquad\lesssim \sum_{k=l_0}^\infty
\left\{\sum_{\alpha \in \ca_k}
\sum_{m=1}^{N(k,\alpha)}\mu\left(\qa\right)
\left|Q_kf\left(\ya\right)\right|^p\right\}^{1/p}
\lesssim\delta^{l_0s}\|f\|_{\dot{B}_{p,\infty}
^s(\cx)}.\notag
\end{align}
By \eqref{j1loc} and \eqref{j2locii}, we know
that \eqref{floc} holds true and the infinite summation
in \eqref{de-f} converges in $L_{\loc}^p(\cx)$
when $p\in[1,\infty)$. Moreover, $\widetilde{f}$
is well defined when $p\in[1,\infty)$.

Next, we consider the case $p=\infty$.
Indeed, in this case, from \eqref{j1}, Lemma \ref{9.14.1},
$s\in(0,\beta\wedge\gamma)$,
and the arbitrariness of $\ya$,
we deduce that, for any $l_0\in\zz$ and
$x\in B(x_1, \delta^{l_0})$,
\begin{align}\label{j1i}
|{\rm J}_1(x)|&\lesssim \delta^{l_0\beta}
\sum_{k=-\infty}^{l_0-1}\delta^{-k\beta}\|
Q_kf\|_{L^\infty(\cx)}\sum_{\alpha\in\ca_k}
\sum_{m=1}^{N(k,\alpha)}\mu\left(\qa\right)
R_\gamma\left(x_1,\ya;k\right)\\
&\lesssim \delta^{l_0\beta}\|f\|_{\dot{B}
^s_{\infty,\infty}(\cx)}\sum_{k=-\infty}
^{l_0-1}\delta^{-k(\beta-s)}\lesssim
\delta^{l_0s}\|f\|_{\dot{B}^s_{\infty,\infty}(\cx)}.\notag
\end{align}
Besides, by \eqref{j2}, Lemma \ref{9.14.1},
$s\in(0,\beta\wedge\gamma)$,
and the arbitrariness of $\ya$,
we conclude that, for any $l_0\in\zz$ and
$x\in B(x_1, \delta^{l_0})$,
\begin{align}\label{j2i}
|{\rm J}_2(x)|\lesssim \sum_{k=l_0}^\infty\|
Q_k\|_{L^\infty(\cx)}
\lesssim \|f\|_{\dot{B}^s_{\infty,\infty}(\cx)}
\sum_{k=l_0}^\infty\delta^{ks}
\lesssim \delta^{l_0s}\|f\|_{\dot{B}^s_{\infty,\infty}(\cx)}.
\end{align}
From this and \eqref{j1i}, we deduce
that \eqref{floc} holds true and the infinite
summation in \eqref{de-f}
converges in $L^\infty(B(x_1,\delta^{l_0}))$.
Moreover, $\widetilde{f}$ is well defined when $p=\infty$.

Now, we show that, for any $l_0\in\zz$, there
exists a constant $C_{(l_0)}$,
depending on $l_0$, such that, for almost every
$x\in B(x_1,\delta^{l_0})$,
$\widetilde{f}(x)-C_{(l_0)}=f(x)$. Indeed, for any
$n\in\nn$, let
\begin{align*}
f_n(\cdot)&:=\sum_{k=l_0-n}
^{l_0-1}\sum_{\alpha \in \ca_k}
\sum_{m=1}^{N(k,\alpha)}
\mu\left(\qa\right)Q_kf\left(\ya\right)
\left[\widetilde{Q}_k\left(\cdot,\ya\right)-
\widetilde{Q}_k\left(x_1,\ya\right)\right]\\
&\qquad +\sum_{k=l_0}^{l_0+n-1}\sum_{\alpha \in \ca_k}
\sum_{m=1}^{N(k,\alpha)}
\mu\left(\qa\right)Q_kf\left(\ya\right)
\widetilde{Q}_k\left(\cdot,\ya\right).
\end{align*}
Then we know that $f_n\to f$ in $(\cgg)'$
as $n\to\infty$, and,
by \eqref{j1loc}, \eqref{j2locii}, \eqref{j1i},
\eqref{j2i}, and the dominated convergence
theorem, we find that
$f_n\to\widetilde{f}$
in $L^p(B(x_1,\delta^{l_0}))$ as $n\to\infty$.
Thus, for any given $\varphi\in\cgg$ with
$\supp \varphi \subset B(x_1,\delta^{l_0})$
and $\beta,\ \gamma\in(0,\eta)$, we obtain
\begin{align*}
\int_{B(x_1,\delta^{l_0})}f(x)\varphi(x)\,
d\mu(x)&=\langle f,\varphi\rangle
=\lim_{n\to\infty} \langle f_n,\varphi\rangle
=\lim_{n\to\infty}\int_{B(x_1,\delta^{l_0})}f_n(x)\varphi(x)\, d\mu(x)\\
&= \int_{B(x_1,\delta^{l_0})}\widetilde{f}(x)\varphi(x)\, d\mu(x),
\end{align*}
which implies that there exists a constant
$C_{(l_0)}$, depending on $l_0$,
such that, for almost every
$x\in B(x_1,\delta^{l_0})$,
$\widetilde{f}(x)-C_{(l_0)}=f(x)$.
This proves the above claim.

Next, we show that, for any $f\in\hb$ with
$s,\ p$, and $q$ as in (i),
\begin{equation}\label{lipbesov}
\|f\|_{\lips}\lesssim\|f\|_{\hb}.
\end{equation}
We first consider the case
$p\in [1,\infty)$. Indeed, there
exists a unique $l_0\in\zz$
such that
$\widetilde{C}\in(\delta^{l_0+1},\delta^{l_0}]$,
where $\widetilde{C}$ is as in Definition \ref{lipi}(i).
Then, for any $l\in\zz$, $x\in\cx$, and
$y\in B(x,\widetilde{C}\delta^l)$, we write
\begin{align*}
f(x)-f(y)&=\sum_{k=-\infty}
^{l+l_0-1}\sum_{\alpha \in \ca_k}
\sum_{m=1}^{N(k,\alpha)}
\mu\left(\qa\right)Q_kf\left(\ya\right)
\left[\widetilde{Q}_k\left(x,\ya\right)-
\widetilde{Q}_k\left(y,\ya\right)\right]\\
&\qquad +\sum_{k=l+l_0}^\infty\sum_{\alpha \in \ca_k}
\sum_{m=1}^{N(k,\alpha)}\cdots\\
&=:{\rm J}_3(x,y)+{\rm J}_4(x,y).
\end{align*}
We first estimate ${\rm J}_3$. Since
$d(x,y)<\widetilde{C}\delta^l\leq\delta^{l+l_0}
<(2A_0)^{-1}[\delta^k+d(x,\ya)]$
for any $k\in(-\infty,l+l_0-1]\cap \zz$, from
the regularity condition of $\widetilde{Q}_k$,
it then follows that
\begin{align}\label{j3}
|{\rm J}_3(x,y)|&\lesssim \sum_{k=-\infty}
^{l+l_0-1}\sum_{\alpha \in \ca_k}
\sum_{m=1}^{N(k,\alpha)}\mu\left(\qa\right)
\left|Q_kf\left(\ya\right)\right|\left[\frac{d(x,y)}{\delta^k+d(x,
\ya)}\right]^\beta
R_\gamma\left(x,\ya;k\right),
\end{align}
where $R_\gamma\left(x,\ya;k\right)$ is as in \eqref{rxy}.
By this, \eqref{4.23y}, the H\"older inequality,
and Lemma \ref{9.14.1}, we have
\begin{align*}
|{\rm J}_3(x,y)|&\lesssim \sum_{k=-\infty}
^{l+l_0-1}\sum_{\alpha \in \ca_k}
\sum_{m=1}^{N(k,\alpha)}\mu\left(\qa\right)
\left|Q_kf\left(\ya\right)\right|\left[\frac{d(x,y)}{\delta^k+d(x,
\ya)}\right]^\beta
R_\gamma\left(x,\ya;k\right)\\
&\lesssim \delta^{l\beta}\sum_{k=-\infty}
^{l+l_0-1}\delta^{-k\beta}\left[\sum_{\alpha \in \ca_k}
\sum_{m=1}^{N(k,\alpha)}\mu\left(\qa\right)
\left|Q_kf\left(\ya\right)\right|^pR_{\beta+\gamma}\left(x,\ya;k\right)\right]^{1/p}.\notag
\end{align*}
From this, Lemma \ref{6.15.1}(ii), and the Minkowski inequality,
we deduce that, for any $l\in\zz$ and $x\in\cx$,
\begin{align}\label{j3iii}
&\left[\int_{\cx}\frac{1}{\mu(B(x,\widetilde{C}
\delta^l))}\int_{B(x,\widetilde{C}\delta^l)}
|{\rm J}_3(x,y)|^p\,d\mu(y)\,d\mu(x)\right]^{1/p}\\
&\qquad\lesssim \delta^{l\beta}\left\{\int_{\cx}
\left[\sum_{k=-\infty}^{l+l_0-1}\delta^{-k\beta}
\left\{\sum_{\alpha \in \ca_k}
\sum_{m=1}^{N(k,\alpha)}\mu\left(\qa\right)
\left|Q_kf\left(\ya\right)\right|^p\right.\right.
\right.\notag\\
&\qquad\qquad\times\left.\left.\left.
\vpz{\sum_{k=-\infty}^{l+l_0-1}}
R_{\beta+\gamma}\left(x,\ya;k\right)\right\}^{1/p}
\right]^p\,d\mu(x)\right\}^{1/p}\notag\\
&\qquad \lesssim \delta^{l\beta}\sum_{k=-
\infty}^{l+l_0-1}
\delta^{-k\beta}\left\{\sum_{\alpha \in \ca_k}
\sum_{m=1}^{N(k,\alpha)}\mu\left(\qa\right)
\left|Q_kf\left(\ya\right)\right|^p\int_{\cx}
R_{\beta+\gamma}\left(x,\ya;k\right)\,d\mu(x)
\right\}^{1/p}\notag\\
&\qquad \lesssim \delta^{l\beta}\sum_{k=-
\infty}^{l+l_0-1}
\delta^{-k\beta}\left\{\sum_{\alpha \in \ca_k}
\sum_{m=1}^{N(k,\alpha)}\mu\left(\qa\right)
\left|Q_kf\left(\ya\right)\right|^p\right\}^{1/p}\notag\\
&\qquad \lesssim\delta^{l\beta'}
\left\{\sum_{k=-\infty}^{l+l_0-1}\delta^{-k\beta'q}
\left[\sum_{\alpha \in \ca_k}
\sum_{m=1}^{N(k,\alpha)}\mu\left(\qa\right)
\left|Q_kf\left(\ya\right)\right|^p\right]^{q/p}
\right\}^{1/q},\notag
\end{align}
where $\beta':=\beta$ when $q\in(0,1]$
by using \eqref{r},
or $\beta'\in(s,\beta)$ when $q\in(1,\infty]$
by using the H\"older inequality.
By \eqref{j3iii}, we conclude that, if
$p\in[1,\infty)$, then
\begin{align}\label{j3besoviii}
&\left\{\sum_{l=-\infty}^\infty\delta^{-lsq}
\left[\int_{\cx}\frac{1}{\mu(B(x,\widetilde{C}
\delta^l))}\int_{B(x,\widetilde{C}\delta^l)}
|{\rm J}_3(x,y)|^p\,d\mu(y)\,d\mu(x)\right]^{q/
p}\right\}^{1/q}\lesssim \|f\|_{\hb}.
\end{align}

Next, we estimate ${\rm J}_4$.
Note that, by Lemma \ref{6.15.1}(i), we obtain,
for any given $p\in(0,\infty)$
and for any $k,\ l\in\zz$, $\alpha\in\ca_k$,
and $x,\ y\in\cx$,
\begin{align}\label{qxqy}
&\int_{\cx}\frac{1}{\mu(B(x,\widetilde{C}
\delta^l))}\int_{B(x,\widetilde{C}\delta^l)}
\left|\widetilde{Q}_{k}(y,\ya)\right|^p\,d\mu(y)
\,d\mu(x)\\
&\qquad=\int_{\cx}\int_{\cx}\frac{1}{\mu(B(x,
\widetilde{C}\delta^l))}
\left|\widetilde{Q}_{k}(y,\ya)\right|^p\mathbf{1}
_{\{y\in\cx:\,d(x,y)<\widetilde{C}\delta^l\}}(y)
\,d\mu(y)\,d\mu(x)\notag\\
&\qquad =\int_{\cx}\left|\widetilde{Q}_{k}(y,
\ya)\right|^p\left[\int_{\cx}\frac{1}{\mu(B(x,
\widetilde{C}\delta^l))}
\mathbf{1}_{\{x\in\cx:\,d(x,y)<\widetilde{C}
\delta^l\}}(x)\,d\mu(x)\right]\,d\mu(y)\notag\\
&\qquad\lesssim \int_{\cx}\left|\widetilde{Q}
_{k}(y,\ya)\right|^p\,d\mu(y).\notag
\end{align}
From \eqref{4.23x}, the H\"older inequality, and
Lemma \ref{9.14.1},
we deduce that
\begin{align*}
|{\rm J}_4(x,y)|&\lesssim \sum_{k=l+l_0}
^\infty\sum_{\alpha \in \ca_k}
\sum_{m=1}^{N(k,\alpha)}\mu\left(\qa\right)
\left|Q_kf\left(\ya\right)\right|
\left[\left|\widetilde{Q}_k\left(x,\ya\right)\right|
+\left|\widetilde{Q}_k\left(y,\ya\right)\right|\right]\\
&\lesssim\left\{\sum_{k=l+l_0}
^\infty\sum_{\alpha \in \ca_k}
\sum_{m=1}^{N(k,\alpha)}\delta^{-ksp/
2}\mu\left(\qa\right)
\left|Q_kf\left(\ya\right)\right|^p
\left[\left|\widetilde{Q}_k\left(x,\ya\right)\right|
+\left|\widetilde{Q}_k\left(y,\ya\right)\right|
\right]\right\}^{1/p}\\
&\qquad\times\left\{\sum_{k=l+l_0}
^\infty\sum_{\alpha \in \ca_k}
\sum_{m=1}^{N(k,\alpha)}\delta^{ksp'/
2}\mu\left(\qa\right)
\left[\left|\widetilde{Q}_k\left(x,\ya\right)\right|
+\left|\widetilde{Q}_k\left(y,\ya\right)\right|
\right]\right\}^{1/p'}\\
&\lesssim\delta^{ls/2}\left\{\sum_{k=l+l_0}
^\infty\sum_{\alpha \in \ca_k}
\sum_{m=1}^{N(k,\alpha)}\delta^{-ksp/
2}\mu\left(\qa\right)
\left|Q_kf\left(\ya\right)\right|^p
\left[\left|\widetilde{Q}_k\left(x,\ya\right)\right|
+\left|\widetilde{Q}_k\left(y,\ya\right)\right|
\right]\right\}^{1/p},
\end{align*}
which, together with \eqref{4.23x},
\eqref{qxqy}, and \eqref{9.14.1x}, further implies that
\begin{align}\label{j4ii}
&\delta^{-lsq}\left[\int_{\cx}\frac{1}{\mu(B(x,
\widetilde{C}\delta^l))}\int_{B(x,\widetilde{C}\delta^l)}
|{\rm J}_4(x,y)|^p\,d\mu(y)\,d\mu(x)\right]^{q/p}\\
&\qquad\lesssim \delta^{-lsq/2}\sum_{k=l+l_0}
^\infty\delta^{-ksp/2}\left\{\sum_{\alpha \in \ca_k}
\sum_{m=1}^{N(k,\alpha)}\mu\left(\qa\right)
\left|Q_kf\left(\ya\right)\right|^p\right\}^{q/p}.\notag
\end{align}
If $q/p\in(0,1]$, then, by \eqref{j4ii}, \eqref{r},
and the arbitrariness of $\ya$, we conclude that
\begin{align}\label{j4besoviii}
&\left\{\sum_{l=-\infty}^\infty\delta^{-lsq}
\left[\int_{\cx}\frac{1}{\mu(B(x,\widetilde{C}
\delta^l))}\int_{B(x,\widetilde{C}\delta^l)}
|{\rm J}_4(x,y)|^p\,d\mu(y)\,d\mu(x)\right]^{q/
p}\right\}^{1/q}\\
&\qquad \lesssim \left\{\sum_{l=-\infty}^\infty
\sum_{k=l+l_0}^\infty\delta^{(k-l)sq/
2}\delta^{-ksq}\left[\sum_{\alpha \in \ca_k}
\sum_{m=1}^{N(k,\alpha)}\mu\left(\qa\right)
\left|Q_kf\left(\ya\right)\right|^p\right]^{q/p}
\right\}^{1/q}\notag\\
&\qquad\lesssim\|f\|_{\hb};\notag
\end{align}
while, if $q/p\in(1,\infty]$, by \eqref{j4ii}, the
H\"older inequality, and the arbitrariness of
$\ya$, we obtain
\begin{align*}
&\left\{\sum_{l=-\infty}^\infty\delta^{-lsq}
\left[\int_{\cx}\frac{1}{\mu(B(x,\widetilde{C}
\delta^l))}\int_{B(x,\widetilde{C}\delta^l)}
|{\rm J}_4(x,y)|^p\,d\mu(y)\,d\mu(x)\right]^{q/
p}\right\}^{1/q}\\
&\qquad \lesssim \left\{\sum_{l=-\infty}
^\infty\delta^{-lsq/2}
\sum_{k=l+l_0}^\infty\delta^{-ksq/
4}\left[\sum_{\alpha \in \ca_k}
\sum_{m=1}^{N(k,\alpha)}\mu\left(\qa\right)\left|Q_kf
\left(\ya\right)\right|^p\right]^{q/p}\right.\notag\\
&\qquad\qquad\times\left.\left[\sum_{k=l+l_0}
^\infty\delta^{-ksp(q/p)'/4}\right]^{\frac{1}{(q/
p)'}}\right\}^{1/q}\notag\\
&\qquad\lesssim\|f\|_{\hb},\notag
\end{align*}
which, combined with \eqref{j3besoviii},
\eqref{j4besoviii}, then
completes the proof of \eqref{lipbesov} when $p\in[1,\infty)$.
The proof of the case $p=\infty$ is
similar to that of the case
$p\in [1,\infty)$; we omit the details here.
This finishes the proof of  (i).

Now, we show (ii). To this end, by
\cite[Theorem 3.3 and Lemma 3.16]{awyy}, we
know that, for any $f\in\hb$ with
$s$, $p$, and $q$ as in (ii),
$f\in L_\loc^1(\cx)$.
We claim that \eqref{floc} still holds true for
any $l_0\in\zz$, and $s$ and $p$ as in (ii).
Indeed, if $p\in(\omega/(\omega+s),1]$,
from \eqref{j1}, \eqref{r}, Lemma \ref{6.15.1}(i),
$s\in(0,\beta\wedge\gamma)$,
and $\mu(\qa)\lesssim V_{\delta^k}(\ya)$,
we deduce that
\begin{align*}
|{\rm J}_1(x)|^p&\lesssim \sum_{k=-\infty}
^{l_0-1}\sum_{\alpha \in \ca_k}
\sum_{m=1}^{N(k,\alpha)}
\left[\mu\left(\qa\right)\left|
Q_kf\left(\ya\right)\right|\right]^p\\
&\qquad\times\left[\frac{d(x,x_1)}
{\delta^k+d(x_1,\ya)}\right]^{\beta p}
\left[\frac{1}{V_{\delta^k}(x_1)+V_{\delta^k}(\ya)
+V(x_1,\ya)}\right]^p\notag\\
&\lesssim \frac{[d(x,x_1)]^{\beta p}}
{V_{\delta^{l_0}}(x_1)}
\sum_{k=-\infty}^{l_0-1}\delta^{-k\beta p}
\left[\sum_{\alpha \in \ca_k}
\sum_{m=1}^{N(k,\alpha)}\mu\left(\qa\right)
\left|Q_kf\left(\ya\right)\right|^p\right]\notag\\
&\lesssim \|f\|_{\dot{B}_{p,\infty}^s(\cx)}^p
\frac{[d(x,x_1)]^{\beta p}}{V_{\delta^{l_0}}(x_1)}
\delta^{-l_0(\beta-s)p}.\notag
\end{align*}
By this, we further conclude that,
for any $l_0\in\zz$ and
$p\in(\omega/(\omega+1),1)$,
\begin{align}\label{j1loci}
\int_{B(x_1,\delta^{l_0})}|{\rm J}_1(x)|
^p\,d\mu(x)
&\lesssim \|f\|_{\dot{B}_{p,\infty}^s(\cx)}^p
\frac{\delta^{-l_0(\beta-s)p}}{V_{\delta^{l_0}}(x_1)}
\int_{B(x_1,\delta^{l_0})}[d(x,x_1)]^{\beta p}\,d\mu(x)\\
&\lesssim \delta^{l_0sp}\|f\|_{\dot{B}_{p,\infty}^s(\cx)}^p.\notag
\end{align}
On the other hand, from Lemmas \ref{6.15.1}(i)
and \ref{9.14.1},
we deduce that, for any
$k\in\{l_0,l_0+1,\ldots\}$,
\begin{equation}\label{9.14.1x}
\int_{\cx}\left[R_\gamma\left(x,\ya;k\right)\right]^p\,d\mu(x)
\lesssim \left[\mu(\qa)\right]^{1-p},
\end{equation}
which, together with \eqref{j2} and \eqref{r},
implies that,
for any $p\in(\omega/
(\omega+s),1)$ and $l_0\in\zz$,
\begin{align}\label{j2loci}
&\int_{B(x_1,\delta^{l_0})}|{\rm J}_2(x)|
^p\,d\mu(x)\\
&\qquad\lesssim\int_{\cx}\sum_{k=l_0}^{\infty}
\sum_{\alpha \in \ca_k}
\sum_{m=1}^{N(k,\alpha)}
\left[\mu\left(\qa\right)\right]^p
\left|Q_kf\left(\ya\right)\right|
^p\left[R_\gamma\left(x,\ya;k\right)\right]^p
\,d\mu(x)\notag\\
&\qquad\lesssim \sum_{k=l_0}^{\infty}\sum_{\alpha \in \ca_k}
\sum_{m=1}^{N(k,\alpha)}
\left[\mu\left(\qa\right)\right]^p
\left|Q_kf\left(\ya\right)\right|^p\int_{\cx}
\left[R_\gamma\left(x,\ya;k\right)\right]^p
\,d\mu(x)\notag\\
&\qquad\lesssim \sum_{k=l_0}^{\infty}\sum_{\alpha \in \ca_k}
\sum_{m=1}^{N(k,\alpha)}\mu\left(\qa\right)
\left|Q_kf\left(\ya\right)\right|^p
\lesssim \delta^{l_0sp}\|f\|_{\dot{B}_{p,\infty}
^s(\cx)}^p.\notag
\end{align}
By \eqref{j1loci} and \eqref{j2loci}, we know
that \eqref{floc} still holds true and the infinite summation
in \eqref{de-f} converges in $L_{\loc}^p(\cx)$.
Moreover, $\widetilde{f}$ is well defined when $s$ and $p$ are as in (ii).

Next, we claim that, for any $l_0\in\zz$, there
exists a constant $C_{(l_0)}$,
depending on $l_0$, such that, for almost every
$x\in B(x_1,\delta^{l_0})$,
$\widetilde{f}(x)-C_{(l_0)}=f(x)$. Indeed, by
Lemma \ref{embed}, we know that
$$\dot{B}_{p,q}^s(\cx)\subset \dot{B}_{1,q}^{s-
\omega(\frac{1}{p}-1)}(\cx)\subset \dot{B}
_{1,\infty}^{s-\omega(\frac{1}{p}-1)}(\cx).$$
Since $p\in(\omega/(\omega+s),1)$, it follows
that $s-\omega(\frac{1}{p}-1)>0$.
Using an argument similar to that used in the estimations of
\eqref{j1loc} and \eqref{j2locii}, we also know
that $\widetilde{f}\in L_{\loc}^1(\cx)$.
From this and an argument similar to that used
in the case $p\in[1,\infty)$, we deduce that,
for any $l_0\in\zz$, there exists a constant
$C_{(l_0)}$, depending on $l_0$,
such that, for almost every $x\in
B(x_1,\delta^{l_0})$,
$\widetilde{f}(x)-C_{(l_0)}=f(x)$.

Now, we show that, for any $f\in\hb$ with
$s,\ p$, and $q$ as in (ii), \eqref{lipbesov} still holds true.
Indeed, by \eqref{j3}, \eqref{r}, and
\eqref{9.14.1x}, we find that
\begin{align}\label{j3i}
&\left\{\sum_{l=-\infty}^\infty\delta^{-lsq}
\left[\int_{\cx}\frac{1}{\mu(B(x,\widetilde{C}
\delta^l))}\int_{B(x,\widetilde{C}\delta^l)}
|{\rm J}_3(x,y)|^p\,d\mu(y)\,d\mu(x)\right]^{q/
p}\right\}^{1/q}\\
&\qquad \lesssim \left\{\sum_{l=-\infty}
^\infty\delta^{l(\beta-s)q}
\left[\sum_{k=-\infty}^{l+l_0-1}\delta^{-k\beta
p}\sum_{\alpha \in \ca_k}
\sum_{m=1}^{N(k,\alpha)}\mu\left(\qa\right)
\left|Q_kf\left(\ya\right)\right|^p\right]^{q/p}
\right\}^{1/q}.\notag
\end{align}
If $q/p\in(0,1]$, by \eqref{j3i}, \eqref{r},
and the arbitrariness of $\ya$, we have
\begin{align}\label{j3besovi}
&\left\{\sum_{l=-\infty}^\infty\delta^{-lsq}
\left[\int_{\cx}\frac{1}{\mu(B(x,\widetilde{C}
\delta^l))}\int_{B(x,\widetilde{C}\delta^l)}
|{\rm J}_3(x,y)|^p\,d\mu(y)\,d\mu(x)\right]^{q/
p}\right\}^{1/q}\lesssim \|f\|_{\hb};
\end{align}
while, if $q/p\in(1,\infty]$, by \eqref{j3},
the H\"older inequality,
$s\in (0,\beta)$,  and the arbitrariness of
$\ya$, we obtain
\begin{align}\label{j3besovii}
&\left\{\sum_{l=-\infty}^\infty\delta^{-lsq}
\left[\int_{\cx}\frac{1}{\mu(B(x,\widetilde{C}
\delta^l))}\int_{B(x,\widetilde{C}\delta^l)}
|{\rm J}_3(x,y)|^p\,d\mu(y)\,d\mu(x)\right]^{q/
p}\right\}^{1/q}\\
&\qquad\lesssim \left\{\sum_{l=-\infty}
^\infty\delta^{l(\beta-s)q}
\left[\sum_{k=-\infty}^{l+l_0-1}\delta^{-k\beta
p}\sum_{\alpha \in \ca_k}
\sum_{m=1}^{N(k,\alpha)}\mu\left(\qa\right)
\left|Q_kf\left(\ya\right)\right|^p\right]^{q/p}
\right\}^{1/q}\notag\\
&\qquad \lesssim \left(\sum_{l=-\infty}
^\infty\delta^{l(\beta-s)q}
\left\{\sum_{k=-\infty}^{l+l_0-1}\delta^{-
k(\beta+s)q/2}\left[\sum_{\alpha \in \ca_k}
\sum_{m=1}^{N(k,\alpha)}\mu\left(\qa\right)
\left|Q_kf\left(\ya\right)\right|^p\right]^{q/p}
\right\}\right.\notag\\
&\qquad\qquad\times\left.\left\{\sum_{k=-
\infty}^{l+l_0-1}
\delta^{-k(\beta-s)p(q/p)'/2}\right\}^{\frac{q}
{(q/p)'p}}\right)^{1/q}\notag\\
&\qquad\lesssim \left\{\sum_{k=-\infty}
^\infty\delta^{-ksq}\sum_{l=k-l_0+1}
^\infty\delta^{(l-k)(\beta-s)q/
2}\left[\sum_{\alpha \in \ca_k}
\sum_{m=1}^{N(k,\alpha)}\mu\left(\qa\right)
\left|Q_kf\left(\ya\right)\right|^p\right]^{q/p}
\right\}^{1/q}\notag\\
&\qquad\lesssim\|f\|_{\hb}.\notag
\end{align}
On the other hand,
by \eqref{r}, \eqref{qxqy}, \eqref{4.23x},  and
\eqref{9.14.1x}, we find that
\begin{align}\label{j4i}
&\left[\int_{\cx}\frac{1}{\mu(B(x,\widetilde{C}
\delta^l))}\int_{B(x,\widetilde{C}\delta^l)}
|{\rm J}_4(x,y)|^p\,d\mu(y)\,d\mu(x)\right]^{1/p}\\
&\qquad \lesssim\left\{\sum_{k=l+l_0}
^\infty\sum_{\alpha \in \ca_k}
\sum_{m=1}^{N(k,\alpha)}
\left[\mu\left(\qa\right)\left|
Q_kf\left(\ya\right)\right|\right]^p\right.\notag\\
&\qquad\qquad\times\left.\vpz{\sum_{k=-\infty}^{l+l_0-1}}
\int_{\cx}\frac{1}
{\mu(B(x,\widetilde{C}\delta^l))}
\int_{B(x,\widetilde{C}\delta^l)}\left[\left|
\widetilde{Q}_k\left(x,\ya\right)\right|^p
+\left|\widetilde{Q}_k\left(y,\ya\right)\right|
^p\right]\,d\mu(y)\,d\mu(x)\right\}^{1/p}\notag\\
&\qquad\lesssim \left\{\sum_{k=l+l_0}
^\infty\sum_{\alpha \in \ca_k}
\sum_{m=1}^{N(k,\alpha)}
\left[\mu\left(\qa\right)
\left|Q_kf\left(\ya\right)\right|\right]^p
\int_{\cx}\left|\widetilde{Q}_k\left(x,\ya\right)
\right|^p\,d\mu(x)\right\}^{1/p}\notag\\
&\qquad\lesssim \left\{\sum_{k=l+l_0}
^\infty\sum_{\alpha \in \ca_k}
\sum_{m=1}^{N(k,\alpha)}\mu\left(\qa\right)
\left|Q_kf\left(\ya\right)\right|^p\right\}^{1/p}.\notag
\end{align}
If $q/p\in(0,1]$, by \eqref{r} and \eqref{j4i},
we conclude that
\begin{align}\label{j4besovi}
&\left\{\sum_{l=-\infty}^\infty\delta^{-lsq}
\left[\int_{\cx}\frac{1}{\mu(B(x,\widetilde{C}
\delta^l))}\int_{B(x,\widetilde{C}\delta^l)}
|{\rm J}_4(x,y)|^p\,d\mu(y)\,d\mu(x)\right]^{q/
p}\right\}^{1/q}\lesssim \|f\|_{\hb};
\end{align}
while, if $q/p\in(1,\infty]$, by \eqref{j4i} and
the H\"older inequality, we find that
\begin{align*}
&\left[\int_{\cx}\frac{1}{\mu(B(x,\widetilde{C}
\delta^l))}\int_{B(x,\widetilde{C}\delta^l)}
|{\rm J}_4(x,y)|^p\,d\mu(y)\,d\mu(x)\right]^{q/p}\\
&\qquad\lesssim \delta^{lsq/2}\sum_{k=l+l_0}
^\infty\delta^{-ksq/2}\left\{\sum_{\alpha \in \ca_k}
\sum_{m=1}^{N(k,\alpha)}\mu\left(\qa\right)
\left|Q_kf\left(\ya\right)\right|^p\right\}^{q/p},
\end{align*}
which further implies that
\begin{align}\label{j4besovii}
&\left\{\sum_{l=-\infty}^\infty\delta^{-lsq}
\left[\int_{\cx}\frac{1}{\mu(B(x,\widetilde{C}
\delta^l))}\int_{B(x,\widetilde{C}\delta^l)}
|{\rm J}_4(x,y)|^p\,d\mu(y)\,d\mu(x)\right]^{q/
p}\right\}^{1/q}\lesssim \|f\|_{\hb}.
\end{align}
This, combined with \eqref{j3besovi}, \eqref{j3besovii},
and \eqref{j4besovi}, shows that
\eqref{lipbesov} still holds true for any given
$s$, $p$, and $q$ as in (ii), and hence finishes
the proof of (ii).

Next, we show (iii). Let $\{Q_k\}_{k=-\infty}
^\infty$ be an exp-ATI and $f\in\hf$ with $p\in(1,\infty)$.
Note that there exists a unique $l_0\in\zz$
such that
$\widetilde{C}\in(\delta^{l_0+1},\delta^{l_0}]$,
where $\widetilde{C}$ is as in Definition \ref{lipi}(i).
Using an argument similar to that used in the
estimation of \eqref{floc},
we know that, for any $s,\ p,$ and $q$ as in (ii),
$\hf\subset L_{\loc}^p(\cx)$,
and,  for any $l\in\zz$, $x\in\cx$ and $y\in
B(x,\widetilde{C}\delta^l)$,  we write
\begin{align*}
f(x)-f(y)&=\sum_{k=-\infty}
^{l+l_0-1}\sum_{\alpha \in \ca_k}
\sum_{m=1}^{N(k,\alpha)}
\mu\left(\qa\right)Q_kf\left(\ya\right)
\left[\widetilde{Q}_k\left(x,\ya\right)-
\widetilde{Q}_k\left(y,\ya\right)\right]\\
&\qquad +\sum_{k=l+l_0}^\infty\sum_{\alpha \in \ca_k}
\sum_{m=1}^{N(k,\alpha)}\cdots\\
&=:{\rm J}_5(x,y)+{\rm J}_6(x,y).
\end{align*}

We first estimate ${\rm J}_5$. By
\eqref{4.23y}, we have
\begin{align*}
|{\rm J}_5(x,y)|&\lesssim \sum_{k=-\infty}
^{l+l_0-1}\sum_{\alpha \in \ca_k}
\sum_{m=1}^{N(k,\alpha)}\mu\left(\qa\right)
\left|Q_kf\left(\ya\right)\right|
\left[\frac{d(x,y)}{\delta^k+d(x,
\ya)}\right]^\beta
R_\gamma\left(x,\ya;k\right)\\
&\lesssim \sum_{k=-\infty}^{l+l_0-1}\delta^{(l-
k)\beta}\sum_{\alpha \in \ca_k}
\sum_{m=1}^{N(k,\alpha)}\mu\left(\qa\right)
\left|Q_kf\left(\ya\right)\right|
R_\gamma\left(x,\ya;k\right),
\end{align*}
where $R_\gamma(x,\ya;k)$ is as in \eqref{rxy}.
From this, Lemma \ref{10.18.5}, and the
H\"older inequality, we deduce that
\begin{align*}
&\left\{\sum_{l=-\infty}^\infty\delta^{-lsq}
\left[\frac{1}{\mu(B(x,\widetilde{C}\delta^l))}
\int_{B(x,\widetilde{C}\delta^l)}
|{\rm J}_5(x,y)|\,d\mu(y)\right]^{q}\right\}^{1/q}\\
&\qquad\lesssim \left[\sum_{k=-\infty}
^\infty\delta^{-ksq}
\left\{M\left(\sum_{\alpha\in\ca_k}\sum_{m=1}^{N(k,\alpha)}
\left|Q_kf\left(\ya\right)\mathbf{1}_{\qa}\right|
\right)(x)\right\}^q\right]^{1/q}
\end{align*}
with $M$ as in \eqref{m}, which, together with
Lemma \ref{fsvv} and the arbitrariness of
$\ya$, further implies that
\begin{align}\label{j5}
&\left\|\left\{\sum_{l=-\infty}^\infty\delta^{-lsq}
\left[\frac{1}{\mu(B(\cdot,\widetilde{C}
\delta^l))}\int_{B(\cdot,\widetilde{C}\delta^l)}
|{\rm J}_5(\cdot,y)|\,d\mu(y)\right]^{q}\right\}^{1/q}
\right\|_{L^p(\cx)}\\
&\qquad\lesssim \left\|\left[\sum_{k=-\infty}^\infty\delta^{-ksq}
\left\{M\left(\sum_{\alpha\in\ca_k}\sum_{m=1}^{N(k,\alpha)}
\left|Q_kf\left(\ya\right)\mathbf{1}_{\qa}\right|
\right]\right\}^q\right)^{1/q}\right\|_{L^p(\cx)}\notag\\
&\qquad\lesssim\|f\|_{\hf}.\notag
\end{align}

Now, we estimate ${\rm J}_6$. By Lemma
\ref{10.18.5}, we have
\begin{align*}
|{\rm J}_6(x,y)|&\lesssim\sum_{k=l+l_0}
^\infty\sum_{\alpha \in \ca_k}
\sum_{m=1}^{N(k,\alpha)}\mu\left(\qa\right)
\left|Q_kf\left(\ya\right)\right|
\left[\left|\widetilde{Q}_k\left(x,\ya\right)\right|
+\left|\widetilde{Q}_k\left(y,\ya\right)\right|\right]\\
&\lesssim \sum_{k=l+l_0}^\infty\left\{
M\left(\sum_{\alpha\in\ca_k}\sum_{m=1}^{N(k,\alpha)}
\left|Q_kf\left(\ya\right)\right|\mathbf{1}_{\qa}\right)(x)\right.\\
&\qquad\left.+M\left(\sum_{\alpha\in\ca_k}
\sum_{m=1}^{N(k,\alpha)}
\left|Q_kf\left(\ya\right)\right|\mathbf{1}_{\qa}
\right)(y)\right\}.
\end{align*}
From this, the H\"older inequality, and the
Lebesgue differential theorem, we deduce that,
for almost every $x\in\cx$,
\begin{align*}
&\left\{\sum_{l=-\infty}^\infty\delta^{-lsq}
\left[\frac{1}{\mu(B(x,\widetilde{C}\delta^l))}
\int_{B(x,\widetilde{C}\delta^l)}
|{\rm J}_6(x,y)|\,d\mu(y)\right]^{q}\right\}^{1/q}\\
&\qquad\lesssim \left\{\sum_{l=-\infty}
^\infty\delta^{-lsq}
\left[\sum_{k=l+l_0}^\infty
\left\{M\left(\sum_{\alpha\in\ca_k}\sum_{m=1}
^{N(k,\alpha)}
\left|Q_kf\left(\ya\right)\right|\mathbf{1}_{\qa}
\right)(x)\right.\right.\right.\\
&\qquad\qquad\qquad\left.\left.\left.+M\circ
M\left(\sum_{\alpha\in\ca_k}\sum_{m=1}^{N(k,\alpha)}
\left|Q_kf\left(\ya\right)\right|\mathbf{1}_{\qa}
\right)(x)\right\}\right]^q\right\}^{1/q}\\
&\qquad \lesssim\left[\sum_{k=-\infty}
^\infty\delta^{-ksq}
\left\{M\circ M\left(\sum_{\alpha\in\ca_k}
\sum_{m=1}^{N(k,\alpha)}
\left|Q_kf\left(\ya\right)\mathbf{1}_{\qa}\right|
\right)(x)\right\}^q\right]^{1/q},
\end{align*}
where, for any $f\in L^1_{\loc}(\cx)$, $M\circ
M(f):=M(M(f))$.
This, combined with Lemma \ref{fsvv} and the
arbitrariness of $\ya$, implies that
\begin{equation*}
\left\|\left\{\sum_{l=-\infty}^\infty\delta^{-lsq}
\left[\frac{1}{\mu(B(\cdot,\widetilde{C}
\delta^l))}\int_{B(\cdot,\widetilde{C}\delta^l)}
|{\rm J}_6(\cdot,y)|\,d\mu(y)\right]^{q}\right\}^{1/q}
\right\|_{L^p(\cx)}\lesssim\|f\|_{\hf},
\end{equation*}
which, together with \eqref{j5},
then completes the proof of (iii).

Finally, we prove (iv). Let $\{Q_k\}_{k=-\infty}
^\infty$ be an exp-ATI and $f\in\hf$ with
$s,\ p$, and $q$ as in (iii).
Note that there exists a unique $l_0\in\zz$
such that
$\widetilde{C}\in(\delta^{l_0+1},\delta^{l_0}]$,
where $\widetilde{C}$ is as in Definition
\ref{lipii}(i). Using an argument similar to that
used in the estimation of \eqref{floc},
we know that $f\in L^1_{\loc}(\cx)$ and, for
any $l\in\zz$, $x\in\cx$,
and $y\in B(x,\widetilde{C}\delta^l)$, we  write
\begin{align*}
f(x)-f(y)&=\sum_{k=-\infty}
^{l+l_0-1}\sum_{\alpha \in \ca_k}
\sum_{m=1}^{N(k,\alpha)}
\mu\left(\qa\right)Q_kf\left(\ya\right)
\left[\widetilde{Q}_k\left(x,\ya\right)-
\widetilde{Q}_k\left(y,\ya\right)\right]\\
&\qquad +\sum_{k=l+l_0}^\infty\sum_{\alpha \in \ca_k}
\sum_{m=1}^{N(k,\alpha)}\cdots\\
&=:{\rm J}_7(x,y)+{\rm J}_8(x,y).
\end{align*}

We first estimate ${\rm J}_7$. By \eqref{4.23y}, we have
\begin{align*}
|{\rm J}_7(x,y)|&\lesssim \sum_{k=-\infty}
^{l+l_0-1}\sum_{\alpha \in \ca_k}
\sum_{m=1}^{N(k,\alpha)}\mu\left(\qa\right)
\left|Q_kf\left(\ya\right)\right|
\left[\frac{d(x,y)}{\delta^k+d(x,
\ya)}\right]^\beta
R_\gamma\left(x,\ya;k\right)\\
&\lesssim \sum_{k=-\infty}^{l+l_0-1}\delta^{(l-
k)\beta}\sum_{\alpha \in \ca_k}
\sum_{m=1}^{N(k,\alpha)}\mu\left(\qa\right)
\left|Q_kf\left(\ya\right)\right|R_\gamma\left(x,\ya;k\right),
\end{align*}
where $R_\gamma(x,\ya;k)$ is as in \eqref{rxy}.
From this and Lemma \ref{10.18.5}, we deduce
that, for any $r\in(\omega/(\omega+
\gamma),1]$ and $x\in\cx$,
\begin{align*}
&\left\{\sum_{l=-\infty}^\infty\delta^{-lsq}
\left[\frac{1}{\mu(B(x,\widetilde{C}\delta^l))}
\int_{B(x,\widetilde{C}\delta^l)}
|{\rm J}_7(x,y)|^u\,d\mu(y)\right]^{\frac{q}{u}}
\right\}^{1/q}\\
&\qquad\lesssim \left[\sum_{k=-\infty}
^\infty\delta^{-ksq}
\left\{M\left[\sum_{\alpha\in\ca_k}\sum_{m=1}
^{N(k,\alpha)}
\left|Q_kf\left(\ya\right)\right|^r\mathbf{1}
_{\qa}\right](x)\right\}^{\frac{q}{r}}\right]^{1/q},
\end{align*}
where we used \eqref{r} when $q\in (0,1]$, or
the H\"older inequality
when $q\in (1, \infty]$. This, combined with
Lemma \ref{fsvv}
and the arbitrariness of $\ya$, implies that
\begin{align}\label{j7}
&\left\|\left\{\sum_{l=-\infty}^\infty\delta^{-lsq}
\left[\frac{1}{\mu(B(\cdot,\widetilde{C}
\delta^l))}\int_{B(\cdot,\widetilde{C}\delta^l)}
|{\rm J}_7(\cdot,y)|^u\,d\mu(y)\right]^{\frac{q}
{u}}\right\}^{1/q}
\right\|_{L^p(\cx)}\\
&\qquad\lesssim \left\|\left[\sum_{k=-\infty}
^\infty\delta^{-ksq}
\left\{M\left(\sum_{\alpha\in\ca_k}\sum_{m=1}
^{N(k,\alpha)}
\left|Q_kf\left(\ya\right)\right|^r\mathbf{1}
_{\qa}\right)\right\}^{\frac{q}{r}}\right]^{1/q}
\right\|_{L^p(\cx)}\notag\\
&\qquad\lesssim \left\|\left[\sum_{k=-\infty}
^\infty\delta^{-ksq}
\left\{M\left(\sum_{\alpha\in\ca_k}\sum_{m=1}
^{N(k,\alpha)}
\left|Q_kf\left(\ya\right)\right|^r\mathbf{1}
_{\qa}\right)\right\}^{\frac{q}{r}}
\right]^{\frac{r}{q}}\right\|_{L^{\frac{p}{r}}(\cx)}
^{\frac{1}{r}}\notag\\
&\qquad\lesssim \left\|\left\{\sum_{k=-\infty}
^\infty\delta^{-ksq}
\left[\sum_{\alpha\in\ca_k}\sum_{m=1}^{N(k,
\alpha)}
\left|Q_kf\left(\ya\right)\right|^r\mathbf{1}
_{\qa}\right]^{\frac{q}{r}}\right\}^{\frac{r}{q}}
\right\|_{L^{\frac{p}{r}}(\cx)}^{\frac{1}{r}}\notag\\
&\qquad\sim \left\|\left[\sum_{k=-\infty}
^\infty\delta^{-ksq}
\sum_{\alpha\in\ca_k}\sum_{m=1}^{N(k,\alpha)}
\left|Q_kf\left(\ya\right)\right|^q\mathbf{1}
_{\qa}\right]^{\frac{1}{q}}\right\|_{L^{p}(\cx)}
\lesssim\|f\|_{\hf},\notag
\end{align}
where we chose $r\in (\omega/(\omega+
\gamma), \min\{p,q\})$.

Now, we estimate ${\rm J}_8$. By the size
condition of $\widetilde{Q}_k$ and Lemma
\ref{10.18.5}, we have, for any $r\in (\omega/
(\omega+\gamma), 1]$,
\begin{align*}
|{\rm J}_8(x,y)|&\lesssim\sum_{k=l+l_0}
^\infty\sum_{\alpha \in \ca_k}
\sum_{m=1}^{N(k,\alpha)}\mu\left(\qa\right)
\left|Q_kf\left(\ya\right)\right|
\left[\left|\widetilde{Q}_k\left(x,\ya\right)\right|
+\left|\widetilde{Q}_k\left(y,\ya\right)\right|\right]\\
&\lesssim \sum_{k=l+l_0}^\infty
\left\{M\left(\sum_{\alpha\in\ca_k}\sum_{m=1}
^{N(k,\alpha)}
\left|Q_kf\left(\ya\right)\right|^r\mathbf{1}
_{\qa}\right)(x)\right\}^{\frac{1}{r}}\\
&\qquad+\sum_{k=l+l_0}^\infty
\left\{M\left(\sum_{\alpha\in\ca_k}\sum_{m=1}
^{N(k,\alpha)}
\left|Q_kf\left(\ya\right)\right|^r\mathbf{1}
_{\qa}\right)(y)\right\}^{\frac{1}{r}}\\
&=: {\rm J}_{8,1}(x)+{\rm J}_{8,2}(y).
\end{align*}

Using an argument similar to that used in the
estimation of \eqref{j7}, we conclude that
\begin{align}\label{j81}
&\left\|\left\{\sum_{l=-\infty}^\infty\delta^{-lsq}
\left[\frac{1}{\mu(B(\cdot,\widetilde{C}
\delta^l))}\int_{B(\cdot,\widetilde{C}\delta^l)}
|{\rm J}_{8,1}(\cdot)|^u\,d\mu(y)\right]^{\frac{q}
{u}}\right\}^{1/q}
\right\|_{L^p(\cx)}\lesssim\|f\|_{\hf}.
\end{align}

To estimate ${\rm J}_{8,2}$, since
$u\in(0,\min\{p,q\})$, it follows that
\begin{align*}
&\left\{\sum_{l=-\infty}^\infty\delta^{-lsq}
\left[\frac{1}{\mu(B(x,\widetilde{C}\delta^l))}
\int_{B(x,\widetilde{C}\delta^l)}
|{\rm J}_{8,2}(y)|^u\,d\mu(y)\right]^{\frac{q}{u}}
\right\}^{1/q}\\
&\qquad\lesssim \left\{\sum_{l=-\infty}
^\infty\delta^{-lsq}
\left[M\left(\left\{\sum_{k=l+l_0}
^\infty\left[M\left(\sum_{\alpha\in\ca_k}
\sum_{m=1}^{N(k,\alpha)}
\left|Q_kf\left(\ya\right)\right|^r\mathbf{1}
_{\qa}\right)\right]^{\frac{1}{r}}
\right\}^u\right)(x)\right]^{\frac{q}{u}}\right\}^{1/q},
\end{align*}
which, together with Lemma \ref{fsvv}, implies that
\begin{align*}
&\left\|\left\{\sum_{l=-\infty}^\infty\delta^{-lsq}
\left[\frac{1}{\mu(B(\cdot,\widetilde{C}
\delta^l))}\int_{B(\cdot,\widetilde{C}\delta^l)}
|{\rm J}_{8,2}(y)|\,d\mu(y)\right]^{q}\right\}^{1/q}
\right\|_{L^p(\cx)}\\
&\qquad\lesssim \left\|\left\{\sum_{l=-\infty}^\infty\delta^{-lsq}
\left[M\left(\left\{\sum_{k=l+l_0}
^\infty\left[M\left(\sum_{\alpha\in\ca_k}
\sum_{m=1}^{N(k,\alpha)}
\left|Q_kf\left(\ya\right)\right|^r\mathbf{1}
_{\qa}\right)\right]^{\frac{1}{r}}
\right\}^u\right)\right]^{\frac{q}{u}}\right\}^{1/q}
\right\|_{L^p(\cx)}\\
&\qquad\sim \left\|\left\{\sum_{l=-\infty}
^\infty\delta^{-lsq}
\left[M\left(\left\{\sum_{k=l+l_0}
^\infty\left[M\left(\sum_{\alpha\in\ca_k}
\sum_{m=1}^{N(k,\alpha)}
\left|Q_kf\left(\ya\right)\right|^r\mathbf{1}
_{\qa}\right)\right]^{\frac{1}{r}}\right\}^u
\right)\right]^{\frac{q}{u}}\right\}^{\frac{u}{q}}
\right\|_{L^{\frac{p}{u}}(\cx)}^{\frac{1}{u}}\\
&\qquad\lesssim\left\|\left\{\sum_{l=-\infty}
^\infty\delta^{-lsq}
\left\{\sum_{k=l+l_0}
^\infty\left[M\left(\sum_{\alpha\in\ca_k}
\sum_{m=1}^{N(k,\alpha)}
\left|Q_kf\left(\ya\right)\right|^r\mathbf{1}
_{\qa}\right)\right]^{\frac{1}{r}}\right\}^q\right\}
^{\frac{1}{q}}
\right\|_{L^{p}(\cx)}.
\end{align*}
From this and an argument similar to that used
in the estimation of ${\rm J}_7$, we deduce that
\begin{align*}
\left\|\left\{\sum_{l=-\infty}^\infty\delta^{-lsq}
\left[\frac{1}{\mu(B(\cdot,\widetilde{C}
\delta^l))}\int_{B(\cdot,\widetilde{C}\delta^l)}
|{\rm J}_{8,2}(y)|\,d\mu(y)\right]^{q}\right\}^{1/q}
\right\|_{L^p(\cx)}\lesssim\|f\|_{\hf},
\end{align*}
where we chose $r\in(\omega/(\omega+
\gamma),\min\{p,q\})$.
Combining this, \eqref{j7}, and \eqref{j81}, we
then complete the proof of (iv) and hence of
Proposition \ref{btlinlip}.
\end{proof}

\begin{remark}
We point out that Proposition \ref{btlinlip}(iv) is new even when $\cx$ is an RD-spaces.
\end{remark}

Using Propositions \ref{pro-lipi}(i),
\ref{lipinbtl}, and \ref{btlinlip},
we obtain the following difference
characterization of homogeneous Besov
and Triebel--Lizorkin spaces, and we omit the details here.

\begin{theorem}\label{dc}
Let $\eta$ be as in Definition \ref{10.23.2},
$\beta,\ \gamma\in(0,\eta)$, and
$s\in(0,\beta\wedge\gamma)$.
\begin{enumerate}
\item[{\rm (i)}] If $p\in[1,\infty]$ and
$q\in(0,\infty]$, then
$\lips=\lipsb=\hb$
with equivalent (quasi-)norms.
\item[{\rm (ii)}] If $p\in(1,\infty)$ and
$q\in(1,\infty]$, then
$\lipst=\hf$
with equivalent norms.
\end{enumerate}
\end{theorem}

\section{Relations with inhomogeneous Besov
and Triebel--Lizorkin spaces}\label{s4}

In this section, we consider the relations
between spaces of Lipschitz-type
and inhomogeneous Besov spaces or Triebel--
Lizorkin spaces on spaces of homogenous type.
In this section, we do not need to assume that
$\mu(\cx)=\infty$. Let us begin with the following proposition.

\begin{proposition}\label{ilipinbtl}
Let $\eta$ be as in Definition \ref{10.23.2} and
$\omega$ as in \eqref{eq-doub},
$\beta,\ \gamma\in(0,\eta)$, and
$s\in(0,\beta\wedge\gamma)$.
\begin{enumerate}
\item[{\rm(i)}] If $p\in[1,\infty]$
and $q\in(0,\infty]$, then $\ilipsb\subset \ihb$;
\item[{\rm(ii)}] If $p\in[1,\infty]$ and
$q\in(\omega/[\omega+
(\beta\wedge\gamma)],\infty]$,
then $\ilipst\subset \ihf$.
\end{enumerate}
\end{proposition}

\begin{proof}
We first prove (i). By the definition of $\ilipsb$
with $s,\ p$, and  $q$ as in this proposition,
we know that $\ilipsb\subset L^p(\cx)$.
From this and $p\in[1,\infty]$, it follows that
$L^p(\cx)\subset (\icgg)'$ and hence
$\ilipsb\subset (\icgg)'$
for any given $\beta,\ \gamma\in(0,\eta)$.

Let $\{Q_k\}_{k=0}^\infty$ be an exp-IATI and $f\in\ilipsb$.
Note that, by \cite[Proposition 2.2(iii)]{hlyy}, we
know that, for any $p\in[1,\infty]$ and
$k\in\{0,\dots,N\}$,
\begin{equation}\label{q0}
\|Q_kf\|_{L^p(\cx)}\lesssim\|f\|_{L^p(\cx)}.
\end{equation}
From this, \eqref{qfi}, and the Minkowski
inequality, it follows that
\begin{align}\label{ibesov}
\|f\|_{\ihb}&=\left\{\sum_{k=0}^N\sum_{\alpha
\in \ca_k}\sum_{m=1}^{N(k,\alpha)}
\mu\left(\qa\right)\left[m_{\qa}\left(|Q_k(f)|
\right)\r]^p\right\}^{1/p}\\
&\qquad+\left[\sum_{k=N+1}^{\infty} \delta^{-ksq}
\|Q_k(f)\|_{\lp}^q\right]^{1/q}\notag\\
&\lesssim \|f\|_{L^p(\cx)}+\left\{\sum_{k=N+1}
^\infty\delta^{-ksq}
\left[\sum_{j=0}^\infty\delta^{j\gamma}\left\{\int_{\cx}
\left[J_{j-k}(f;x)\right]^p\,d\mu(x)\right\}^{1/p}
\right]^q\right\}^{1/q},\notag
\end{align}
where $J_{j-k}(f;x)$ is as in \eqref{jkf}.
If $q\in(0,1]$, by \eqref{ibesov} and \eqref{r}, we find that
\begin{align}\label{ibeslip}
\|f\|_{\ihb}&\lesssim \|f\|_{L^p(\cx)}+
\left[\sum_{k=N+1}^\infty\delta^{-ksq}
\sum_{j=0}^\infty\delta^{j\gamma q}
\left\{\int_{\cx}
\left[J_{j-k}(f;x)\right]^p\,d\mu(x)\right\}^{q/p}
\right]^{1/q}\\
&\lesssim \|f\|_{L^p(\cx)}+\left[
\sum_{j=0}^\infty\delta^{j\gamma q}
\sum_{k=N+1}^\infty\delta^{-ksq}\left\{\int_{\cx}
\left[J_{j-k}(f;x)\right]^p\,d\mu(x)\right\}^{q/p}
\right]^{1/q}\notag\\
&\lesssim \|f\|_{L^p(\cx)}+\|f\|_{\lipsb}\sim\|f\|
_{\ilipsb};\notag
\end{align}
while, if $q\in(1,\infty]$, by \eqref{ibesov} and
the Minkowski inequality, we obtain
\begin{align*}
\|f\|_{\ihb}&\lesssim \|f\|_{L^p(\cx)}+
\sum_{j=0}^\infty\delta^{j\gamma}
\left[\sum_{k=N+1}^\infty\delta^{-ksq}\left\{\int_{\cx}
\left[J_{j-k}(f;x)\right]^p\,d\mu(x)\right\}^{q/p}
\right]^{1/q}\notag\\
&\lesssim \|f\|_{L^p(\cx)}+\|f\|_{\lipsb}\sim\|f\|_{\ilipsb}.
\end{align*}
This, combined with \eqref{ibeslip},
then finishes the proof of (i).

Now, we prove (ii). Assume that $f\in\ilipst$
with $s,\ p$, and  $q$ as in this proposition.
By the definition of $\ilipst$, we know that
$\ilipst\subset L^p(\cx)$.
Since $p\in[1,\infty]$ and
$L^p(\cx)\subset(\icgg)'$, it follows that
$\ilipst\subset (\icgg)'$
for any given $\beta,\ \gamma\in(0,\eta)$.

Let $\{Q_k\}_{k=0}^\infty$ be an exp-IATI.
If $p\in[1,\infty)$, using \eqref{q0} and an
argument similar to
that used in the estimations of \eqref{tllip1} and
\eqref{tllip2}, we find that
\begin{align*}
\|f\|_{\ihf}&\lesssim\left\{\sum_{k=0}^N
\sum_{\alpha \in \ca_k}\sum_{m=1}^{N(k,
\alpha)}\mu\left(\qa\right)\left[m_{\qa}
\left(|Q_k(f)|\right)\r]^p\right\}^{1/p}+\left\|
\left\{\sum_{k=N+1}^\infty
\delta^{-ksq}|Q_kf|^q\right\}^{1/q}\right\|
_{L^p(\cx)}\\
&\lesssim \|f\|_{L^p(\cx)}+\|f\|_{\lipst}\sim \|f\|
_{\ilipst}.\notag
\end{align*}
If $p=\infty$, then, from \eqref{10.23.3} and \eqref{9.14.1x},
it follows that, for any $k\in\{0,\dots,N\}
$ and $x\in\cx$,
\begin{equation*}
|Q_kf(x)|=\left|\int_{\cx}Q_k(x,y)f(y)\,d\mu(y)
\right|\lesssim \|f\|_{L^\infty(\cx)}
\end{equation*}
and hence, for any $k\in\{0,\dots,N\}$,
$\alpha\in\ca_k$,
and $m\in \{1,\dots,N(k,\alpha)\}$,
\begin{equation}\label{tllinf}
m_{\qa}(|Q_kf|)\lesssim\|f\|_{L^\infty(\cx)}.
\end{equation}
Moreover, by an argument similar to that used in
\eqref{tlinfty}, we find that,  for any
$j\in\{N+1,N+2,\dots\}$ and $\alpha\in\ca_{j}$,
\begin{align*}
\frac{1}{\mu(Q_\alpha^j)}\int_{Q_\alpha^j}
\sum_{k=j}^\infty\delta^{-ksq}|Q_kf(x)|^q\,d\mu(x)
\lesssim\|f\|_{\dot{L}_t(s,\infty,q;\cx)}^q.\notag
\end{align*}
Using this and \eqref{tllinf}, we conclude that
$f\in\ihfi$ and
$$\|f\|_{\ihfi}\lesssim \|f\|_{L^\infty(\cx)} +\|f\|
_{\lipst}\sim\|f\|_{\ilipst},$$
which completes the proof of (ii) and hence of Proposition \ref{ilipinbtl}.
\end{proof}

To establish the converse of Proposition
\ref{ilipinbtl}, we need the following
notion of local lower bounds; see, for instance,
\cite[Definition 1.1]{hhhlp20}.

\begin{definition}\label{ilower}
Suppose that $(\cx, d, \mu)$ is a space of
homogeneous type
with upper dimension $\omega$ as in
\eqref{eq-doub}.
The measure $\mu$ is said to have a
\emph{local lower bound} $Q$
with $Q\in(0,\infty)$, if
there exists a positive constant $C$ such that,
for any $x\in\cx$ and $r\in(0,1]$,
$$\mu(B(x,r))\geq Cr^Q.$$
\end{definition}

\begin{remark}
\begin{enumerate}
\item[{\rm(i)}] Differently from the global lower
bound $Q$, which only makes sense
for any $Q\in(0,\omega]$,
of $\cx$ in Definition \ref{lower}, the local lower
bound $Q$ of $\cx$ indeed makes sense for
any $Q\in(0,\infty)$.
\item[{\rm(ii)}] Observe that the local lower bound $Q$ of $\cx$ in
\cite[Definition 1.1]{hhhlp20} is directly
required  to be the same as
in the upper dimension $\omega$ of $\cx$.
\end{enumerate}
\end{remark}

Next, we establish the converse of Proposition \ref{ilipinbtl}.

\begin{proposition}\label{ibtlinlip}
Let $\eta$ be as in Definition \ref{10.23.2},
$\omega$ as in \eqref{eq-doub},
$\beta,\ \gamma\in(0,\eta)$,
$s\in(0,\beta\wedge\gamma)$,
and $Q\in[\omega,\infty)$.
\begin{enumerate}
\item[{\rm(i)}] If $p\in[1,\infty]$
and $q\in(0,\infty]$, then $\ihb\subset\ilips$;
\item[{\rm(ii)}] If $\cx$ has a local lower
bound $Q$,
$p\in(Q/(Q+s),1)$ satisfies
$-\eta<s-Q/p$, and
$q\in(0,\infty]$, then $\ihb\subset\ilips$;
\item[{\rm(iii)}] If $p\in(1,\infty)$ and
$q\in(1,\infty]$,
then $\ihf\subset\ilipst$;
\item[{\rm(iv)}] If $\cx$ has
a local lower bound $Q$,
$p\in(Q/
(Q+s),1]$ satisfies $-\eta<s-Q/p$, and
$q\in(Q/[Q+
(\beta\wedge\gamma)],\infty]$,
then $\ihf\subset\ilipstu$ with
$u\in (0,\min\{p,q\})$.
\end{enumerate}
\end{proposition}

To show Proposition \ref{ibtlinlip}, we need the
following embedding lemma for inhomogeneous
Besov and Triebel--Lizorkin spaces on spaces
of homogeneous type,
which
comes from the combination of
\cite[Theorem 1.3]{hhhlp20} and
\cite[Theorem 7.4]{hwyy20}
with some slight modifications based on the
observation that it is easy to check that
\cite[Proposition 3.1]{hhhlp20} still holds true
when $Q\in[\omega, \infty)$. Recall that
Lemma \ref{iembed} was proved in
\cite[Theorem 1.3]{hhhlp20} only in the case
$Q=\omega$.

\begin{lemma}\label{iembed}
Let $\eta$ be as in Definition \ref{10.23.2},
$\omega$ as in \eqref{eq-doub}, $Q\in [\omega,\infty)$,
$\beta,\ \gamma\in(0,\eta)$,
$s\in(0,\beta\wedge\gamma)$, and $p\in
(Q/(Q+s),1]$ satisfy $-\eta<s-Q/p$.
Assume that $\cx$ has a local lower bound $Q$.
\begin{enumerate}
\item[{\rm (i)}] If $q\in (0,\infty]$, then
$$B_{p,q}^s(\cx)\subset B_{1,q}^{s-Q(
\frac{1}{p}-1)}(\cx).$$
\item[{\rm (ii)}] If $q\in (Q/[Q+
(\beta\wedge\gamma)],\infty]$, then
$$F_{p,q}^s(\cx)\subset F_{1,q}^{s-Q(
\frac{1}{p}-1)}(\cx).$$
\end{enumerate}
\end{lemma}

\begin{remark}
We point that the local lower bound $Q$ of $\cx$
in Proposition \ref{ibtlinlip} and Lemma \ref{iembed} is
required to be in $[\omega,\fz)$. It is still
unclear whether or not the conclusions
of Proposition \ref{ibtlinlip} and Lemma \ref{iembed}
still hold true if the local lower bound $Q$ of $\cx$ belongs to $(0,\omega)$.
Indeed, we prove Proposition \ref{ibtlinlip} by using Lemma \ref{iembed}, while
Lemma \ref{iembed} strongly depends on \cite[Proposition 3.1]{hhhlp20}
which needs $Q=\omega$ but it is easy to check that \cite[Proposition 3.1]{hhhlp20}
still holds true when $Q\in (\omega,\fz)$. This results in the restriction of Proposition \ref{ibtlinlip}
on the local lower bound of $\cx$.
\end{remark}

Now, we prove Proposition \ref{ibtlinlip}.
\begin{proof}[Proof of Proposition \ref{ibtlinlip}]
We first prove (i). By an argument similar to
that used in the proof of \cite[Theorem 6.12]{whhy},
we know that, for any $p\in[1,\infty]$ and
$q\in(0,\infty]$,
$\ihb\subset L^p(\cx)$.  Using this and
Proposition \ref{llblt}(ii), to prove (i), it suffices
to show that
\begin{equation}\label{leftihb}
\left\{\sum_{l=0}^\infty\delta^{-lsq}
\left[\frac{1}{\mu(B(x,\widetilde{C}\delta^l))}
\int_{B(x,\widetilde{C}\delta^l)}
|f(x)-f(y)|^p\,d\mu(y)\,d\mu(x)\right]^{q/p}
\right\}^{1/q}\lesssim\|f\|_{\ihb}.
\end{equation}
Note that there exists a unique $l_0\in\zz$ such that
$\widetilde{C}\in(\delta^{l_0+1},\delta^{l_0}]$.
Since the space $\lipsb$ is independent
of the choice of $\widetilde{C}$, without loss
of generality,
we may assume that $l_0\in(-\infty,0]\cap\zz$.
By Lemma \ref{icrf}, we have
\begin{align}\label{deleftihb}
&\left\{\sum_{l=0}^\infty\delta^{-lsq}
\left[\int_{\cx}\frac{1}{\mu(B(x,\widetilde{C}
\delta^l))}\int_{B(x,\widetilde{C}\delta^l)}
|f(x)-f(y)|^p\,d\mu(y)\,d\mu(x)\right]^{q/p}
\right\}^{1/q}\\
&\qquad \lesssim\left\{\sum_{l=0}
^\infty\delta^{-lsq}
\left[\int_{\cx}\frac{1}{\mu(B(x,\delta^{l+l_0}))}\right.\right.\notag\\
&\qquad\qquad\times\int_{B(x,\delta^{l+l_0})}
\left\{\sum_{\alpha \in \ca_0}\sum_{m=1}
^{N(0,\alpha)}\int_{\qo}
\left|\widetilde{Q}_0(x,z)-\widetilde{Q}_0(x,z)
\right|\,d\mu(z)Q^{0,m}_{\alpha,1}(f)\right.\notag\\
&\qquad\qquad+\sum_{k=1}^N\sum_{\alpha \in
\ca_k}\sum_{m=1}^{N(k,\alpha)}
\mu\left(\qa\right)\notag\\
&\qquad\qquad\left.\left.\left.\times\left|
\widetilde{Q}_k(x,\ya)-\widetilde{Q}_k(y,\ya)
\vpz{\sum_{m=1}^{N(k,\alpha)}}\right|
Q^{k,m}_{\alpha,1}(f)\right\}^p\,d\mu(y)
\,d\mu(x)\right]^{q/p}\right\}^{1/q}\notag\\
&\qquad\qquad + \left\{\sum_{l=0}^{N+2-
l_0}\delta^{-lsq}
\left[\int_{\cx}\frac{1}{\mu(B(x,\delta^{l+l_0}))}
\int_{B(x,\delta^{l+l_0})}
\left\{\sum_{k=N+1}^\infty\sum_{\alpha \in
\ca_k}\sum_{m=1}^{N(k,\alpha)}
\mu\left(\qa\right)\right.\right.\right.\notag\\
&\qquad\qquad\left.\left.\left.\times \left|
\widetilde{Q}_k(x,\ya)-\widetilde{Q}_k(y,\ya)
\vpz{\sum_{m=1}^{N(k,\alpha)}}\right|
\left|Q_kf\left(\ya\right)\right|\right\}
^p\,d\mu(y)\,d\mu(x)\right]^{q/p}\right\}^{1/q}\notag\\
&\qquad\qquad + \left\{\sum_{l=N+3-l_0}
^\infty\delta^{-lsq}
\left[\int_{\cx}\frac{1}{\mu(B(x,\delta^{l+l_0}))}
\int_{B(x,\delta^{l+l_0})}
\left\{\sum_{k=N+1}^{l+l_0-1}\sum_{\alpha \in
\ca_k}\sum_{m=1}^{N(k,\alpha)}
\mu\left(\qa\right)\right.\right.\right.\notag\\
&\qquad\qquad\left.\left.\left.\times \left|
\widetilde{Q}_k(x,\ya)-\widetilde{Q}_k(y,\ya)
\vpz{\sum_{m=1}^{N(k,\alpha)}}\right|
\left|Q_kf\left(\ya\right)\right|\right\}
^p\,d\mu(y)\,d\mu(x)\right]^{q/p}\right\}^{1/q}\notag\\
&\qquad\qquad + \left\{\sum_{l=N+3-l_0}
^\infty\delta^{-lsq}
\left[\int_{\cx}\frac{1}{\mu(B(x,\delta^{l+l_0}))}
\int_{B(x,\delta^{l+l_0})}
\left\{\sum_{k=l+l_0-1}^\infty\sum_{\alpha \in
\ca_k}\sum_{m=1}^{N(k,\alpha)}
\mu\left(\qa\right)
\right.\right.\right.\notag\\
&\qquad\qquad\left.\left.\left.\times
\vpz{\sum_{m=1}^{N(k,\alpha)}}
\left|
\widetilde{Q}_k(x,\ya)-\widetilde{Q}_k(y,\ya)\right|
\left|Q_kf\left(\ya\right)\right|\right\}
^p\,d\mu(y)\,d\mu(x)\right]^{q/p}\right\}^{1/q}\notag\\
&\qquad =: {\rm Y}_1+{\rm Y}_2+{\rm Y}_3+
{\rm Y}_4\notag.
\end{align}

We first estimate ${\rm Y}_1$. By \cite[(4.6)]{whhy}, we further obtain
\begin{align}\label{y1}
{\rm Y}_1&\lesssim \left\{\sum_{l=0}
^\infty\delta^{-lsq}
\left[\int_{\cx}\frac{1}{\mu(B(x,\delta^{l+l_0}))}
\int_{B(x,\delta^{l+l_0})}
\left\{\sum_{k=0}^N\sum_{\alpha \in \ca_k}
\sum_{m=1}^{N(k,\alpha)}
\mu\left(\qa\right)m_{\qa}(|Q_kf|)\right.\right.\right.\\
&\qquad\left.\left.\left.\times\sup_{z\in\qa}
\left|\widetilde{Q}_k(x,z)-\widetilde{Q}_k(y,z)\right|
\right\}^p\,d\mu(y)\,d\mu(x)\right]^{q/p}
\right\}^{1/q}\notag\\
&\lesssim \left\{\sum_{l=0}^{N+1-l_0}\delta^{-lsq}
\left[\int_{\cx}\frac{1}{\mu(B(x,\delta^{l+l_0}))}
\int_{B(x,\delta^{l+l_0})}
\left\{\sum_{k=0}^N\sum_{\alpha \in \ca_k}
\sum_{m=1}^{N(k,\alpha)}
\mu\left(\qa\right)m_{\qa}(|Q_kf|)\right.\right.\right.\notag\\
&\qquad\left.\left.\left.
\times\left[\sup_{z\in\qa}
\left|\widetilde{Q}_k(x,z)\right|+\sup_{z\in\qa}
\left|\widetilde{Q}_k(y,z)\right|\right]
\right\}^p\,d\mu(y)\,d\mu(x)\right]^{q/p}
\right\}^{1/q}\notag\\
&\qquad+\left\{\sum_{l=N+2-l_0}
^\infty\delta^{-lsq}
\left[\int_{\cx}\frac{1}{\mu(B(x,\delta^{l+l_0}))}
\int_{B(x,\delta^{l+l_0})}
\left\{\sum_{k=0}^N\sum_{\alpha \in \ca_k}
\sum_{m=1}^{N(k,\alpha)}
\mu\left(\qa\right)\right.\right.\right.\notag\\
&\qquad\left.\left.\left.\times m_{\qa}(|Q_kf|)
\sup_{z\in\qa}\left|\widetilde{Q}_k(x,z)-\widetilde{Q}_k(y,z)\right|
\right\}^p\,d\mu(y)\,d\mu(x)\right]^{q/p}
\right\}^{1/q}\notag\\
&=: {\rm Y}_{1,1}+{\rm Y}_{1,2}.\notag
\end{align}

To estimate ${\rm Y}_{1,1}$, note that,
for any $x\in\cx$ and $z\in\qa$,
$\delta^k+d(x,z)\sim\delta^k+d(x,\ya)$.
By this, the H\"older inequality, and Lemmas
\ref{6.15.1}(i) and \ref{9.14.1}, we obtain
\begin{align*}
&\sum_{k=0}^N\sum_{\alpha \in \ca_k}
\sum_{m=1}^{N(k,\alpha)}
\mu\left(\qa\right)m_{\qa}(|Q_kf|)
\left[\sup_{z\in\qa}
\left|\widetilde{Q}_k(x,z)\right|+\sup_{z\in\qa}
\left|\widetilde{Q}_k(y,z)\right|\right]\\
&\qquad\lesssim \left\{\sum_{k=0}
^N\sum_{\alpha \in \ca_k}\sum_{m=1}^{N(k,\alpha)}
\mu\left(\qa\right)\left[m_{\qa}(|Q_kf|)]
\right]^p\left[\sup_{z\in\qa}
\left|\widetilde{Q}_k(x,z)\right|+\sup_{z\in\qa}
\left|\widetilde{Q}_k(y,z)\right|\right]\right\}^{1/p},
\end{align*}
which, together with Lemmas \ref{6.15.1}(i) and  \ref{9.14.1},
\eqref{qxqy}, \eqref{9.14.1x}, and the
arbitrariness of $\ya$, further implies that
\begin{align*}
{\rm Y}_{1,1}&\lesssim \left\{\sum_{l=0}^{N+1-
l_0}\delta^{-lsq}
\sum_{k=0}^N\sum_{\alpha\in\ca_k}
\sum_{m=1}^{N(k,\alpha)}
\mu(\qa)\left[m_{\qa}(|Q_kf|)\right]^p\int_{\cx}
R_\gamma\left(x,\ya;k\right)
\,d\mu(x)\right\}^{1/p}\\
&\lesssim \|f\|_{\ihb},
\end{align*}
where $R_\gamma(x,\ya;k)$ is as in \eqref{rxy}.

Now, we estimate ${\rm Y}_{1,2}$.
By \eqref{4.23y},
the fact that $\delta^k+d(x,z)\sim\delta^k+d(x,
\ya)$ for any $x\in\cx$,
the regularity of $\widetilde{Q}_k$, and
$z\in\qa$, we find that
\begin{align}\label{y12}
{\rm Y}_{1,2}&\lesssim \left\{\sum_{l=N+2-l_0}
^\infty\delta^{-lsq}
\left[\int_{\cx}\frac{1}{\mu(B(x,\delta^{l+l_0}))}
\int_{B(x,\delta^{l+l_0})}
\left\{\sum_{k=0}^N\sum_{\alpha \in \ca_k}
\sum_{m=1}^{N(k,\alpha)}
\mu\left(\qa\right)m_{\qa}(|Q_kf|)\right.\right.\right.\\
&\qquad\left.\left.\left.\times\left[\frac{d(x,y)}
{\delta^k+d(x,\ya)}\right]^\beta
R_\gamma\left(x,\ya;k\right)
\right\}^p\,d\mu(y)\,d\mu(x)\right]^{q/p}\right\}^{1/q}\notag\\
&\lesssim \left\{\sum_{l=N+2-l_0}
^\infty\delta^{l(\beta-s)q}
\left[\int_{\cx}\frac{1}{\mu(B(x,\delta^{l+l_0}))}
\int_{B(x,\delta^{l+l_0})}
\left\{\sum_{k=0}^N\sum_{\alpha \in \ca_k}
\sum_{m=1}^{N(k,\alpha)}
\mu\left(\qa\right)m_{\qa}(|Q_kf|)\right.\right.\right.\notag\\
&\qquad\left.\left.\left.\times
\vpz{\frac{1}
{V_{\delta^k}(x)+V(x,\ya)}
\left[\frac{\delta^k}{\delta^k+d(x,\ya)}
\right]^\gamma}R_\gamma\left(x,\ya;k\right)
\right\}^p\,d\mu(y)\,d\mu(x)\right]^{q/p}
\right\}^{1/q}.\notag
\end{align}
From this, $s\in(0,\beta)$, $p\in[1,\infty]$,
and an argument similar to that used in the
estimation of ${\rm Y}_{1,1}$, we deduce that
$${\rm Y}_{1,2}\lesssim \|f\|_{\ihb}.$$

Next, we estimate ${\rm Y}_2$. We first note
that, by the H\"older inequality,
\eqref{4.23x}, and Lemma \ref{9.14.1}, for any
fixed $\widetilde{s}\in(0,s)$,
\begin{align*}
&\sum_{k=N+1}^\infty\sum_{\alpha \in \ca_k}
\sum_{m=1}^{N(k,\alpha)}
\mu\left(\qa\right)\left|\widetilde{Q}_k(x,\ya)-
\widetilde{Q}_k(y,\ya)\right|
\left|Q_kf\left(\ya\right)\right|\\
&\qquad\lesssim\left\{\sum_{k=N+1}
^\infty\delta^{-k\widetilde{s}p}\sum_{\alpha
\in \ca_k}\sum_{m=1}^{N(k,\alpha)}
\mu\left(\qa\right)\left|\widetilde{Q}_k(x,\ya)-
\widetilde{Q}_k(y,\ya)\right|
\left|Q_kf\left(\ya\right)\right|^p\right\}^{1/p}\\
&\qquad \qquad\times\left\{\sum_{k=N+1}
^\infty\delta^{k\widetilde{s}p'}\sum_{\alpha \in
\ca_k}\sum_{m=1}^{N(k,\alpha)}
\mu\left(\qa\right)\left|\widetilde{Q}_k(x,\ya)-
\widetilde{Q}_k(y,\ya)\right|\right\}^{1/p'}\\
&\qquad\lesssim\left\{\sum_{k=N+1}
^\infty\delta^{-k\widetilde{s}p}\sum_{\alpha
\in \ca_k}\sum_{m=1}^{N(k,\alpha)}
\mu\left(\qa\right)\left|\widetilde{Q}_k(x,\ya)-
\widetilde{Q}_k(y,\ya)\right|
\left|Q_kf\left(\ya\right)\right|^p\right\}^{1/p}.
\end{align*}
From this, \eqref{qxqy}, \eqref{9.14.1x}, and the
arbitrariness of $\ya$, we deduce that
\begin{align*}
{\rm Y}_2&\lesssim \left\{\sum_{l=0}^{N+2-
l_0}\delta^{-lsq}
\left(\sum_{k=N+1}^\infty\delta^{-ksq}
\left[\sum_{\alpha\in\ca_k}\sum_{m=1}^{N(k,\alpha)}
\mu(\qa)\left|Q_kf\left(\ya\right)\right|
^p\right]^{q/p}\right)\right.\\
&\qquad\left.\times\left[\sum_{k=N+1}
^\infty\delta^{k(s-\widetilde{s})p(q/p)'}
\right]^{1/(q/p)'}\right\}^{1/p}\\
&\lesssim \|f\|_{\ihb},
\end{align*}
where we used \eqref{r} when $q/p\in(0,1]$, or
the H\"older inequality when $q/p\in (1,\infty]$.

We now estimate ${\rm Y}_3$. By
\eqref{4.23y}, we conclude that
\begin{align}\label{y3}
{\rm Y}_3&\lesssim\left\{\sum_{l=N+3-l_0}
^\infty\delta^{-lsq}
\left[\int_{\cx}\frac{1}{\mu(B(x,\delta^{l+l_0}))}\int_{B(x,\delta^{l+l_0})}
\left\{\sum_{k=N+1}^{l+l_0-1}\sum_{\alpha \in \ca_k}\sum_{m=1}^{N(k,\alpha)}
\mu\left(\qa\right)\right.\right.\right.\\
&\qquad\left.\left.\left.\times\left|Q_kf\left(\ya\right)
\right|\left[\frac{d(x,y)}{\delta^k+d(x,\ya)}
\right]^\beta
R_\gamma\left(x,\ya;k\right)\right\}^p\,d\mu(y)\,d\mu(x)
\right]^{q/p}\right\}^{1/q}\notag\\
&\lesssim\left\{\sum_{l=N+3-l_0}
^\infty\delta^{l(\beta-s)q}
\left[\int_{\cx}\frac{1}{\mu(B(x,\delta^{l+l_0}))}\int_{B(x,\delta^{l+l_0})}
\left\{\sum_{k=N+1}^{l+l_0-1}\delta^{-k\beta}
\sum_{\alpha \in \ca_k}\sum_{m=1}^{N(k,\alpha)}
\mu\left(\qa\right)\right.\right.\right.\notag\\
&\qquad\left.\left.\left.\times
\vpz{\sum_{k=N+1}^{l+l_0-1}}
\left|Q_kf\left(\ya\right)\right|
R_\gamma\left(x,\ya;k\right)\right\}^p\,d\mu(y)\,d\mu(x)
\right]^{q/p}\right\}^{1/q},\notag
\end{align}
where $R_\gamma(x,\ya;k)$ is as in \eqref{rxy}.
Due to $p\in[1,\infty]$, by \eqref{y3}, the
H\"older inequality, Lemma \ref{9.14.1}, and
\eqref{9.14.1x}, we have
\begin{align*}
{\rm Y}_3&\lesssim\left\{\sum_{l=N+3-l_0}
^\infty\delta^{l(\widetilde{\beta}-s)q}
\left[\int_{\cx}\frac{1}{\mu(B(x,\delta^{l+l_0}))}
\int_{B(x,\delta^{l+l_0})}
\left\{\sum_{k=N+1}^{l+l_0-1}\delta^{-
k\widetilde{\beta}p}\sum_{\alpha \in \ca_k}
\sum_{m=1}^{N(k,\alpha)}
\mu\left(\qa\right)\right.\right.\right.\\
&\qquad\left.\left.\left.\times
\left|Q_kf\left(\ya\right)\right|^p
\vpz{\frac{1}{V_{\delta^k}(x)+V(x,\ya)}
\left[\frac{\delta^k}{\delta^k+d(x,\ya)}
\right]^\gamma} R_\gamma\left(x,\ya;k\right)\right\}\,d\mu(y)\,d\mu(x)
\right]^{q/p}\right\}^{1/q}\\
&\lesssim \left\{\sum_{l=N+3-l_0}
^\infty\delta^{l(\widetilde{\beta}-s)q}
\left[\sum_{k=N+1}^{l+l_0-1}\delta^{-
k\widetilde{\beta}p}\sum_{\alpha \in \ca_k}
\sum_{m=1}^{N(k,\alpha)}
\mu\left(\qa\right)\left|Q_kf\left(\ya\right)
\right|^p\right]^{q/p}\right\}^{1/q},
\end{align*}
where $\widetilde{\beta}\in(s,\beta)$ is a fixed constant.
From this, \eqref{r} when $q/p\in(0,1]$, or the
H\"older inequality
when $q/p\in(1,\infty]$, and the arbitrariness of
$\ya$, we deduce that
$${\rm Y}_3\lesssim \|f\|_{\ihb}.$$

Finally, using an argument similar to that used
in the estimation of ${\rm Y}_2$, we also
obtain ${\rm Y}_4\lesssim \|f\|_{\ihb}$. This
finishes the proof of \eqref{leftihb} and hence of (i).

Next, we show (ii). Assume that $p\in(Q/
(Q+s),1)$,
$q\in(0,\infty]$, $\{Q_k\}_{k=0}
^\infty$ is an exp-IATI, and $f\in\ihb$. Then, by Lemma
\ref{icrf}, for any $x\in\cx$, let
\begin{align*}
\widetilde{f}(x) &:= \sum_{\alpha \in
\ca_0}\sum_{m=1}^{N(0,\alpha)}\int_{\qo}
\widetilde{Q}_0(x,y)\,d\mu(y)Q^{0,m}_{\alpha,1}(f)\\
&\qquad+\sum_{k=1}^N\sum_{\alpha \in \ca_k}
\sum_{m=1}^{N(k,\alpha)}
\mu\left(\qa\right)\widetilde{Q}_k(x,
\ya)Q^{k,m}_{\alpha,1}(f)\\
& \qquad+\sum_{k=N+1}^\infty\sum_{\alpha
\in \ca_k}\sum_{m=1}^{N(k,\alpha)}
\mu\left(\qa\right)\widetilde{Q}_k(x,
\ya)Q_kf\left(\ya\right).
\end{align*}
We claim that $\widetilde{f}$ is well defined
and, moreover, $\widetilde{f}\in L^p(\cx)$.
Indeed, from \eqref{9.14.1x}, we deduce that
\begin{align*}
\left\|\widetilde{f}\right\|_{L^p(\cx)}
^p&\lesssim\sum_{k=0}
^N\sum_{\alpha\in\ca_k}\sum_{m=1}^{N(k,\alpha)}
\left[\mu(\qa)m_{\qa}(|Q_kf|)\right]^p
\int_{\cx}\left[R_\gamma\left(x,
\ya;k\right)\right]^p\,d\mu(x)\\
&\qquad+\sum_{k=N+1}
^\infty\sum_{\alpha\in\ca_k}\sum_{m=1}^{N(k,\alpha)}
\left[\mu(\qa)|Q_kf(\ya)|\right]^p
\int_{\cx}\left[R_\gamma\left(x,
\ya;k\right)\right]^p\,d\mu(x)\\
&\lesssim \sum_{k=0}^N\sum_{\alpha\in\ca_k}
\sum_{m=1}^{N(k,\alpha)}
\mu(\qa)\left[m_{\qa}(|Q_kf|)\right]^p+
\sum_{k=N+1}^\infty\sum_{\alpha\in\ca_k}
\sum_{m=1}^{N(k,\alpha)}
\mu(\qa)|Q_kf(\ya)|^p\\
&\lesssim \|f\|_{\ihb}^p,
\end{align*}
where $R_\gamma(x,\ya;k)$ is as in \eqref{rxy} and we used \eqref{r} when $q/p\in(0,1]$, or
the H\"older inequality
when $q/p\in(1,\infty]$. Using this,
we obtain $\widetilde{f}\in L^p(\cx)$
and $\|\widetilde{f}\|_{L^p(\cx)}\lesssim \|f\|_{\ihb}$.
Moreover, by Lemma \ref{iembed}, we
conclude that
$f\in B_{1,q}^{s-Q(\frac{1}{p}-1)}(\cx)$.
Thus, from an argument similar to that used in
the proof of \cite[Theorem 6.12]{whhy},
we deduce
that $f=\widetilde{f}$ in $(\icgg)'$.
Based on this, in what follows,
we do not need to distinguish $f$ or
$\widetilde{f}$.
Using this and Proposition \ref{llblt}(ii),
we know that, to prove
(i), it suffices to show that  \eqref{leftihb}
still holds true for any given $s$, $p$, and $q$ as in (ii).
Therefore, by \eqref{deleftihb} and \eqref{y1},
we need to estimate
${\rm Y}_{1,1}$, ${\rm Y}_{1,2}$, ${\rm Y}_2$, $
{\rm Y}_3$, and ${\rm Y}_4$, respectively.

To estimate ${\rm Y}_{1,1}$,  by \eqref{r}, \eqref{4.23x},
Lemma \ref{6.15.1}(i), \eqref{qxqy},
\eqref{9.14.1x}, and the arbitrariness of $\ya$, we have
\begin{align*}
{\rm Y}_{1,1}&\lesssim \left\{\sum_{l=0}^{N+1-
l_0}\delta^{-lsq}
\sum_{k=0}^N\sum_{\alpha\in\ca_k}
\sum_{m=1}^{N(k,\alpha)}
\left[\mu(\qa)m_{\qa}(|Q_kf|)\right]^p
\int_{\cx}\left[R_\gamma\left(x,\ya;k\right)
\right]^p\,d\mu(x)\right\}^{1/p}\\
&\lesssim \|f\|_{\ihb}.
\end{align*}

To estimate ${\rm Y}_{1,2}$, from \eqref{y12} and
an argument similar to
that used in the estimation of ${\rm Y}_{1,1}$, we deduce that
$${\rm Y}_{1,2}\lesssim \|f\|_{\ihb}.$$

Now, we estimate ${\rm Y}_2$.
By $p\in(Q/(Q+s),1]$,
\eqref{r}, \eqref{qxqy}, \eqref{4.23x}, and
\eqref{9.14.1x}, we find that
\begin{align*}
{\rm Y}_2&\lesssim \left[\sum_{l=0}^{N+2-
l_0}\delta^{-lsq}
\left\{\sum_{k=N+1}
^\infty\sum_{\alpha\in\ca_k}\sum_{m=1}^{N(k,\alpha)}
\left[\mu(\qa)\left|Q_kf\left(\ya\right)\right|
\right]^p\right.\right.\\
&\qquad\left.\left.\times\int_{\cx}\left[\frac{1}
{V_{\delta^k}(\ya)+{V(x,\ya)}}\right]^p
\left[\frac{\delta^k}{\delta^k+d(x,\ya)}
\right]^{\gamma p}\,d\mu(x)\right\}^{q/p}\right]^{1/p}\\
&\lesssim \left[\sum_{l=0}^{N+2-l_0}\delta^{-lsq}
\left\{\sum_{k=N+1}
^\infty\sum_{\alpha\in\ca_k}\sum_{m=1}^{N(k,\alpha)}
\mu(\qa)\left|Q_kf\left(\ya\right)\right|
^p\right\}^{q/p}\right]^{1/p}.
\end{align*}
If $q/p\in(0,1]$, by this, \eqref{r},
$s\in(0,\beta\wedge\gamma)$,  and the
arbitrariness of $\ya$, we know that
\begin{align*}
{\rm Y}_2&\lesssim \left[\sum_{l=0}^{N+2-
l_0}\delta^{-lsq}
\sum_{k=N+1}^\infty\delta^{-ksq}\left
\{\sum_{\alpha\in\ca_k}\sum_{m=1}^{N(k,\alpha)}
\mu(\qa)\left|Q_kf\left(\ya\right)\right|
^p\right\}^{q/p}\right]^{1/p}\\
&\lesssim \|f\|_{\ihb};
\end{align*}
while, if $q/p\in(0,\infty]$, by this, the H\"older
inequality, and the arbitrariness of $\ya$, we
find that
\begin{align*}
{\rm Y}_2&\lesssim \left\{\sum_{l=0}^{N+2-
l_0}\delta^{-lsq}
\left(\sum_{k=N+1}^\infty\delta^{-ksq}
\left[\sum_{\alpha\in\ca_k}\sum_{m=1}^{N(k,\alpha)}
\mu(\qa)\left|Q_kf\left(\ya\right)\right|
^p\right]^{q/p}\right)\left[\sum_{k=N+1}^\infty
\delta^{lsp(q/p)'}\right]^{1/(q/p)'}\right\}^{1/p}\\
&\lesssim \|f\|_{\ihb}.
\end{align*}

Next, we estimate ${\rm Y}_3$. Since
$p\in(Q/(Q+s),1]$,
from \eqref{y3}, \eqref{r}, and \eqref{9.14.1x}, it
follows that
$${\rm Y}_3\lesssim\left\{\sum_{l=N+3-l_0}
^\infty\delta^{l(\beta-s)q}
\left[\sum_{k=N+1}^{l+l_0-1}\delta^{-k\beta p}
\left\{\sum_{\alpha \in \ca_k}\sum_{m=1}^{N(k,\alpha)}
\mu\left(\qa\right)\left|Q_kf\left(\ya\right)\right|^p
\right\}\right]^{q/p}\right\}^{1/q}.$$
If $q/p\in(0,1]$, by this, \eqref{r}, and the
arbitrariness of $\ya$, we conclude that
\begin{align*}
{\rm Y}_3&\lesssim\left\{\sum_{k=N+1}
^\infty\delta^{-k\beta p}
\left[\sum_{l=k-l_0+1}^\infty\delta^{l(\beta-
s)q}\right]\left[\sum_{\alpha \in \ca_k}
\sum_{m=1}^{N(k,\alpha)}
\mu\left(\qa\right)\left|Q_kf\left(\ya\right)
\right|^p
\right]^{q/p}\right\}^{1/q}\\
&\lesssim\|f\|_{\ihb};
\end{align*}
while, if $p/q\in(1,\infty]$, by this, the H\"older
inequality, $s\in(0,\beta)$, and the arbitrariness
of $\ya$, we conclude that
\begin{align*}
{\rm Y}_3&\lesssim\left\{\sum_{l=N+3-l_0}
^\infty\delta^{l(\beta-s)q}
\left[\sum_{k=N+1}^{l+l_0-1}\delta^{-k(\beta+s)
p/2}\left\{\sum_{\alpha \in \ca_k}\sum_{m=1}^{N(k,\alpha)}
\mu\left(\qa\right)\left|Q_kf\left(\ya\right)\right|^p
\right\}^{q/p}\right]\right.\\
&\qquad\times\left.\left[\sum_{k=N+1}
^{l+l_0-1}\delta^{-k(\beta-s)p(q/p)'/
2}\right]^{\frac{1}{(q/p)'}\frac{q}{p}}\right\}^{1/q}\\
&\lesssim\|f\|_{\ihb}.
\end{align*}

Finally, using an argument similar to that used
in the estimation of ${\rm Y}_2$,
we also obtain ${\rm Y}_4\lesssim \|f\|_{\ihb}$.
Combining the estimates of  ${\rm Y}_{1,1}$, $
{\rm Y}_{1,2}$, ${\rm Y}_2$,
${\rm Y}_3$, and ${\rm Y}_4$, we show that
\eqref{leftihb} still holds true
for any given $s$, $p$, and $q$ as in (ii), which completes
the proof of (ii).

Now, we show (iii). Using an argument similar to
that used in the proof of
\cite[Theorem 6.12 (II)]{whhy}, we know that,
for any given $s\in(0,\beta\wedge\gamma)$,
$p\in(1,\infty)$, and $q\in(1,\infty]$,
$$\ihf\subset L^p(\cx)\quad\text {and, for
any}\ f\in\ihf,\
\|f\|_{L^p(\cx)}\lesssim\|f\|_{\ihf}.$$
By this and Proposition \ref{llblt}(iii), to prove
(iii), it suffices to show that,
for any $f\in\ihf$ with
$s\in(0,\beta\wedge\gamma)$, $p\in(1,\infty)
$, and $q\in(1,\infty]$,
\begin{equation}\label{lpihf}
\left\|\left\{\sum_{l=0}^\infty\delta^{-lsq}
\left[\frac{1}{\mu(B(\cdot,\widetilde{C}
\delta^l))}\int_{B(\cdot,\widetilde{C}\delta^l)}
|f(\cdot)-f(y)|\,d\mu(y)\right]^q\right\}^{1/q}
\right\|_{L^p(\cx)}\lesssim\|f\|_{\ihf}.
\end{equation}

Next, we show \eqref{lpihf}. Let $\{Q_k\}_{k=0}
^\infty$ be an exp-IATI and $f\in\ihf$. Note that
there exists a unique $l_0\in\zz$ such that
$\widetilde{C}\in(\delta^{l_0+1},\delta^{l_0}]$.
Since the space $\ilipst$ is independent of $
\widetilde{C}$ in Definition \ref{lipii}(ii),
without loss of generality, we
may assume $l_0\in[N+2, \infty)\cap\zz$.
By Lemma \ref{icrf},  for any $l\in\zz_+$,
$x\in\cx$,
and $y\in B(x,\widetilde{C}\delta^l)$,  we write
\begin{align*}
f(x)-f(y)&= \sum_{\alpha \in \ca_0}\sum_{m=1}
^{N(0,\alpha)}\int_{\qo}
\left[\widetilde{Q}_0(x,z)-\widetilde{Q}_0(y,z)
\right]\,d\mu(z)Q^{0,m}_{\alpha,1}(f)\\
&\qquad+\sum_{k=1}^N\sum_{\alpha \in \ca_k}
\sum_{m=1}^{N(k,\alpha)}
\mu\left(\qa\right)\left[\widetilde{Q}_k(x,\ya)-
\widetilde{Q}_k(y,\ya)\right]Q^{k,m}_{\alpha,1}(f)\\
& \qquad+\sum_{k=N+1}^{l+l_0-1}\sum_{\alpha
\in \ca_k}\sum_{m=1}^{N(k,\alpha)}
\mu\left(\qa\right)\left[\widetilde{Q}_k(x,\ya)-
\widetilde{Q}_k(y,\ya)
\right]Q_kf\left(\ya\right)\\
&\qquad +\sum_{k=l+l_0}^\infty\sum_{\alpha
\in \ca_k}\sum_{m=1}^{N(k,\alpha)}\cdots\\
&=: {\rm Z}_1(x,y)+{\rm Z}_2(x,y)+{\rm Z}_3(x,y)
+{\rm Z}_4(x,y).
\end{align*}

We first estimate ${\rm Z}_1$ and ${\rm Z}_2$.
By \cite[(4.6)]{whhy}, \eqref{4.23y}, and
Lemmas \ref{6.15.1}(i) and \ref{10.18.5}, we have
\begin{align*}
&|{\rm Z}_1(x,y)|+|{\rm Z}_2(x,y)|\\
&\qquad\lesssim
\sum_{k=0}^N\sum_{\alpha\in\ca_k}
\sum_{m=1}^{N(k,\alpha)}
\mu\left(\qa\right)m_{\qa}(|Q_kf|)
\left[\frac{d(x,y)}{\delta^k+d(x,
\ya)}\right]^\beta R_\gamma\left(x,\ya;k\right)\\
&\qquad\lesssim \delta^{l\beta}\sum_{k=0}
^N\sum_{\alpha\in\ca_k}\sum_{m=1}^{N(k,
\alpha)}\mu\left(\qa\right)m_{\qa}(|Q_kf|)
R_\gamma\left(x,\ya;k\right)\\
&\qquad\lesssim \delta^{l\beta}\sum_{k=0}
^NM\left(\sum_{\alpha\in\ca_k}\sum_{m=1}
^{N(k,\alpha)}m_{\qa}(|Q_kf|)\mathbf{1}_{\qa}\right)(x),
\end{align*}
where $R_\gamma(x,\ya;k)$ is as in \eqref{rxy}.

Similarly, we obtain
$$|{\rm Z}_3(x,y)|\lesssim \sum_{k=N+1}
^{l+l_0-1}\delta^{(l-k)\beta}
M\left(\sum_{\alpha\in\ca_k}\sum_{m=1}^{N(k,\alpha)}
|Q_kf\left(\ya\right)|\mathbf{1}_{\qa}\right)(x).$$
From these estimates, the H\"older inequality,
Lemma \ref{fsvv}, and the arbitrariness of
$\ya$,  we deduce that
\begin{align*}
&\left\|\left\{\sum_{l=0}^\infty\delta^{-lsq}
\left[\frac{1}{\mu(B(\cdot,\widetilde{C}
\delta^l))}\int_{B(\cdot,\widetilde{C}\delta^l)}
|{\rm Z}_1(\cdot,y)+{\rm Z}_2(\cdot,y)+{\rm Z}
_3(\cdot,y)|\,d\mu(y)\right]^q\right\}^{1/q}
\right\|_{L^p(\cx)}\\
&\qquad\lesssim\left\|\sum_{k=0}
^NM\left(\sum_{\alpha\in\ca_k}\sum_{m=1}^{N(k,\alpha)}
m_{\qa}(|Q_kf|)\mathbf{1}_{\qa}\right)\right\|_{L^p(\cx)}\\
&\qquad\qquad +\left\|\left\{\sum_{k=N+1}
^{l+l_0-1}\delta^{-ksq}
\left[M\left(\sum_{\alpha\in\ca_k}
\sum_{m=1}^{N(k,\alpha)}\left|
Q_kf\left(\ya\right)\right|\mathbf{1}_{\qa}
\right)\right]^q\right\}^{1/q}\right\|_{L^p(\cx)}\\
&\qquad\lesssim\|f\|_{\ihf}.
\end{align*}

Now, we estimate ${\rm Z}_4$. By
\eqref{4.23x} and  Lemma \ref{10.18.5}, we have
\begin{align*}
|{\rm Z}_4(x,y)|&\lesssim \sum_{k=l+l_0}
^\infty\left[M\left(\sum_{\alpha\in\ca_k}
\sum_{m=1}^{N(k,\alpha)}
|Q_kf\left(\ya\right)|\mathbf{1}_{\qa}\right)(x)
+M\left(\sum_{\alpha\in\ca_k}\sum_{m=1}^{N(k,\alpha)}
|Q_kf\left(\ya\right)|\mathbf{1}_{\qa}\right)(y)\right].
\end{align*}
From this and the H\"older inequality, we
further deduce that
\begin{align*}
&\frac{1}{\mu(B(x,\widetilde{C}\delta^l))}
\int_{B(x,\widetilde{C}\delta^l)}
|{\rm Z}_4(x,y)|\,d\mu(y)\\
&\qquad\lesssim \delta^{ls/
2}\left\{\sum_{k=l+l_0}^\infty\delta^{-ksq/2}
\left[M\left(\sum_{\alpha\in\ca_k}
\sum_{m=1}^{N(k,\alpha)}|Q_kf\left(\ya\right)|
\mathbf{1}_{\qa}\right)(x)\right.\right.\\
&\qquad \qquad\left.\left.+M\circ
M\left(\sum_{\alpha\in\ca_k}\sum_{m=1}^{N(k,\alpha)}
|Q_kf\left(\ya\right)|\mathbf{1}_{\qa}\right)(x)
\right]^q\right\}^{1/q},
\end{align*}
which, together with Lemma \ref{fsvv} and the
arbitrariness of $\ya$, implies that
\begin{align*}
&\left\|\left\{\sum_{l=0}^\infty\delta^{-lsq}
\left[\frac{1}{\mu(B(\cdot,\widetilde{C}
\delta^l))}\int_{B(\cdot,\widetilde{C}\delta^l)}
|{\rm Z}_4(\cdot,y)|\,d\mu(y)\right]^q\right\}
^{1/q}\right\|_{L^p(\cx)}\\
&\qquad\lesssim\left\|\left\{\sum_{k=N+1}
^\infty\delta^{-ksq}
\left[M\left(\sum_{\alpha\in\ca_k}\sum_{m=1}
^{N(k,\alpha)}
|Q_kf\left(\ya\right)|\mathbf{1}_{\qa}\right)
\right.\right.\right.\\
&\qquad \qquad\left.\left.\left.+M\circ
M\left(\sum_{\alpha\in\ca_k}\sum_{m=1}^{N(k,\alpha)}
|Q_kf\left(\ya\right)|\mathbf{1}_{\qa}\right)
\right]^q\right\}^{1/q}\right\|_{L^p(\cx)}\\
&\qquad\lesssim\|f\|_{\ihf}.
\end{align*}
This finishes the proof of \eqref{lpihf} and hence of (iii).

Finally, we show (iv). By an argument similar to
that used in the proof of (ii),
we know that, for any given $s$, $p$, and $q$ as in
this proposition,
$$\ihf\subset L^p(\cx)\quad\text {and, for any}\quad f\in\ihf,\
\|f\|_{L^p(\cx)}\lesssim\|f\|_{\ihf}.$$

Let $\{Q_k\}_{k=0}^\infty$ be an exp-IATI and
$f\in\ihf$. Note that
there exists a unique $l_0\in\zz$ such that
$\widetilde{C}\in(\delta^{l_0+1},\delta^{l_0}]$.
In what follows,
we assume $l_0\in[N+2, \infty)\cap\zz$.
By Lemma \ref{icrf},
for any $l\in\zz_+$, $x\in\cx$, and $y\in B(x,
\widetilde{C}\delta^l)$, we write
\begin{align*}
f(x)-f(y)&= \sum_{\alpha \in \ca_0}\sum_{m=1}
^{N(0,\alpha)}\int_{\qo}
\left[\widetilde{Q}_0(x,z)-\widetilde{Q}_0(y,z)
\right]\,d\mu(z)Q^{0,m}_{\alpha,1}(f)\\
&\qquad+\sum_{k=1}^N\sum_{\alpha \in \ca_k}
\sum_{m=1}^{N(k,\alpha)}
\mu\left(\qa\right)\left[\widetilde{Q}_k(x,\ya)-
\widetilde{Q}_k(y,\ya)\right]Q^{k,m}_{\alpha,1}(f)\\
& \qquad+\sum_{k=N+1}^{l+l_0-1}\sum_{\alpha
\in \ca_k}\sum_{m=1}^{N(k,\alpha)}
\mu\left(\qa\right)\left[\widetilde{Q}_k(x,\ya)-
\widetilde{Q}_k(y,\ya)
\right]Q_kf\left(\ya\right)\\
&\qquad +\sum_{k=l+l_0}^\infty\sum_{\alpha
\in \ca_k}\sum_{m=1}^{N(k,\alpha)}\cdots\\
&=: {\rm R}_1(x,y)+{\rm R}_2(x,y)+{\rm R}_3(x,y)
+{\rm R}_4(x,y).
\end{align*}

We first estimate ${\rm R}_1$ and ${\rm R}_2$.
By \cite[(4.6)]{whhy}, \eqref{4.23y}, and
Lemmas \ref{6.15.1}(i) and \ref{10.18.5}, we have,
for any fixed $r\in (Q/(Q+\gamma),1]$
and for any $x\in\cx$ and
$y\in B(x,\widetilde{C}\delta^{l})$,
\begin{align*}
|{\rm R}_1(x,y)|+|{\rm R}_2(x,y)|&\lesssim
\sum_{k=0}^N\sum_{\alpha\in\ca_k}
\sum_{m=1}^{N(k,\alpha)}
\mu\left(\qa\right)m_{\qa}(|Q_kf|)
\left[\frac{d(x,y)}{\delta^k+d(x,\ya)}\right]^\beta
R_\gamma\left(x,\ya;k\right)\\
&\lesssim \delta^{l\beta}\sum_{k=0}
^N\sum_{\alpha\in\ca_k}\sum_{m=1}^{N(k,\alpha)}
\mu\left(\qa\right)m_{\qa}(|Q_kf|)
R_\gamma\left(x,\ya;k\right)\\
&\lesssim \delta^{l\beta}\sum_{k=0}
^N\left[M\left(\sum_{\alpha\in\ca_k}
\sum_{m=1}^{N(k,\alpha)}
\left[m_{\qa}(|Q_kf|)\right]^r\mathbf{1}_{\qa}
\right)(x)\right]^{\frac{1}{r}},
\end{align*}
where $R_\gamma(x,\ya;k)$ is as in \eqref{rxy}.
Using these estimates and Lemma \ref{fsvv}, and
choosing $r\in (Q/(Q+\gamma),p)$,
we conclude that
\begin{align*}
&\left\|\left\{\sum_{l=0}^\infty\delta^{-lsq}
\left[\frac{1}{\mu(B(\cdot,\widetilde{C}
\delta^l))}\int_{B(\cdot,\widetilde{C}\delta^l)}
|{\rm R}_1(\cdot,y)+{\rm R}_2(\cdot,y)|
^u\,d\mu(y)\right]^{\frac{q}{u}}\right\}^{\frac{1}
{q}}\right\|_{L^p(\cx)}\\
&\qquad\lesssim\left\|\left\{\sum_{l=0}
^\infty\delta^{l(\beta-s)q}\sum_{k=0}^N
\left[M\left(\sum_{\alpha\in\ca_k}\sum_{m=1}
^{N(k,\alpha)}\left[m_{\qa}(|Q_kf|)\right]^r
\mathbf{1}_{\qa}\right)\right]^{\frac{q}{r}}
\right\}^{\frac{1}{q}}\right\|_{L^p(\cx)}\\
&\qquad\lesssim\sum_{k=0}^N\left\|
\left\{M\left(\sum_{\alpha\in\ca_k}\sum_{m=1}
^{N(k,\alpha)}\left[m_{\qa}(|Q_kf|)\right]^r
\mathbf{1}_{\qa}\right)\right\}^{\frac{1}{r}}
\right\|_{L^p(\cx)}\\
&\qquad\lesssim\sum_{k=0}^N\left\|
\left\{\sum_{\alpha\in\ca_k}\sum_{m=1}^{N(k,
\alpha)}\left[m_{\qa}(|Q_kf|)\right]^r
\mathbf{1}_{\qa}\right\}^{\frac{1}{r}}\right\|
_{L^p(\cx)}\\
&\qquad\sim\left\{\sum_{k=0}
^N\sum_{\alpha\in\ca_k}\sum_{m=1}^{N(k,\alpha)}
\left[m_{\qa}(|Q_kf|)\right]^p\mu(\qa)\right\}^{\frac{1}{p}}
\lesssim\|f\|_{\ihf}.
\end{align*}

Now, we estimate ${\rm R}_3$.
Similarly to the estimations of ${\rm R}_1$ and
${\rm R}_2$,  we have,
for any $x\in\cx$ and $y\in B(x,\widetilde{C}\delta^{l})$,
$$|{\rm R}_3(x,y)|\lesssim \sum_{k=N+1}
^{l+l_0-1}\delta^{(l-k)\beta}
\left[M\left(\sum_{\alpha\in\ca_k}\sum_{m=1}
^{N(k,\alpha)}
\left|Q_kf\left(\ya\right)\right|^r\mathbf{1}
_{\qa}\right)(x)\right]^{\frac{1}{r}},$$
which further implies that
\begin{align*}
&\left\{\sum_{l=0}^\infty\delta^{-lsq}
\left[\frac{1}{\mu(B(x,\widetilde{C}\delta^l))}
\int_{B(x,\widetilde{C}\delta^l)}
|{\rm R}_3(x,y)|^u\,d\mu(y)\right]^{\frac{q}{u}}
\right\}^{\frac{1}{q}}\\
&\qquad \lesssim \left[\sum_{l=0}
^\infty\delta^{-lsq}
\left\{\sum_{k=N+1}^{l+l_0-1}\delta^{(l-k)\beta}
\left[M\left(\sum_{\alpha\in\ca_k}
\sum_{m=1}^{N(k,\alpha)}\left|
Q_kf\left(\ya\right)\right|^r\mathbf{1}_{\qa}
\right)(x)
\right]^{\frac{1}{r}}\right\}^q\right]^{\frac{1}{q}}\\
&\qquad\lesssim
\left\{\sum_{k=N+1}^\infty\delta^{-ksq}
\left[M\left(\sum_{\alpha\in\ca_k}\sum_{m=1}
^{N(k,\alpha)}
\left|Q_kf\left(\ya\right)\right|^r\mathbf{1}
_{\qa}\right)(x)\right]^{\frac{q}{r}}\right\}
^{\frac{1}{q}},
\end{align*}
where, in the last inequality, we used \eqref{r}
when $q\in(Q/[Q+
(\beta\wedge\gamma),1]$,
or the H\"older inequality when $q\in(1,\infty]$,
and $s\in (0,\beta)$.
Using this, Lemma \ref{fsvv}, and the
arbitrariness of $\ya$,
and choosing $r\in (Q/(Q+
\gamma), \min\{p,q\})$,
we find that
\begin{align*}
&\left\|\left\{\sum_{l=0}^\infty\delta^{-lsq}
\left[\frac{1}{\mu(B(\cdot,\widetilde{C}
\delta^l))}\int_{B(\cdot,\widetilde{C}\delta^l)}
|{\rm R}_3(\cdot,y)|^u\,d\mu(y)\right]^{\frac{q}
{u}}\right\}^{\frac{1}{q}}\right\|_{L^p(\cx)}\\
&\qquad\lesssim \left\|\left\{\sum_{k=N+1}
^\infty\delta^{-ksq}
\left[M\left(\sum_{\alpha\in\ca_k}\sum_{m=1}
^{N(k,\alpha)}
\left|Q_kf\left(\ya\right)\right|^r\mathbf{1}
_{\qa}\right)\right]^{\frac{q}{r}}\right\}^{\frac{1}
{q}}\right\|_{L^p(\cx)}\\
&\qquad\sim \left\|\left\{\sum_{k=N+1}
^\infty\delta^{-ksq}
\left[M\left(\sum_{\alpha\in\ca_k}\sum_{m=1}^{N(k,\alpha)}
\left|Q_kf\left(\ya\right)\right|^r\mathbf{1}
_{\qa}\right)\right]^{\frac{q}{r}}\right\}^{\frac{r}{q}}
\right\|_{L^{\frac{p}{r}}(\cx)}^{\frac{1}{r}}\\
&\qquad\lesssim \left\|\left\{\sum_{k=N+1}
^\infty\delta^{-ksq}
\sum_{\alpha\in\ca_k}\sum_{m=1}^{N(k,\alpha)}
\left|Q_kf\left(\ya\right)\right|^q\mathbf{1}
_{\qa}\right\}^{\frac{r}{q}}
\right\|_{L^{\frac{p}{r}}(\cx)}^{\frac{1}{r}}\lesssim\|f\|_{\ihf}.
\end{align*}

Finally, from an argument similar to that used
in the proof of Proposition \ref{btlinlip}(iii),
we deduce that
$$\left\|\left\{\sum_{l=0}^\infty\delta^{-lsq}
\left[\frac{1}{\mu(B(\cdot,\widetilde{C}
\delta^l))}\int_{B(\cdot,\widetilde{C}\delta^l)}
|{\rm R}_4(\cdot,y)|^u\,d\mu(y)\right]^{\frac{q}
{u}}\right\}^{\frac{1}{q}}\right\|_{L^p(\cx)}
\lesssim\|f\|_{\ihf}.$$
This finishes the proof of (iii) and hence of
Proposition \ref{ibtlinlip}.
\end{proof}

\begin{remark}
We point out that Proposition \ref{ibtlinlip}(iv) is new even when $\cx$ is an RD-spaces.
\end{remark}

Combining Propositions \ref{pro-lipi}(i),
\ref{ilipinbtl}, and \ref{ibtlinlip},
we obtain the following difference
characterization of inhomogeneous Besov
and Triebel--Lizorkin spaces; we omit the details.

\begin{theorem}\label{idc}
Let $\eta$ be as in Definition \ref{10.23.2},
$\beta,\ \gamma\in(0,\eta)$, and
$s\in(0,\beta\wedge\gamma)$.
\begin{enumerate}
\item[{\rm (i)}] If $p\in[1,\infty]$ and
$q\in(0,\infty]$, then
$\ilips=\ilipsb=\ihb$
with equivalent (quasi-)norms.
\item[{\rm (ii)}] If $p\in(1,\infty)$ and
$q\in(1,\infty]$, then
$\ilipst=\ihf$
with equivalent norms.
\end{enumerate}
\end{theorem}

\bigskip

\noindent Fan Wang, Dachun Yang (Corresponding author) and Wen Yuan

\medskip

\noindent Laboratory of Mathematics and Complex Systems (Ministry of Education of China),
School of Mathematical Sciences, Beijing Normal University, Beijing 100875, People's Republic of China

\smallskip

\noindent{\it E-mails:} \texttt{fanwang@mail.bnu.edu.cn} (F. Wang)

\noindent\phantom{{\it E-mails:} }\texttt{dcyang@bnu.edu.cn} (D. Yang)

\noindent\phantom{{\it E-mails:} }\texttt{wenyuan@bnu.edu.cn} (W. Yuan)

\bigskip

\noindent Ziyi He

\medskip

\noindent School of Science, Beijing University of Posts and Telecommunications, Beijing 100876,
People's Republic of China

\medskip

\noindent{\it E-mail:} \texttt{ziyihe@bupt.edu.cn}

\end{document}